\documentclass[a4paper,12pt,oneside,reqno]{amsart}
\usepackage[latin1]{inputenc}
\usepackage{hyperref}
\usepackage[headinclude,DIV13]{typearea}
\areaset{16cm}{24.7cm}
\usepackage{amsfonts,amssymb,amsmath,amsthm,bbm}
\usepackage{cancel} 

\usepackage{graphicx} 
\usepackage{caption}
\usepackage{subcaption}
\usepackage{mathrsfs}
  
\usepackage{xcolor}
\usepackage{enumerate}
\usepackage{dsfont}
\usepackage{nicefrac}	
\theoremstyle{plain}
\newtheorem{theorem}{Theorem}[section]
\newtheorem{lemma}[theorem]{Lemma}
\newtheorem{proposition}[theorem]{Proposition}
\newtheorem{corollary}[theorem]{Corollary}

\theoremstyle{definition}
\newtheorem{definition}[theorem]{Definition}
\newtheorem{assumption}{Assumption}

\newtheorem*{notations}{Notation}{\itshape}{\rmfamily}

\theoremstyle{remark}
\newtheorem{remark}{Remark}

\def\paragraph#1{\noindent \textbf{#1}}

\numberwithin{equation}{section}

\def\d{\mathrm{d}}

\def\<{\langle}
\def\>{\rangle}

\def\a{\alpha}

\def\e{\e}

\def\k{\kappa}
\def\l{\lambda}

\def\s{\sigma}

\def\th{\theta}

\def\L{\Lambda}

\def\P{{\Bbb P}}

\let\cal=\mathcal

\def\mfm{{\mathfrak M}}

 \def \L {{\Lambda}}
 \def \k {{\kappa}}
 
 \def \e {{\epsilon}}
 \def \s {{\sigma}}

 \def \n {{\nu}}

 \def \th {{\theta}}

 \def \l {{\lambda}}
 \def \d {{\delta}}
 \def \a {{\alpha}}

 \def \th {{\theta}}
 \def \n {{\nu}}

 \newcommand{\be}{\begin{equation}}
 \newcommand{\ee}{\end{equation}}

\newcommand{\bea}{\begin{eqnarray}}
 \newcommand{\eea}{\end{eqnarray}}
 
\def\TH(#1){\label{#1}}\def\thv(#1){\ref{#1}}
\def\Eq(#1){\label{#1}}\def\eqv(#1){(\ref{#1})}

 \def \1{\mathbbm{1}}

\def \lb {\left(}
\def \rb {\right)}

\begin{document}
%\pagenumbering{arabic}
 \title[]{From stochastic, individual-based models to the canonical equation of adaptive dynamics -- in one step}
  
\author[M. Baar]{Martina Baar}
\address{M. Baar\\ Institut f\"ur Angewandte Mathematik\\
Rheinische Friedrich-Wilhelms-Universit\"at\\ Endenicher Allee 60\\ 53115 Bonn, Germany}
\email{mbaar@uni-bonn.de}
\author[A. Bovier]{Anton Bovier}
\address{A. Bovier\\Institut f\"ur Angewandte Mathematik\\
Rheinische Friedrich-Wilhelms-Universit\"at\\ Endenicher Allee 60\\ 53115 Bonn, Germany}
\email{bovier@uni-bonn.de}
\author[N. Champagnat]{Nicolas Champagnat}
\address{N. Champagnat\\Institut Elie Cartan de Lorraine\\ UMR CNRS 7502\\
Universit\'e de Lorraine\\ Site de Nancy B.P. 70239\\ 54506 Vand\oe uvre-l\`es-Nancy Cedex, France \\and TOSCA team \\Inria Nancy - Grand Est, France.}
\email{Nicolas.Champagnat@inria.fr}


\thanks{M. B. is supported by the German Research Foundation through 
the Priority Programme  1590 ``Probabilistic Structures in Evolution''. A.B. is partially supported by the German Research Foundation in 
the Collaborative Research Center 1060 "The Mathematics of Emergent Effects", 
the Priority Programme  1590 ``Probabilistic Structures in Evolution'',
the Hausdorff Center for Mathematics (HCM), and  the Cluster of Excellence ``ImmunoSensation'' at Bonn University.
}

%%%%%%%%%%%%%%%%%%%%%%%%%%%%%%%%%%%%%%%%%%%%%%%%%%%%%%%%%%%%%%%%%%%%%%%%
%%Abstract
%%%%%%%%%%%%%%%%%%%%%%%%%%%%%%%%%%%%%%%%%%%%%%%%%%%%%%%%%%%%%%%%%%%%%%%%
\begin{abstract}
We consider a model for Darwinian evolution in an asexual population with a large but non-constant populations size characterized by a natural birth rate, a logistic death rate modelling competition and a probability of mutation at each birth event. In the present paper, 
we study the long-term behavior of the system in the limit of  large population ($K\to \infty$) size,  
 rare mutations ($u\to 0$),  and small mutational effects ($\s\to 0$), proving convergence to the 
  canonical equation of adaptive dynamics (CEAD).
 In contrast to earlier works, e.g. by Champagnat and M\'el\'eard, we take the three limits simultaneously,  i.e. $u=u_K$ and $\s=\s_K$, tend to zero with 
 $K$, subject to conditions that ensure that the time-scale of 
   birth and death events remains separated from that of successful mutational events. This
  slows down the dynamics of the microscopic system and leads to serious technical difficulties
  that requires the use of completely different methods. In particular,
  we cannot use the law of large numbers on the diverging time needed for fixation to approximate the stochastic system with the corresponding deterministic one. 
   To solve this problem we develop a "stochastic Euler scheme" based on coupling arguments that
  allows to control the time evolution of the stochastic system over time-scales that diverge with $K$.   
\end{abstract}
\maketitle

%%%%%%%%%%%%%%%%%%%%%%%%%%%%%%%%%%%%%%%%%%%%%%%%%%%%%%%%%%%%%%%%%%%%%%%%%%%%%%%%%%%%%%%%%%%%%
%%Introduction
%%%%%%%%%%%%%%%%%%%%%%%%%%%%%%%%%%%%%%%%%%%%%%%%%%%%%%%%%%%%%%%%%%%%%%%%%%%%%%%%%%%%%%%%%%%%%

\section{Introduction}
In this paper we study a microscopic model for evolution in a population characterized by a birth rate with a probability of mutation at each event and a logistic death rate, 
which has been studied in many works before \cite{C_TSS, C_CEAD, C_ME, C_PES, F_MA}. 
More precisely, it is a model for an asexual population in which each individual's ability to survive and to reproduce is a function of a one-dimensional phenotypic trait,
 such as body size, the age at maturity, or the rate of food intake. The evolution acts on the trait distribution and is the consequence of three basic mechanisms: heredity, mutation and selection.
Heredity passes the traits trough generations, mutation drives the variation of the trait values in the population,
and selection acts on individuals with different traits and is a consequence of competition between the individuals for limited resources or area.

The model is a generic  stochastic individual-based model and belongs to the models of adaptive dynamics. In general, adaptive 
dynamic models aim to study the interplay between ecology (viewed as driving selection) and evolution, more precisely, the interplay between the three basic 
mechanisms mentioned above. It tries to develop general tools to study the long time evolution of a wide variety of ecological 
scenarios \cite{D_DTC, D_ADPS, M_AD}. These tools are based on the assumption of separation of ecological and evolutionary time 
scales and on the notion of invasion fitness \cite{M_F, M_HSDF}.
While the biological theory of adaptive dynamics is based on partly heuristic derivations, various aspects of the theory 
have been derived rigorously over the last years in the context of stochastic, individual-based models~\cite{C_TSS,C_CEAD,C_ME,C_PES,G_CSS,H_AD}. All of them 
concern the limit when the population size, $K$, tends to infinity.
They either study the separation of ecological and evolutionary time scales based on a limit of rare mutations, $u\rightarrow 0$, 
combined with a limit of large population \cite{C_TSS,C_PES}, the limit of small mutation effects, $\sigma\rightarrow 0$, 
\cite{C_CEAD,C_PES,G_CSS}, the stationary behavior of the system \cite{H_AD}, or the links between individual-based and 
infinite-population models \cite{C_ME}.
A important concept in the theory of adaptive dynamics is the canonical
equation of adaptive dynamics (CEAD), introduced by U. Dieckmann and R. Law
\cite{D_DTC}. It is an ODE that describes the evolution in 
time of the expected trait value in a monomorphic population. 
The heuristics leading to the CEAD are based on the biological assumptions of large population and rare mutations with small effects 
and the assumption that no two different traits can coexist. (Note that we write sometimes mutation steps instead of effects.)
There are  mathematically rigorous papers that  show that the limit of large 
population combined with rare mutations leads to a jump process, the Trait Substitution Sequence, \cite{C_TSS}, 
and that this jump
process converges, in the  limit of small mutation steps,  to the CEAD, \cite{C_PES}.
Since these two limits are applied separately and on different time scales, they give no clue about how the
  biological parameters (population size $K$, probability of mutations $u$ and size of mutation steps $\s$) should compare to ensure that the CEAD approximation of the
  individual-based model is correct.
 
The purpose of the present paper is to analyse the situation when 
the limits of  large population size, $K\rightarrow\infty$, rare mutations, 
$u_K \rightarrow 0$, 
and small mutation steps, $\sigma_K\rightarrow 0$, are 
taken \emph{simultaneously}.
We consider populations with monomorphic initial condition, meaning that at time zero the population consists only 
of individuals with the same trait. 
Then we identify a time-scale where evolution can be described as a succession of 
mutant invasions. To prove convergence  to the CEAD, we show that if a mutation occurs, the individuals holding this 
mutant trait can either die out or invade the resident population on this time scale, where 
invasion means that the 
mutant trait supersedes the resident trait i.e.  the individuals with the resident trait become extinct after some time. This implies that the 
population stays essentially monomorphic with  a trait that evolves in time. 
We will impose conditions on the mutation rates  that imply a separation of ecological 
and evolutionary time scales in the sense that an invading  mutant population converges to its ecological equilibrium before a new 
invading 
(successful) mutant appears. In order to avoid too restrictive hypothesis on the mutation rates, we do, however, allow
non-invading (unsuccessful) mutation events during this time, in contrast to all earlier works.

We will see that the combination of the three limits simultaneously, entails some considerable technical 
difficulties. The fact that the mutants have only a $K$-dependent small evolutionary advantage  
decelerates the dynamics of the microscopic process such that the time of any 
macroscopic change between resident and mutant diverges with $K$. 
This makes it impossible to use a law of large numbers as in  \cite{C_TSS} 
 to approximate the stochastic system with the corresponding deterministic system during the time of invasion.
Showing  that the stochastic system
 still follows in an appropriate sense the  
 corresponding competition Lotka-Volterra system (with $K$-dependent coefficients)
 requires a completely  new approach. Developing this approach, which can be seen 
 as a rigorous "stochastic Euler-scheme", is the main novelty of the present paper. The proof requires methods, based on 
 couplings 
 with discrete time Markov chains combined some standard potential theory arguments for the "exit from a domain problem" 
in a moderate deviations regime, as well 
 as comparison and convergence results of branching processes.   
 Note that since the result of \cite{C_TSS} is already different from 
 classical time scales separations results (cf. \cite{F_RPoDS}), our result differs from them a fortiori. 
 Thus, our result can be seen as a rigorous justification of the biologically motivated, heuristic assumptions which lead to 
 CEAD.\\

The remainder of this paper is organised as follows.
In Section \ref{model}  and \ref{known_results} we introduce the model and give an overview on  previous related results. 
 In Section \ref{result} we state our results and give a detailed outline of the proof. Full details of the proof
 are presented in the Section \ref{First_Phase}, \ref{2.Phase} and \ref{Convergence}. In the appendix we state and prove several elementary
  facts that are  used throughout  the proof.

%%%%%%%%%%%%%%%%%%%%%%%%%%%%%%%%%%%%%%%%%%%%%%%%%%%%%%%%%%%%%%%%%%%%%%%%
%%Model and Historical Context
%%%%%%%%%%%%%%%%%%%%%%%%%%%%%%%%%%%%%%%%%%%%%%%%%%%%%%%%%%%%%%%%%%%%%%%%

\section{The individual-based model} \label{model}

In this section we introduce the model we analyze. 
We consider a population of a single asexual species that is composed of a finite number of individuals, each of them characterized by a one-dimensional phenotypic trait.
The microscopic model is an individual-based model with non-linear density-dependence, which has already been studied in ecological or evolutionary contexts by many authors \cite{C_ME,C_TSS,C_PES,F_MA}.  \\

The \textit{trait space} $\mathcal X$ is assumed to be a compact interval of $\mathbb R$. 
We introduce the following biological parameters:
\begin{enumerate}[(i)]
\setlength{\itemsep}{6pt}
\item{$b(x)\in\mathbb R_+$ is the \textit{rate of birth} of an individual with trait $x\in\mathcal X$.}
\item{$d(x)\in\mathbb R_+$ is the \textit{rate of natural death} of an individual with trait $x\in\mathcal X$.}
\item{$K\in\mathbb N$ is a parameter which scales the population size.}
\item{$c(x,y)K^{-1}\in\mathbb R_+$ is the \textit{competition kernel}
		which models the competition pressure felt by an individual with trait $x\in\mathcal X$ from an individual with trait $y\in\mathcal X$.}
\item{$u_K m(x)$ with $u_K, m(x)\in [0,1]$ is the \textit{probability that a mutation occurs at birth} from an individual with trait $x\in\mathcal X$, where $u_K\in [0,1]$ is a scaling parameter.}
\item{$M(x,dh)$ is the \textit{mutation law} of the mutational jump $h$. If the mutant is born from an individual with trait $x$, then the mutant trait  is given by $x+\sigma_K h\in \mathcal X$, where $\sigma_K \in[0,1]$ is a parameter scaling the size of mutation and $h$ is a random variable with law $M(x,dh)$.  We restrict for simplicity the setting to mutation measures with support included in $\mathbb Z$. }	
\end{enumerate}

The three scaling  parameters of the model are 
 the \emph{population size}, controlled by the scaling parameter $K$, 
the \emph{mutation probability}, controlled by the scaling parameter $u_K$, 
the \emph{ mutation size}, controlled by the scaling parameter $\sigma_K$.
The novelty of our approach is that we consider the case where all these parameters tend to their limit jointly, more
precisely that both $u_K$ and $\sigma_K$ are functions of $K$ and tend to zero  as $K$ tends to infinity
 (subject to certain constraints).

At any time $t$  we consider a finite number, $N_t$, of individuals, each of them having  a trait value $x_i(t)\in \mathcal X$. 
It is convenient to represent the population state at time $t$ by the rescaled point measure, 
$\nu^K $, which depends on $K$, $u_K$ and $\sigma_K$
\be
 	\nu^K _t=\frac 1 K \sum_{i=1}^{N_t}\delta_{x_i(t)}.
 \ee 
Let $\langle \mu , f \rangle$ denote the integral of a measurable function $f$ with respect to the measure $\mu $. 
Then $\langle\nu^K _t,\mathds 1 \rangle=N_t K^{-1}$ and for any $x\in\mathcal X$, 
the positive number $\langle\nu^K _t,\mathds 1_{\{x\}}\rangle$ is called the \textit{density of trait $x$ at time $t$}. 
With this notation, an individual with trait $x$ in the population $\nu^K _t$ dies due to age or competition with rate
\be
 d(x)+\int_{\mathcal X}\ c(x,y)\nu^K _t(dy).
\ee
Let $\mathcal M(\mathcal X)$ denote the set of finite nonnegative measures on $\mathcal X$, equipped with the weak topology, 
and define 
\be
	\mathcal M^K(\mathcal X)\equiv\left\{\frac 1 K \sum_{i=1}^{n}\delta_{x_i}\,:\, n\geq 0,\; x_1,...,x_n\in \mathcal X\right\}.
 \ee
Similar as in \cite{F_MA}, we obtain that the population process, $(\nu^K_t)_{t\geq0}$, is a  
$\mathcal M^K(\mathcal X)$-valued Markov process with infinitesimal generator, $\mathscr L^K$, 
defined for any bounded measurable function $f$ from $\mathcal M^K(\mathcal X)$ to $\mathbb R$ and for all $\mu^K\in \mathcal M^K(\mathcal X)$ by
\begin{align} 
\mathscr L^K f(\mu^K )= & \int_{\mathcal X}\biggl(f\Bigl(\mu^K +\frac {\delta_x}{K}\Bigr)-f(\mu^K )\biggr)\bigl(1-u_K m(x)\bigr)b(x)\:K\mu^K (dx)\\ 							\nonumber
						&+\int_{\mathcal X}\int_{\mathbb Z}\biggl(f\Bigl(\mu^K +\frac{\delta_{x+\sigma_K h}} K\Bigr)-f(\mu^K )\biggr)u_K 			
							m(x)b(x)\:M(x,dh)\: K\mu^K (dx)\\\nonumber
						&+\int_{\mathcal X}\biggl(f\Bigl(\mu^K -\frac {\delta_x} K\Bigr)-f(\mu^K )\biggr)\Bigl(d(x)+
						\int_{\mathcal X}c(x,y)\mu^K (dy)\Bigr)\:K\mu^K (dx).
\end{align}
The first and second terms are linear (in $\mu^K$) and describe the births (without and with mutation), 
but the third term is non-linear and describes the deaths due to age or competition. 
The density-dependent non-linearity of the third term models the competition in the population, 
and hence drives the selection process.
\begin{assumption}\label{ass}We will use the following assumptions on the parameters of the model:
\begin{enumerate}
\renewcommand{\labelenumi}{(\roman{enumi})}
\item{$b$, $d$ and $c$ are measurable functions, and there exist $\overline b,\;\overline d,\overline c<\infty$ such that}
\begin{center}$\;b(.)\leq \overline b,\quad d(.)\leq \overline d\quad$ and $\quad c(.\:,.)\leq \overline c.$\end{center}
\item{For all $x\in\mathcal X$, $b(x)-d(x)>0$, and there exists $\underline c>0$ such that $x\in\mathcal X$, $\underline c\leq c(x,x)$.}
\item{The support of $M(x,\:.\:)$ is a subset of $\mathbb{Z}\cap\mathcal X-x$ and uniformly bounded for all $x\in \mathcal X$. This means that there exists an $A\in\mathbb N$ such that}
 \begin{center}$ M(x,dh)=\sum_{k=-A}^Ap_k(x)\delta_k(dh)$,\quad where $\sum_{k=-A}^Ap_k(x)=1$ for any $x\in\mathcal X$.\end{center}

\item{$b,d,m\in C^2(\mathcal X, \mathbb R)$ and $c\in C^2(\mathcal X^2,\mathbb R)\;$.}

\end{enumerate} 
\end{assumption}
Assumptions (i) and (iii) allow to deduce the existence and uniqueness in law of a process on $\mathbb D(\mathbb R_+,$ $\mathcal M^K(\mathcal X))$ with infinitesimal generator $ \mathscr L^K$ (cf. \cite{F_MA}). 
Note that Assumption (iii) differs from the assumptions in \cite{F_MA} because we restrict the setting to mutation measures with support included in $\mathbb Z$ and that it ensures that a mutant trait remains in $\mathcal X$.
Assumption (ii) prevents the population from exploding or becoming extinct too fast. Since $\mathcal X$ is compact,
Assumption (iv) ensures that the derivatives of the functions $b, c, d$ and $m$ are uniformly Lipschitz-continuous.
%%%%%%%%%%%%%%%%%%
%Historical_Context
%%%%%%%%%%%%%%%%%
\section {Some notation and previous results}\label{known_results}
We start with a theorem, due to N. Fournier and S. M\'el\'eard, which describes the behavior of the populations process, for fixed $u$ and $\sigma$, when $K\rightarrow\infty$.
\begin{theorem}[Theorem 5.3 in \cite{F_MA}]
Fix $u$ and $\sigma$.
Let Assumption \ref{ass} hold and assume in addition that 
the initial conditions $\nu_0^K$ converge for $K\to\infty$ in law and for the weak topology on 
$\mathcal M(\mathcal X)$ to some deterministic finite measure $\xi_0\in \mathcal M(\mathcal X)$ 
and that $\sup_K \mathbb E[ \langle \nu^{K}_0, \mathds 1\rangle^3] <\infty$. \\[0.2em]
Then for all $T > 0$, the sequence $\nu^K$, generated by $\mathscr L^K$, converges  for $K\rightarrow \infty$ in law, 
in  $\mathbb D([0,T ],$ $\mathcal M ( \mathcal X))$, to a deterministic continuous function 
$\xi\in C([0,T ],\mathcal M ( \mathcal X))$.
This measure-valued function $\xi$ is the unique solution, 
satisfying $\sup_{t\in[0,T ]}\langle \xi_t, \mathds 1\rangle <\infty $, 
of the integro-differential equation written in its weak form:
for all bounded and measurable functions, $f:\mathcal X \to\mathbb R$,
\begin{align}
\int_{\mathcal X} \xi_t(dx)f(x)
	=&\int_{\mathcal X} \xi_0(dx)f(x) 
		+\!\int_0^t ds \int_{\mathcal X }\xi_s(dx) u m(x)b(x) \int_{\mathbb Z} M(x,dh) f(x\!+\!\sigma h)
						  \\
	&+\!\int_0^t  ds \int_{\mathcal X }\xi_s(dx) f(x)\Big(\left(1\!-\!u m(x)\right)b(x)
							\!-\!d(x)\!-\!\int_{\mathcal X}\xi_s(dy)c(x,y) \Big).\nonumber
\end{align}
\end{theorem}
Without mutation one obtains a convergence to the competitive system of Lotka-Volterra equations defined below (see \cite{F_MA}).
\begin{corollary}[The special case $u=0$ and $\xi_0$ is n-morphic]\label{cor}
If the same assumptions as in the theorem above with $u=0$ hold and if in addition $\xi_0=\sum_{i=1}^{n} z_i(0)\delta_{x_i}$, then $\xi_t$ is given by  
$\xi_t=\sum_{i=1}^{n}z_i(t)\delta_{x_i}$, where $z_i$ is the solution of the competitive system of Lotka-Volterra equations defined below.
\end{corollary}
\begin{definition}For any $(x_1,...,x_n)\in\mathcal X^n$, 
we denote by $LV(n,(x_1,...,x_n))$ the \emph{competitive system of Lotka-Volterra equations}  defined by\begin{align}\label{LV-System}
\frac {d\:z_i(t)}{dt}= z_i\biggl(b(x_i)-d(x_i)-\sum_{j=1}^n c(x_i,x_j)z_j\biggr), \qquad 1\leq i\leq n.
\end{align}
\end{definition}
Next, we introduce the notation of coexisting traits and of invasion fitness (see \cite{C_PES}).
\begin{definition}
We say that the distinct traits $x$ and $y$ \emph{coexist} if the system $LV(2,(x,y))$ admits an unique non-trivial equilibrium, named $\overline { z} (x,y) \in (0,\infty)^2$, which is locally strictly stable in the sense that the eigenvalues of the Jacobian matrix of the system $LV(2, (x,y))$ at $ \overline { z }(x,y)$ are all strictly negative. 
\end{definition}
The invasion of a single mutant trait in a monomorphic population which is close to its equilibrium is governed by its initial growth rate. Therefore, it is convenient to define the fitness of a mutant trait by its initial growth rate.
\begin{definition} If the resident population has the trait  $x\in\mathcal X$, then we call the following function \emph{invasion fitness} of the mutant trait $y$
\begin{align}
f(y, x)=b(y)-d(y)-c(y,x)\overline z(x).
\end{align}
\end{definition}
\begin{remark}
The unique strictly stable equilibrium of $LV(1,x)$ is $\ \overline z(x)=\frac{b(x)-d(x)}{c(x,x)}\,$, and hence $f(x,x)=0$ for all $x\in\mathcal{X}$.
\end{remark}
There is a relation between coexistence and invasion fitness (cf. \cite{I_MMLS}).
\begin{proposition}\label{pro_coex}
There is coexistence in the system $LV(2,(x,y))$ if and only if 
\begin{align} f(x,y)\equiv r(x)-c(x,y)\overline z(y)>0 \quad \text{and}\quad f(y,x)\equiv r(y)-c(y,x)\overline z(x)>0.\end{align}
\end{proposition}
The following convergence result from  \cite{C_TSS} describes the limit behavior of the populations process, 
for fixed $\sigma$, when $K\rightarrow\infty$ and $u_K\rightarrow 0$.
More precisely, it says that  the rescaled individual-based process converges in the sense of finite dimensional distributions 
to the "trait substitution sequence" (TSS),
if one assumes  in addition to Assumption 1 the following "Invasion implies fixation" condition. 
\begin{assumption}\label{ass2}
	Given any $x\in \mathcal X$, Lebesgue almost any $y\in\mathcal X$ satisfies one of the following conditions:\qquad
						(i)\;\; $f(y,x)<0$ \qquad or\qquad (ii)\;\; $f(y,x)>0\;$ and $\;f(x,y)<0$.
\end{assumption}
Note that by Proposition \ref{pro_coex}, this means that either a mutant cannot invade, or cannot coexist with the resident.
\begin{theorem}[Corollary 1 in \cite{C_TSS}]
\label{TSS}
Let Assumption \ref{ass} and \ref{ass2} hold. Fix $\sigma$ and assume that 
\begin{align}\label{Con_TSS}
\forall V>0, \qquad \exp(-VK)\ll u_K \ll \frac{1}{K\ln(K)}, \qquad \text{as } K\rightarrow \infty.
\end{align}
Fix also $x\in \mathcal X$ and let $(N^{K}_0)_{K\geq1}$ be a sequence of $\mathbb N$-valued random variables such that $({N^{K}_0}/{K})$ converges for $K\to \infty$ in law to $\bar z(x)$ and is bounded in ${\mathbb L}^{{}^{p}}$ for some $p>1$. Consider the processes $\nu^{K}$ generated by $\mathscr L^K$ with monomorphic initial state $(N^{K}_0 /K)\delta_{\{x\}}$.
\\Then the sequence of the rescaled processes $\nu^{{}^{K}}_{ t/ Ku_K}$ converges in the sense of finite dimensional distributions to the measure-valued process
\begin{align}
\overline z(X_t)\delta_{X_t},
\end{align}
where the $\mathcal X$-valued Markov jump process $X$ has initial state $X_0=x$ and infinitesimal generator
\begin{align}
A\phi(x)=\int_{\mathbb Z}\left(\phi(x+\sigma h)-\phi(x)\right)m(x)b(x)\overline z(x)\frac{[f(x+\sigma h,x)]_+}{b(x+\sigma h)}M(x,dh).
\end{align}
\end{theorem}

Here we write $f(K)\ll g(K)$ if $f(K)/g(K)\rightarrow 0$ when $K\rightarrow \infty$. Note that,
for any $s<t$, the convergence does not hold  in law for the Skorokhod topology on $\mathbb D([s,t],\mathcal M(\mathcal X))$, for any topology $\mathcal M(\mathcal X)$ such that the total mass function $\nu\mapsto \langle\nu,\mathds 1\rangle$ is continuous, because the total mass of the limit process is a discontinuous function.
The main part of the proof of this theorem is the study of the invasion of a mutant trait $y$ 
that has just appeared in a monomorphic population with trait $x$.
The invasion can be divided into three steps.
Firstly, as long as the mutant population size $\langle\nu^K_t,\mathds 1_{\{y\}}\rangle$
is smaller than a fixed small $\epsilon>0$, the resident population size
$\langle\nu^K_t,\mathds 1_{\{x\}}\rangle$ stays close to $\overline z(x)$. Therefore, $\langle\nu^K_t,\mathds 1_{\{y\}}\rangle$ can be approximated by a linear branching process with birth rate $b(y)$ and death rate $d(y)+c(y,x)\bar z(x)$ 
until it goes extinct or reaches $\e$.
Secondly, once $\langle\nu^K_t,\mathds1_{\{y\}}\rangle$  
has reached $\epsilon$, for large $K$, $\nu^K_t$ is close to the solution of $LV(2,(x,y))$
with initial state $(\overline z(x),\epsilon)$, which reaches the $\epsilon$-neighborhood of $(0, \overline z(y))$ in finite time.  This is a consequence of Corollary \ref{cor}.
Finally, once $\langle\nu^K_t,\mathds 1_{\{y\}}\rangle$ is close to $\overline z(y)$ and 
$\langle\nu^K_t,\mathds 1_{\{x\}}\rangle$ is small, $\langle\nu^K_t,\mathds 1_{\{x\}}\rangle$ can be approximated by a subcritical process, which becomes extinct a.s.\! . 
The time of the first and third step are proportional to $\ln(K)$, whereas the time of the second step is bounded.
Thus, the second inequality in (\ref{Con_TSS}) guarantees that, with high probability, 
the three steps of invasion are completed before a new mutation occurs.
\\[0.5em]   
Without Assumption \ref{ass2} it is  possible to construct the "polymorphic evolution sequence" (PES)
under additional assumptions on the $n$-morphic logistic system. This is done in \cite{C_PES}. 
Finally, in \cite{C_PES}, the convergence of the TSS with small mutation steps scaled by $\sigma$
to the "canonical equation of adaptive dynamics" (CEAD) is proved.
We indicate the dependence of the TSS of the previous Theorem on $\sigma$ with the notation $(X_t^{\sigma})_{t\geq 0}$.
\begin{theorem}[Remark 4.2 in \cite{C_PES}]
\label{1.2.2thm} 
If Assumption \ref{ass} is satisfied and the family of initial states of the rescaled TSS, $X^{\sigma}_0$, is bounded in $\mathbb L^2$ 
and converges to a random variable $X_0$ as $\sigma \rightarrow 0$,
then, for each $T>0$, the rescaled TSS $X^{\sigma}_{t/\sigma^{2}}$ converges when $\sigma\rightarrow \infty$, 
in the Skorohod topology on $\mathbb D([0,T],\mathcal X)$,
 to the process $(x_t)_{t\leq T}$ with initial state $X_0$ and with deterministic sample paths, 
 unique solution of the ordinary differential equation, known as CEAD:
\begin{equation}\label{(CEAD)}
\frac {d x_t}{dt}=\int_{ \mathbb Z} h\:[h\:m(x_t)\:\overline z(x_t)\:\partial_1 f(x_t,x_t)]_+ M(x_t,dh) ,
\end{equation} 
where $\partial_1f$ denotes the partial derivative of 
 $f(x,y)$ with respect to the first variable $x$.
\end{theorem} 
\begin{remark}
If $M(x,\cdot)$ is a symmetric measure on $\mathbb Z$ for all $x\in \mathcal X$, then the equation 
(\ref{(CEAD)}) has the classical form, c.f. \cite{D_DTC},
\begin{align}
\frac {d\:x_t}{dt}=\frac 1 2 \int_{ \mathbb Z} h^2\:m(x_t)\:\overline z(x_t)\:\partial_1 f(x_t,x_t) M(x_t,dh) ,
\end{align}
\end{remark}
Note that this result does not imply that, applying to the individual-based model first the limits $(K,u_K)\rightarrow
(\infty,0)$ and afterwards the limit $\sigma\rightarrow 0$ yields its convergence to the CEAD. 
One problem of theses two successive limits is, for example, that
the first convergence holds on a finite time interval, the second requires to look at the Trait Substitution Sequence on a time
interval which diverges. Moreover, as already mentioned these two limits give no clue about how 
 $K$, $u$ and $\s$ should be compared to ensure that the CEAD approximation is correct.

%%%%%%%%%%%%%%%%%%%%%%%%%%%%%%%%%%%%%%%%%%%%%%%%%%%%%%%%%%%%%%%%%%%%%%%%%%%%%%%%%%%%%%%%%%%%%
%%Main Result
%%%%%%%%%%%%%%%%%%%%%%%%%%%%%%%%%%%%%%%%%%%%%%%%%%%%%%%%%%%%%%%%%%%%%%%%%%%%%%%%%%%%%%%%%%%%%

\section{The main result} \label{result}
In this section, we present the main result of this paper, namely the convergence to the canonical equation of adaptive dynamics in one step. 
The time scale on which we control the population process
 is $t/(\sigma_K^2u_K K)$ 
and corresponds to the combination of the two time scales of Theorem \ref{TSS} and \ref{1.2.2thm}.
Since we combine the  limits we  have to modify the assumptions to obtain the convergence.	
We use in this section the notations and definitions introduced in Section \ref{known_results}. 	
\begin{assumption}\label{ass3}
 For all $x\in\mathcal X$, \quad$\partial_1 f(x,x)\neq 0$.
\end{assumption}
Assumption \ref{ass3} implies that either  \;$\forall x\in\mathcal X$: \;$\partial_1 f(x,x)>0$\; or \; $\forall x\in\mathcal X$:\;$\partial_1 f(x,x)<0$.\: Therefore coexistence of two traits is not possible. Without loss of generality we can assume that, $\forall x\in\mathcal X$, \:$\partial_1 f(x,x)>0$. In fact, a weaker assumption is sufficient, see Remark \ref{remark_main_thm}.(iii).  
\begin{theorem} \label{main_thm}
Assume that Assumptions \ref{ass} and \ref{ass3} hold and that there exists a small $\alpha>0$ such that
\begin{align}
\label{conv1}  		K^{- \nicefrac  1 2 +\alpha}&\ll \sigma_K\ll 1 \qquad \qquad \text{ and }  \\
\label{conv2}	\qquad\qquad	\exp(-K^{\alpha})&\ll u_K\ll \frac{\sigma_K^{1+\alpha}}{K\ln K},	\quad \text{ as } \qquad K\rightarrow \infty.
\end{align} 
Fix $x_0\in\mathcal X$ and let $(N^{K}_0)_{K\geq 0}$ be a sequence of $\mathbb N$-valued random variables such that
$N^{K}_0 K^{-1}$ converges in law, as $K\to\infty$,  to the positive constant $\overline z(x_0)$ and
is bounded in $\mathbb L{}^p$, for some $p > 1$. \\[0.5em]
For each $K\geq 0$,
let $\nu^{K}_{t}$ be the  process generated by $\mathscr L^K$ with monomorphic initial state ${N^{K}_0} K^{-1} \delta_{\{x_0\}}$.
Then, for all $T>0$, the sequence  of rescaled processes, $\big(\nu^{K}_{t / (Ku_K \sigma_K{}^2)}\big)_{0\leq t\leq T}$,
converges in probability, as $K\rightarrow \infty$,  with respect to the Skorokhod topology 
on  $\mathbb D([0,T],\mathcal M(\mathcal X))$ to the measure-valued process $\overline z(x_t)\delta_{x_t}$, 
where $(x_t)_{0\leq t\leq T}$ is given as a solution of the canonical equation of adaptive dynamics, 
\begin{equation}\label{CEAD}
\frac {d x_t}{dt}=\int_{ \mathbb Z} h\:[h\:m(x_t)\:\overline z(x_t)\:\partial_1 f(x_t,x_t)]_+ M(x_t,dh),
\end{equation}
with initial condition $x_0$.  
\end{theorem}
\begin{remark}\label{remark_main_thm}
\begin{enumerate}[(i)]
\setlength{\itemsep}{3pt}
\item If $x_t\in\partial \mathcal X$ for $t>0$, then (\ref{CEAD}) is $\frac {d\:x_t}{dt}=0$, i.e.  the process stops.
\item We can prove convergence for a stronger topology. 
	Namely, let us equip $\mathcal M_S(\mathcal X)$, the vector space of signed finite Borel-measures on $\mathcal X$, 
	with the following Kantorovich-Rubinstein norm:
\begin{equation}
	\Vert \mu_t \Vert^{}_0\equiv\sup\big\{\int_{\mathcal X} f d\mu_t: f\in \text{Lip}_1(\mathcal X) \text{ with } \sup_{x\in\mathcal X}|f(x)|\leq 1\big\},
\end{equation}		
	where $\text{Lip}_1(\mathcal X)$  is the space of Lipschitz continuous functions from $\mathcal X$ to $\mathbb R$ 
	with Lipschitz norm one (cf. \cite{B_MT} p. 191). 		
	Then, for all $\d>0$, we will prove that
\begin{equation}\label{conv_in_proba_first}
	\lim_{K\rightarrow \infty} \mathbb P\left [\:\sup_{0\leq t\leq T} \Vert \nu^{K}_{t/(K u_K \sigma_K{}^2)}-\overline z(x_t)\delta_{x^{}_t} \Vert^{}_0>
			\d\:\right] = 0.
	 \end{equation}
	By Proposition \ref{prop0} this implies convergence in probability with respect to the Skorokhod topology.
\item The main result of the paper actually holds under weaker assumptions. More precisely, Assumption \ref{ass3} can be replaced by\\[0.5em]
	\textsc{ Assumption 3'.}  
	The initial state $\nu^K_0$ has a.s.\ (deterministic) support $\{x_0\}$ with $x_0\in\mathcal{X}$ satisfying $\partial_1 f(x_0,x_0)\not=0.$
	\\[0.5em]
	The reason is that since $x\mapsto \partial_1 f(x,x)$ is continuous, the  Assumption 3 (a) is satisfied locally and 
	since $x\mapsto \partial_1 f(x,x)$ is Lipschitz-continuous, the CEAD never reaches an evolutionary singularity 
	(i.e. a value $y\in\mathcal  X$ such that $\partial_1 f(y,y)=0$) in finite time. 
	In particular, for a fixed $T>0$, the CEAD only visits traits in some interval $I$ of ${\cal X}$ where $\partial_1 f(x,x)\not=0$. 
	By modifying the parameters of the model out of $I$ in such a way that $\partial_1 f(x,x)\not=0$ everywhere in ${\cal X}$, we
	can apply Thm.~\ref{main_thm} to this modified process $\tilde{\nu}$ and deduce that $\tilde{\nu}_{t/Ku_K\sigma_K^2}$ has
   	support included in $I$ for $t\in[0,T]$ with high probability, and hence coincides $\nu_{t/Ku_K\sigma_K^2}$ on this time interval.	
\item The condition $ u_K\ll \frac{\sigma_K^{1+\alpha}}{K \ln K}$ allows mutation events during an invasion phase of a mutant trait,
	see below, but ensures that there is no "successful" mutational event during this phase.
\item The fluctuations of the resident population are of order $K^{- \nicefrac  1 2}$, 
	therefore $K^{- \nicefrac  1 2 +\alpha}\ll \sigma_K$ ensures that the sign
	of the initial growth rate is not influenced by the fluctuations of the population size.
	We will see later that if a mutant trait $y$ appears in a monomorphic population with trait $x$, 
	its initial growth rate is $b(y)-d(y)-c(y,x)\langle \nu_t^K,\mathds 1\rangle=f(y,x)+o(\s_K)=(y-x)\partial_1 f(x,x)+o(\s_K)$ since $y-x=O(\s_K)$.
\item  $\exp(K^{\alpha})$ is the time the resident population stays with high probability in a $O(\e\s_K)$-neighborhood of an attractive domain. 
	This can be seen as a moderate derivation result.
	Thus the condition $\exp(-K^{\alpha})\ll u_K$ 
	ensures that the resident population is still in this neighborhood when a mutant occurs.
\item The time scale is  $(Ku_K \sigma_K{}^2)^{-1}$ since the expected time for a mutation event is $(K u_K)^{-1}$, 
	the probability that a mutant invades
	is of order $\s_K$ and one needs $O(\sigma_K^{-1})$ mutant invasions to see a $O(1)$ change of the resident trait value. 
	This is consistent with the combination of of Theorem \ref{TSS} and \ref{1.2.2thm}. 
\item Note that the $\e$ that we use in the proof of the theorem and in the main idea below will not depend on $K$, 
	 but it will converge to zero in the end of the proof of Theorem \ref{main_thm}. 
	 The constant $M$ introduced below is going to be fixed all the time. It depends only the parameters of the model, 
	 in particular not on $K$ and $\e$.
\item The conditions about family of initial states imply that 
	$\sup_{K\geq 1}\sup_{t\geq 0}\mathbb E[ \langle \nu_t^K,\mathds 1\rangle^p]<\infty$ and therefore, since $p>1$, the family of random variable 
	$\{\langle \nu_t^K,\mathds 1\rangle\}_{K\geq 1,t\geq0}$ is uniformly integrable (cf. \cite{C_TSS} Lem. 1).    	 
\end{enumerate}
\end{remark}

\subsection{The main idea and the structure of the proof of Theorem \ref{main_thm}}\label{The main idea}

Under the conditions of the theorem the evolution of the population will be described as 
a succession of mutant invasions.
We prove that, on the timescale of this result, coexistence of two traits cannot occur, 
namely, when a mutant trait invades the population, the resident trait (i.e the trait that gave birth to the mutant trait) dies out.
We say the mutant trait fixates in the population. Note that this does not prevent coexistence with other mutant traits that do not invade.
In order to analyze the invasion of a mutant we divide the time until a mutant trait has fixated in the population into two phases.\\[0.5em]
\emph{The two invasion phases:} (compare with Figure \ref{fig})
First, we prove that, as long as all mutant densities are smaller than $\epsilon\sigma_K$, 
the resident density stays in an $M\epsilon\sigma_K$-neighborhood of  $\overline z(x)$. Note that, 
because the mutations are rare and the population size is large, 
the monomorphic initial population has  time to stabilize in an $M\epsilon\sigma_K$-neighborhood of this equilibrium $\overline z(x)$
before the first mutation occurs.
(The time of stabilization is of order $\ln(K)\sigma_K^{-1}$ and the time where the first mutant occurs is of order $1/Ku_K$).
This allows us to approximate the density of one mutant trait $y_1$ by a branching process with birth 
rate $b(y_1)$ and death rate $d(y_1)-c(y_1,x)\overline z(x)$ such that we can 
compute the probability that the density of the mutant trait $y_1$ reaches $\epsilon 
\sigma_K$, which is of order $\sigma_K$, as well as the time it takes to reach this level or 
to die out. Therefore, the process  needs $O(\sigma_K^{-1})$ mutation events 
until there appears a mutant subpopulation which reaches a size
$\epsilon \sigma_K$. Such a mutant is called \emph{successful mutant} and its trait will 
be the next resident trait. (In fact, we can calculate the distribution of the successful mutant trait only on an event with probability
$1-\e$, but we show that on an event of probability $1-o(\sigma_K)$, this distribution has support in $\{x+\s_Kh:h\in\{1,\ldots,
A\}\}$. Therefore, the exact value of the mutant trait is unknown with probability $\e$, 
but the difference of the possible values is only of order 
$\s_K$.) We prove in this step also that  there are never too many 
 different mutants alive at the same time. From all this  we deduce that the  
 subpopulation of the successful mutant reaches the density $\epsilon \sigma_K$, before a 
 different successful mutant appears.
Note that we cannot use large deviation results our time scale as used in \cite{C_PES} to prove this step. Instead, 
we use some standard potential theory and coupling arguments to obtain estimates of moderate deviations needed to prove 
that a successful mutant will appear before the resident density exists a $M \e\s_K$-neighborhood of its equilibrium.
  
Second, we want to prove that if a mutant population with 
 trait $y_s$  reaches the size 
$\epsilon \sigma_K$, it will increase to an  $M\epsilon\sigma_K$-neighborhood of its equilibrium density $\overline z(y_s)$. 
Simultaneously, the density of the resident trait decreases to $\epsilon \sigma_K$ and finally dies out. %
Since the fitness advantage of the mutant trait is only of order $\sigma_K$, the dynamics 
of the population process and the corresponding deterministic system 
are very slow. (Even if we would start at a  macroscopic density $\epsilon$,  the 
deterministic system needs a time of order $\sigma_K^{-1}$ to reach an 
$\epsilon$-neighborhood of its  
equilibrium density).  Thus we can not apply the law of large 
numbers for density-dependent population processes 
(see Chap. 11 of \cite{E_MP}) on our time scales which was used in \cite{C_TSS} and 
\cite{C_PES}  to approximate the population process by the solution of the corresponding
 competition Lotka-Volterra system. This is the main difficulty, which requires entirely new techniques. 
The method we develop to handle this situation  can be seen as a rigorous 
 stochastic "Euler-Scheme".  
Nevertheless, the proof contains an idea which is strongly connected with the properties of the deterministic 
dynamical system. Namely, the deterministic system of equations for the case $\s_K=0$ has 
an invariant manifold of fix points  with a vector field independent of $\s_K$ pointing towards this manifold. Turning on 
a small 
$\s_K$, we therefore expect the stochastic system to stay close to this invariant manyfold and to move 
along it with a speed of order $\s_K$. 
With this method we are able prove that, in fact, the mutant density reaches the $M\e\s_K$-neighborhood of $\overline z(y_s)$ and
 the resident trait dies out. Note that it is possible that a unsuccessful mutant is alive a this time. 
 Therefore, we prove that after the resident trait has died out, there is a time where the population consists only of one trait,
 namely the one that had fixed, before the next successful mutant occurs.
 We will divide this phase into several steps. 
A more detailed outline of  the structure  of the proof is given in  Section \ref{2.Phase}.

\begin{figure}[ht]
\centering\includegraphics[scale=0.6]{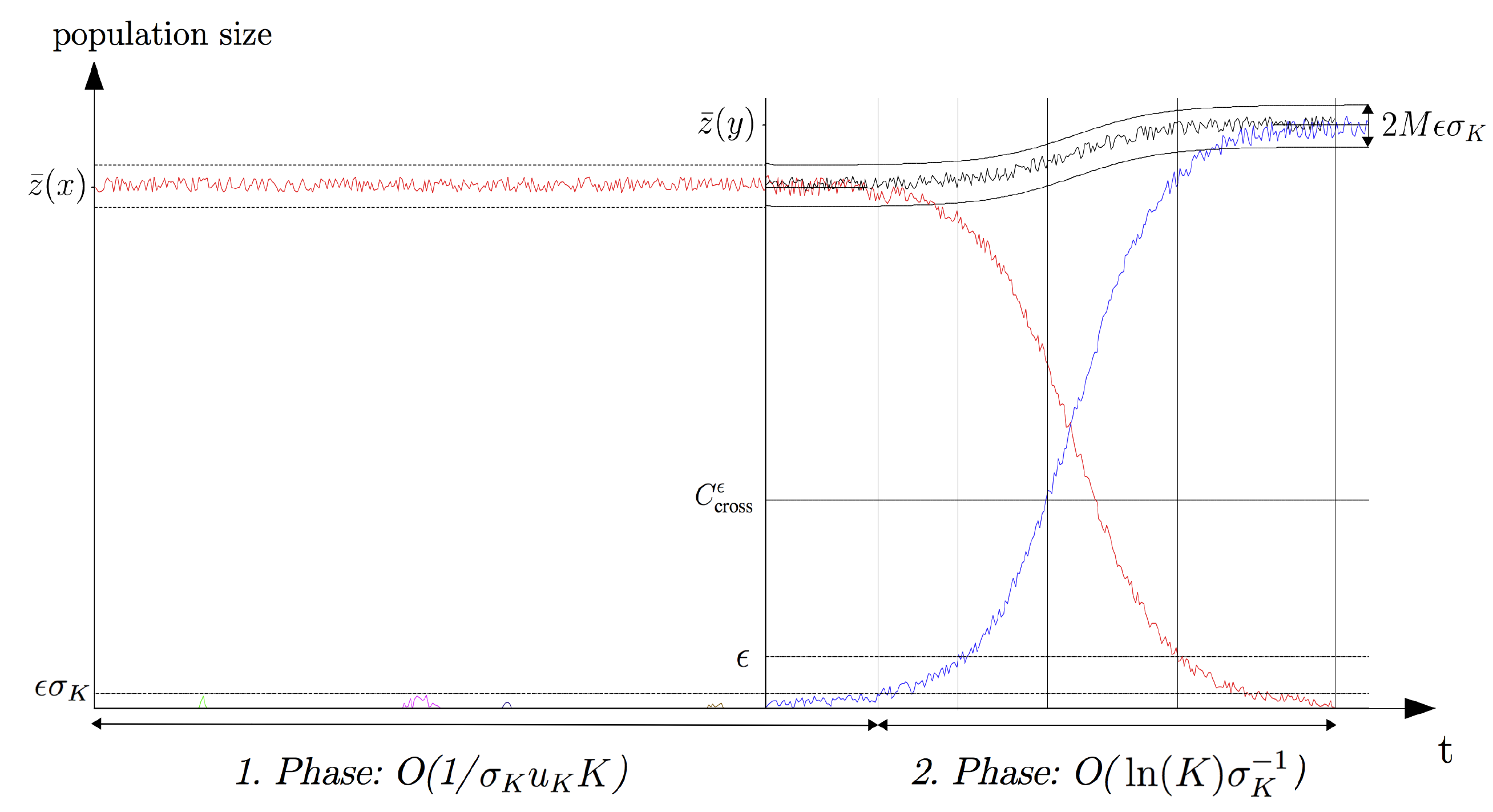}
  \caption{\label{fig}Typical evolution of the population during a mutant invasion. }
\end{figure}
  Note that Figure \ref{fig} is only a rough draft and not a "real" simulation. 
\begin{notations} 
\begin{enumerate}[(a)]
\itemsep6pt
\item
If $X$ and $Y$ are two random variables on $\mathbb R$,  we write $X \preccurlyeq Y$ if we can construct a random variable, $\tilde
Y$ on the probability space as $X$,
 such that $\mathcal L(Y)= \mathcal L(\tilde Y)$, 
 and that for all $\omega\in\Omega: \quad X(\omega)\leq \tilde Y(\omega)$.  
\item If $\mu$ and
  $\nu$ are two measures in $\mathcal M(\mathcal X)$, then we write $\nu \preccurlyeq \mu$ if:\\[0.2em]
\begin{tabular}{ll}
(i)\qquad& $\langle \nu,\mathds 1\rangle \leq \langle \mu,\mathds 1\rangle \quad$ and\\
(ii)\qquad& $\sup\left\{x\in \mathcal X: x\in\text{Supp} \left(\nu\right)\right\}\leq \inf\left\{x\in \mathcal X:
  x\in\text{Supp} \left(\mu\right)\right\}$\end{tabular}
\\[0.2em]
Note that (i) and (ii) imply  that,   for all $f\in \text{Lip}_1(\mathcal X, [-1,1])$ that are  monotone increasing and  for all $0\leq t\leq T$, 
\be \int_{\mathcal X} f(x)d{\nu_t}\leq \int_{\mathcal X} f(x)d{\mu_t}.  
\ee
\end{enumerate}

\end{notations}

\emph{Convergence:}
Given $T>0$, with the results of the two invasion phases, 
 we will define for all $\e>0$  two measure-valued processes, $\mu^{1,K,\e}$ and $\mu^{2,K,\e}$, 
 in $\mathbb D([0,\infty),\mathcal M(\mathcal X))$,  such that, for all $\e>0$,
\be
\lim_{K\rightarrow \infty} \mathbb P\left [\forall\: t\leq \tfrac T{Ku_K\sigma_K^2}:\quad \mu_t^{1,K,\e} \preccurlyeq \nu_t^{K} \preccurlyeq \mu_t^{2,K,\e} \:\right] =1,
\ee
 and, for all $\e>0$ and $i\in\{1,2\}$,
\be
\lim_{K\rightarrow \infty}
		 \mathbb P\left [\:\sup_{0\leq t\leq T/(Ku_K\sigma_K^2)} \Big\Vert \: \mu^{i,K,\e}_{t/(K u_K \sigma_K{}^2)}-
		 \overline z(x_t)\delta_{x^{}_t}\:\Big\Vert^{}_0>\delta(\e) \:\right] = 0,
\ee
for some function $\delta$ such that $\delta(\e)\rightarrow 0$ when $\e\rightarrow 0$.
This implies (\ref{conv_in_proba_first}) and therefore the theorem.

More precisely, let $\theta_{i}^{K}$ be the random time of the $i$th invasion phase, 
i.e. the first time after $\th^K_{i-1}$ such that a mutant density is larger than $\e\s_K$, 
and let $R_{i}^{K}$ be the trait of the $i$th successful mutant. Knowing the random variables $\th^K_{i-1}$ and $R_{i-1}^{K}$, we are able to approximate  $\th^K_{i}$ and $R_{i}^{K}$: After the ($i$-1)th invasion phase 
(of the process $\nu^K$), we define two random times, $\theta_{i}^{K,1}$ and 
$\theta_{i}^{K,2}$, and two random variables $R_{i}^{K,1}$ and $R_{i}^{K,2}$ in $\mathcal X$, 
 such that 
\be
 \lim_{K\rightarrow \infty}
 \mathbb P
 \left[ 
 \forall i \leq \sup\left\{j\!\in\! \mathbb N\!: \theta_{j}^{K} \leq \tfrac{T}{Ku_K\sigma_K^2} \right\}:
  R_{i}^{K,1}\preccurlyeq R_{i}^{K} \preccurlyeq R_{i}^{K,2}
 \text { and }\theta_{i}^{K,2}\preccurlyeq \theta_{i}^{K} \preccurlyeq \theta_{i}^{K,1} 
 \right]=1.
\ee
Thus we  define $\mu^{1,K}$ and $\mu^{2,K}$ through
\bea
\mu_t^{1,K}&\equiv&z^1_t\delta_{R_{i}^{K,1}}, \quad \text{ for } t\in[\theta_{i}^{K,1},\theta_{i+1}^{K,1} ), \\
\mu_t^{2,K}&\equiv&z^2_t\delta_{R_{i}^{K,2}} ,\quad \text{ for } t\in[\theta_{i}^{K,2},\theta_{i+1}^{K,2} ).
\eea 
for some appropriate masses $z^1_t$ and $z^2_t$. In fact, $z^1_t$ will be approximately $\bar z(R_{i}^{K,1})$ for $t\in[\theta_{i}^{K,1},\theta_{i+1}^{K,1} )$ and  $z^2_t$ approximately $\bar z(R_{i}^{K,2})$
for $t\in[\theta_{i}^{K,2},\theta_{i+1}^{K,2} )$.
We will prove that the times $\theta_{i}^{K,1}$ and $\theta_{i}^{K,2}$ are (approximately) exponentially 
distributed with parameters of order $Ku_K\sigma_K$, and that 
the difference of $R_{i}^{K}-R_{i-1}^{K}$ is of order $\sigma_K$. 
The processes $\mu^{1,K,\e}$ and $\mu^{2,K,\e}$ will be constructed by slightly modifying the two processes $\mu^{1,K}$ and
$\mu^{K,2}$ in order to make them Markovian.
This will imply 
by standard arguments from \cite{E_MP} that the processes 
$\mu_{t/Ku_K\sigma_K^2}^{1,K}$  and $\mu_{t/Ku_K\sigma_K^2}^{2,K}$ converge
 to $\overline z(x_t)\delta_{x_t}$ when $\sigma_K \rightarrow 0$, where $x_t$ is the solution of the canonical equation of adaptive dynamics. 
 
All the remaining sections are devoted to the proof of the Theorem \ref{main_thm}.

%%%%%%%%%%%%%%%%%%%%%%%%%%%%%%%%%%%%%%%%%%%%%%%%%%%%%%%%%%%%%%%%%%%%%%%%%%%%%%%%%%%%%%%%%%%%%
%%The augmented Process (Outline_of_Proof)
%%%%%%%%%%%%%%%%%%%%%%%%%%%%%%%%%%%%%%%%%%%%%%%%%%%%%%%%%%%%%%%%%%%%%%%%%%%%%%%%%%%%%%%%%%%%%

\section{An augmented process and some elementary properties}
In the proof of Theorem \ref{main_thm} we need to construct a augmented process
that keeps track of part of the history of the population.
In particular, we record the number of mutations that  occurred before $t$.

Let $\mathcal M_F^K(\mathbb N_0\times \mathcal X)\equiv \left\{  \frac 1 K \sum_{i=1}^n
\delta_{\xi(i)}\;:\; n\geq 0,\:\xi(1),\ldots,\xi(n) \in \mathbb N_0\times \mathcal X   \right\}$ denote 
the set of finite non-negative point measures on $\mathbb N_0\times \mathcal X$
 rescaled by $K$.
We write $\xi(i)=(\xi_1(i),\xi_2(i))$, where $\xi_1(i)\in\mathbb N_0$ and 
$\xi_2(i)\in\mathcal X$. 
The augmented process, $(\tilde \nu^K, L^K)$, is a continuous time stochastic process 
with state space  
$\mathcal M_F^K(\mathbb N_0\times \mathcal X)\times \mathbb N_0$. The label $k$ of an 
individual with trait $(k,x)$ denotes that there were $k-1$ mutational events  in the 
population before the trait $(k,x)$ appeared for the first time in the population. 
As in \cite{F_MA}, we give a path-wise description of  $(\tilde \nu^K, L^K)$.

\begin{notations}
Let $\mu^K=\frac 1 K \sum_{i=1}^n\delta_{\xi(i)}\in \mathcal M_F^K(\mathbb N_0\times 
\mathcal X)$ and  
let $\mfm^k(\mu^K)\equiv  K\int_{\mathbb N_0 \times \mathcal X}\mathds 1_{\{\xi_1=k\}} \mu^K(d\xi)$ be 
the number of individuals holding a mutation of label $k$. Then we 
rewrite $\mu^K$ as follows, 
\begin{align}
\mu^K= \frac 1 K \sum_{k=1}^\infty \sum_{j=1}^{\mfm^k(\mu^K)}\delta_{(k,x^k_j)},
 \quad 
\text{where } \sum_{k=1}^\infty \mfm^k (\mu^K)=n.
\end{align} 
In fact, the $x_1^k,...,x^k_{\mfm^k(\mu^K)}$ will be equal in our situation,
 because the only variation in the trait value is driven by mutational events. 
We need to  define three functions. 
First, $H:\mathcal M_F^K(\mathbb N_0\times \mathcal X)\mapsto(\mathbb N_0\times \mathcal X)^{\mathbb N^2_0}$ is defined as
\begin{equation}
H( \mu^K)\equiv 
\begin{pmatrix} 
(0,x^0_{1})& (0,x^0_{2})& \ldots & (0,x^0_{\mfm^ 0(\mu)})& (0,0) & (0,0) &\ldots\\
(1,x^1_{1})& (1,x^1_{2})& \ldots & \ldots & 
(1,x^1_{\mfm^ 1(\mu)}) & (1,0) &\ldots\\
(2,x^2_{1})& (2,x^2_{2})& \ldots & 
(2,x^2_{\mfm^ 2(\mu)})& (2,0) & (2,0) &\ldots\\
\vdots& \ldots& \ldots & \ldots & \ldots  & \ldots &\ddots\\
\end{pmatrix},
\end{equation}
Second, $h:\mathcal M_F^K(\mathbb N\times \mathcal X)\mapsto ( \mathcal X)^{\mathbb N_0^2}$ us given in terms of   $H$ by
 \begin{equation}
 h_{ij}( \mu^K)\equiv  \text{the second component of }  H_{ij}(\mu^K), 
 \end{equation}
i.e.,  if  $H_{ij}(\mu^K)=(i,x)$, then  $h_{ij}=x$.
Third, $\tilde H:\mathcal M_F^K(\mathbb N\times \mathcal X)\mapsto  \mathcal X^{\mathbb N_0}$ is defined as follows: if $\mu=\frac 1 K \sum_{i=1}^n\delta_{\xi(i)}$, then 
\begin{equation}
\tilde H( \mu) \equiv \big(\xi_2(\sigma(1)),\:\xi_2(\sigma(2)),\:\ldots,\: 
\xi_2(\sigma(n)),\:0,\:\ldots\big),
\end{equation}
where $\xi_2(\sigma(1))\leq \ldots\leq \xi_2(\sigma (n))$.
\end{notations}

\begin{definition}
Let $(\Omega,\mathcal F,\mathbb P)$ be a abstract probability space. 
On this space, we define the following independent random elements:
\begin{itemize}\itemsep 6pt
\item[(i)] a $\mathcal X$-valued random variable $X_0$ (the random initial trait), 
\item[(ii)] a sequence of independent Poisson point measures, 
	$(N^{\text{death}}_k(ds,di,d\theta))_{k\geq0}$, on 
$\mathbb R_+\times \mathbb N\times \mathbb R_+$ with intensity measure
 $ds\sum_{n\geq 0}\delta_n(di)dz$, 
\item[(iii)] a sequence of independent Poisson point measures, 
$(N^{\text{birth}}_k(ds,di,d\theta))_{k\geq0}$, on 
$\mathbb R_+\times \mathbb N\times \mathbb R_+$ with intensity measure
 $ds\sum_{n\geq 0}\delta_n(di)dz$,
\item[(vi)] a Poisson point measures, $N^{\text{mutation}}(ds,di,d\theta,dh)$, on 
$\mathbb R_+\times \mathbb N\times \mathbb R_+\times \{-A,\ldots, A\}$ 
with intensity measure
$ds\sum_{n\geq 0}\delta_n(di)dz\sum_{j\in \{-A,\ldots, A\}}\delta_j(dh)$.
\end{itemize}
Let  $L^K_0\equiv 0$  and $\tilde \nu^K_0\equiv \frac 1 K N^K_0\delta_{X_0}$,
then we consider the process defined by the following equation
\begin{align}
(\tilde \nu_t^K,L_t^K)
= &\: (\tilde \nu_0^K,L_0^K)\\\nonumber
 & +
\sum_{k\geq 0}
\bigg(	\int_0^t \int_{\mathbb N_0}\int_{\mathbb R_+}
\mathds 1_{\left\{i\leq \mfm^k (\tilde\nu^K_{s^-}),\; \theta \leq b\left(h_{k,i}(\tilde 
\nu^K_{s^-})\right)\lb1-u_K m\lb h_{k,i}(\tilde \nu^K_{s^-})\rb\rb\right\}}
\\\nonumber
 &\hspace{6,5cm}\times \lb \tfrac{1}{K}{\d_{H_{k,i}(\tilde \nu^K_{s^-})}},0 \rb 
 N^{\text{birth}}_k(ds,di,d\th)\\\nonumber
&\hspace{1cm} -\int_0^t \int_{\mathbb N_0}\int_{\mathbb R_+}
\mathds 1_{\left\{i\leq \mfm^k (\tilde\nu^K_{s^-}),\; 
\theta \leq d\left(h_{k,i}(\tilde \nu^K_{s^-})\right)
+\int_{\mathbb N_0\times \mathcal X}c\lb h_{k,i}(\tilde \nu^K_{s^-}),\xi_2 \rb \tilde \nu^K_{s^-}(d\xi)\right\}}
 \\\nonumber
&\hspace{6,5cm} \times \lb  \tfrac{1}{K}{\d_{H_{k,i}(\tilde 
\nu^K_{s^-})}} ,0\rb N^{\text{death}}_k(ds,di,d\th)	 \bigg)	\\	\nonumber
&+	\int_0^t \int_{\mathbb N_0}\int_{\mathbb R_+}\int_{\{-A,\ldots, A\}}
\mathds 1_{\left\{i\leq K \langle \tilde\nu^K_{s^-},\mathds 
1\rangle,\; \theta \leq b\left(\tilde H_{i}(\tilde \nu^K_{s^-})\right)
u_K m\lb \tilde H_{i}(\tilde \nu^K_{s^-})\rb M(\tilde H_{i}(\tilde \nu^K_{s^-}),h)\right\}}
\\\nonumber
&\hspace{3,5cm}\times \lb \tfrac{1}{K}{\d_{\left(L(s^-)+1,\: \tilde H_{k,i}(\tilde \nu^K_{s^-})+
\sigma h\right)}} ,1\rb N^{\text{mutation}}(ds,di,d\th, dh).
\end{align}
\end{definition}
Note that  the process $(\tilde \nu_t^K,L_t^K)_{t_\geq 0}$ is a Markov process with  
generator 
\begin{align} 
\tilde{ \mathscr L}^K& f((\tilde \nu, L)) \\\nonumber
=& \sum_{k\geq 0}\bigg(\int_{\mathcal X}\left(f\left(\tilde \nu+\tfrac {\delta_{(k,x)}}{K}, L
\right)-f(\tilde \nu , L)\right)\bigl(1-u_K m(x)\bigr)b(x)\:K\tilde\nu((k,dx))	\\ \nonumber
&\quad 
+\int_{\mathcal X}\left(f\left(\tilde \nu-\tfrac {\delta_{(k,x)}}{K}, L\right)-f(\tilde \nu , L)\right)
\Bigl(d(x)+ \int_{\mathbb N_0 \times \mathcal X}c(x,\xi_2)\tilde\nu(d\xi)\Bigr)\:K\tilde\nu((k,dx))\bigg)\\
\nonumber&
+\int_{\mathbb N_0\times \mathcal X}\int_{\mathbb Z}\left(f\left(\tilde\nu+
\tfrac{\delta_{(L+1,x+\sigma_K h)}} K , L+1\right)-f(\tilde\nu, L)\right)
u_K m(x)b(x)\:M(x,dh)\: K\tilde \nu(d(k,x)).
\end{align}
Naturally, the process generated by $\mathscr L^K$  defined in Section \ref{model}
is a projection of the process with generator $\tilde{ \mathscr L}^K$.

The first elementary property we give is that we there exists a rough upper bound for the total mass of the population.
\begin{lemma} \label{upper_bound_total_mass}
Under the same assumptions as in Theorem \ref{main_thm}, there exists a constant, $V>0$, such that
\begin{equation}
\lim_{K\to \infty}\;  \mathbb P \Big[ \inf\{t\geq 0: \langle\tilde\nu^K_t,\mathds 1\rangle
 \geq 4\overline b/\underline c \}<\exp(V K)\Big]=0.
\end{equation}
\end{lemma}
\begin{proof}
Apply Theorem 2 (a) and then Theorem 3 (c) of \cite{C_TSS}.
\end{proof}

%%%%%%%%%%%%%%%%%%%%%%%%%%%%%%%%%%%%%%%%%%%%%%%%%%%%%%%%%%%%%%%%%%%%%%%%%%%%%%%%%%%%%%%%%%%%%
%%First Phase
%%%%%%%%%%%%%%%%%%%%%%%%%%%%%%%%%%%%%%%%%%%%%%%%%%%%%%%%%%%%%%%%%%%%%%%%%%%%%%%%%%%%%%%%%%%%%

\section{The First Phase of an Invasion}\label{First_Phase}
In the first phase we show that we can approximate the first time when the density of a mutant trait reaches the value
$\epsilon\sigma_K$ and the trait value of this mutant trait both on an event with probability  $1-o(\s_K)$ . Such a trait will be
called successful.
\begin{assumption} \label{ini_con} 
Fix $\e>0$. Let $(R^K)_{K\geq0}$ be a sequence random variables with values in $\mathcal X$. Then, there exists a constant $\tilde M>0$ (independent of $\e$ and $K$) such that for all $K$ large enough
\begin{equation}  
			L^K_{0}=0  \quad \text{ and }\quad \tilde\nu^K_{0}=N_{R^K}^K K^{-1}\delta_{(0,R^K)} 
\end{equation}			 	
where  $N^K_{R^K}\in \mathbb N$ is a sequence of random variable with $\left|\overline z(R^K)-N^K_{R^K}K^{-1}\right|< \tilde M \e \sigma_K$ a.s.. We call $R^K$ the
resident trait.
\end{assumption} 
%%%%%%%%%%%%%%%%%%%%%%%%%%%%%%%%%%%%%%%%%%%%%%%%%%%%%%%%%%%%%%%%%%%%%%%%
Note that Assumption \ref{ini_con} is stronger than the initial condition the assume in Theorem \ref{main_thm}. However we obtain with high probability Assumption \ref{ini_con} after a small time.
%%%%%%%%%%%%%%%%%%%%%%%%%%%%%%%%%%%%%%%%%%%%%%%%%%%%%%%%%%%%%%%%%%%%%%
\begin{proposition} Fix $\e>0$.
Suppose that the assumptions of Theorem \ref{main_thm} hold.
Then, there exists a constant $\tilde M>0$ (independent from $\e$ and $K$), such that 
\be
\lim_{K \to \infty} \mathbb P\left[\inf\bigl\{t\geq 0: |\langle\tilde \nu^K_t,\mathds 1\rangle-\bar z(x)|<\tilde M \e \s_K\bigr\}< \ln(K)\s_K^{-1}\wedge
\inf\left\{t \geq 0: L(t)\geq 1\right\} \right]=1.
\ee
\end{proposition}
%%%%%%%%%%%%%%%%%%%%%%%%%%%%%%%%%%%%%%%%%%%%%%%%%%%%%%%%%%%%%%%%%%%%%%
Since we can assume for the moment that Assumption \ref{ini_con} hold, we do not state the proof here. In fact, it can be proven in similar way as Lemma \ref{Step2} (a). 
We begin with several notations, which we use in the lemmata below.

\begin{notations} 
Fix $\e>0$. Suppose that  Assumption \ref{ass}, \ref{ass3} and \ref{ini_con} hold.
Let $\tau^K_{k}$ be the $k$ th mutant time, and $Y^K_k\in\mathcal X$ the trait of the $k$-th mutant, i.e.
\begin{equation}
\tau^K_k\equiv \inf\{t\geq 0: L^K_t=k\} \quad \text{and}\quad Y^K_k \equiv h_{k,1}(\tilde \nu^K_{\tau^K_k}).
\end{equation} 
 We denote by $\th^K_{\text{invasion}}$ the first time 
such that a mutant density is larger than $\epsilon\sigma_K$, i.e.
\begin{equation}
\th^K_{\text{invasion}}\equiv \inf\left\{t\geq 0:\exists k\in\{1,\ldots, L^K_t\}  \text{ such that } \mfm^k (\tilde \nu^K_t) > \epsilon\sigma_K K\right\},
\end{equation} 
and let $R_1^K$ be the trait value of the mutant which is larger than $\epsilon\sigma_K K$ at time $\th^K_{\text{invasion}}$, i.e.
\begin{equation}
R^K_1\equiv h_{k_1,1}(\tilde \nu^K_{\th^K_{\text{invasion}}})\quad \text{ with }k_1=\inf\left\{k\geq 1: \mfm^k (\tilde \nu^K_{\th^K_{\text{invasion}}}) > \epsilon\sigma_K K\right\}.
\end{equation} 
Furthermore, let $\th^K_{ \text{diversity}}$ be the first time such that $\lceil 3/\alpha\rceil$ different traits are present in the population, i.e.
\begin{equation}
\th^K_{ \text{diversity}}\equiv \inf\left\{t\geq 0:\sum_{k=0}^{L^K(t)}\mathds
 1_{\{\mfm^k (\tilde \nu^K_t)\geq 1\}}= \lceil 3/\alpha\rceil \right\},
\end{equation}
and similarly let $\th^K_{ \text{mut. of mut.}}$ the first time such that a "$2$nd generation mutant" occurs, i.e. a mutant which was born from a
mutant born from the resident trait $R^K$. 
 Note that
\begin{equation}
\th^K_{ \text{mut. of mut.}}\leq \inf\left\{t\geq 0:\exists k\in\{1,\ldots, L^K_t\} \text{ such that }  |R^K-Y^K_k|>A\s_K\right\}.
\end{equation}
Then, we define
\begin{equation}
\hat \th^K\equiv \th^{K}_{ \text{invasion}}\wedge \th^{K}_{ \text{diversity}}\wedge \th^K_{ \text{mut. of mut.}} \wedge\exp({K^{\a}}).
 \end{equation}
\end{notations} 
%%%%%%%%%%%%%%%%%%%%%%%%%%%%%%%%%%%%%%%%%%%%%%%%%%%%%%%%%%%%%%%%%%%%%%%%
The following theorem collects the main results of this section.

\begin{theorem}\label{1.Phase}
Fix $\e>0$. Under the Assumptions \ref{ass}, \ref{ass3} and \ref{ini_con}, there exists a constant 
$M>0$  (independent of $\e$ and $K$) such that for all $K$ large enough
\begin{enumerate}[(i)]
		\setlength{\itemsep}{3pt}
\item $\tilde\nu^K_{0}=N_{R^K}^K K^{-1}\delta_{(0,R^K)}$, where  $\left|\overline z(R^K)-N^K_{R^K}K^{-1}\right|<(M/3) \e \sigma_K$ a.s..
\item  We can construct on $(\Omega,\mathcal{F},\mathds P)$ two random variables,
 $R_1^{K,1}$ and $R_1^{K,2}$,
		such that 
	\bea
\P\left[R_1^{K,1}\leq R_1^K\leq R_1^{K,2} \;\text{ and } \;R^{K,2}_1-R^{K,1}_1\leq A \s_K\right]&=&1-o(\sigma_K), \quad \text{and}
	\\
	\P\left[R^{K,1}_1=R^K_1=R^{K,2}_1\right] &=&1-O(\e).
	\eea
	The distributions of $R_1^{K,1}$ and $R_1^{K,2}$ are given in Corollary  \ref{new_resident_trait}.  
\item  We can construct on $(\Omega,\mathcal{F},\mathds P)$  two exponential random variables,
 $E^{K,1}$ and $E^{K,2}$, with parameters
of order $ \s_K u_K K $, 
		such that 
	\be
	\P\left[
	 E^{K,2}\leq \th^{K}_{\text{invasion}} \leq E^{K,1}+\ln(K)\s_K^{-1-\a/2 }\right]=1-o(\sigma_K).
	 \ee
	  The distributions of $E^{K,1}$ and $E^{K,2}$ are given in Lemma \ref{bounds_for_invasion_time}.  				
\end{enumerate}
Moreover, until the first time of invasion, $\theta^K_{\text{invasion}}$, the resident density stays
 in an $\e  M\sigma_K$-neighborhood of $\bar{z}(R^K)$,
the number of different living mutant traits is bounded by $\lceil \alpha/3\rceil$,
 and there is no mutant of a mutant, with probability $1-o(\sigma_K)$. I.e.
\begin{align} \nonumber
&\mathbb P \left[ \th^{K}_{ \text{invasion}}<\inf \Big\{ t\geq 0: \left|\mfm^ {0}(\tilde \nu^K_t) - \left\lceil K\overline z(R^K) 
\right\rceil \right| >\e M  \s_K K \Big\}\wedge \th^{K}_{ \text{diversity}}\wedge \th^K_{ \text{mut. of mut.}}\right]
\\&\quad
=1-o(\sigma_K).
\end{align}
\end{theorem}

\begin{remark} The constant $M>0$ depends only on 
$\alpha$ and on the  functions $ b(.), d(.), c(.,.)$ and $m(.)$, but not on $K$, $R^K$ and $\e$. 
\end{remark}
%%%%%%%%%%%%%%%%%%%%%%%%%%%%%%%%%%%%%%%%%%%%%%%%%%%%%%%%%%%%%%%%%%%%%%
The first lemma in this section concerns the time of exit from an attracting domain. 
\begin{lemma} \label{exit_from_domain}
Fix $\e>0$. Suppose that the assumptions of Theorem \ref{1.Phase} hold. 
Then, there exists a constant $M>0$ (independent of $\e$ and $K$) such that
\begin{align}
\lim_{K\to \infty}\;\s_K^{-1}\;
\mathbb P \Big[ \inf \Big\{ t\geq 0: \left|\mfm^ {0}(\tilde \nu^K_t) - \left\lceil K\overline z(R^K) 
\right\rceil\right |& >\e M  \s_K K \Big\}<\hat \th^{K} \Big]=0.
\end{align}
\end{lemma}
The statement is stronger than the corresponding one in \cite{C_TSS}, Thm. 3(c), since the diameter of the domain converges to zero,
when $K$ tends to infinity and since it contains a speed of convergence to 0 of the probabilities. Therefore, it follows not from the classical results about the time of exit from an attractive domain (cf. \cite{F_RPoDS}). 
Our proof is based on a coupling with a discrete Markov chain and some standard potential theoretical argument.
%%%%%%%%%%%%%%%%%%%%%%%%%%%%%%%%%%%%%%%%%%%%%%%%%%%%%%%%%%%%%%%%%%%%%%%%
\begin{proof} 
Define 
\begin{equation} 
X_t	\equiv \left|\mfm_t^0-\left\lceil K\overline z(R^K) \right\rceil \right|
\end{equation}
and, for all $M\geq0$,
\begin{equation} 
\tau^{}_0\equiv \inf\{t>0:X_t=0\}\quad\text{ and }
\quad \tau^{}_{M\epsilon \sigma_K K}\equiv \inf\{t>0:X_t\geq M \epsilon \sigma_K K\}.
\end{equation}					
Note that $\tau^{}_0$ and $\tau^{}_{M \epsilon \sigma_K K}$ are stopping times with respect to the natural filtration of $X_t$, 
which is equal to $\sigma\left (\mfm_s^0 ; s\leq t\right)$, and that the process $(\mfm_t^0)_{t\geq 0}$ is not markovian. 
We can associate with the continuous time process $X_t$ a discrete time (non-Markov) process $Y_n$ 
which records the sequence of values that $X_t$ takes.
(This can be formally defined by introducing the sequences $T_k$ of the stopping times which record the instances 
when $X_t\neq X_{t-}$ and setting $Y_n=X_{T_n}$.)
Now, we can compute 
 \begin{equation}
 \label{1.P_a}\mathbb P\bigl[\tau^{}_{M \epsilon \sigma_K K}<\tau^{}_0\wedge\th^{K}_{ \text{invasion}}\wedge \th^{K}_{ \text{diversity}}\wedge \th^{K}_{ \text{mut. of mut.}} \bigr]
 \end{equation}
 with respect to the stopping times defined for the discrete time process $Y_n$ and 
exploit the natural renewal structure on $Y_n$.
 Therefore, we prove the following claim.\\[0.5em]
\textbf{Claim:}
 \textit{For  $1\leq i\ll K$, and $K$ large enough,}
  \bea
 \label{1.cond_pro}\mathbb P\bigl[Y_{n+1}=i+1|Y_n=i, T_{n+1}<\th^{K}_{ \text{invasion}}\wedge \th^{K}_{ \text{diversity}}\wedge \th^{K}_{ \text{mut. of mut.}} \bigr] &&\\\leq \frac 1 2 -(\underline c/ 4\overline b )K^{-1}i+ (\overline c/ \underline b ) \epsilon \sigma_K &\equiv& p_+^K(i),\nonumber
\eea
\textit{where $\underline c,\underline b,\overline c$ and $\overline b$ are the lower, respectively upper bounds for birth and competition rates.}\\
Recall from Remark 1 that the equilibrium $\overline z(R^K)$ is equal to $\frac {b(R^K)-d(R^K)}{c(R^K,R^K)}$ and observe that there are at most 
$\lceil 3/\a\rceil \e\s_K K$ mutant individuals alive at any time $t<\th^{K}_{ \text{invasion}}\wedge \th^{K}_{ \text{diversity}}\wedge \th^{K}_{ \text{mut. of mut.}} $.
Therefore, for  $1\leq i\ll K$, and $K$ large enough, 
 \begin{align}
 \mathbb P\bigl[&Y_{n+1}=i+1|Y_n=i, T_{n+1}<\th^{K}_{ \text{invasion}}\wedge \th^{K}_{ \text{diversity}}\wedge \th^{K}_{ \text{mut. of mut.}} \bigr]\\\nonumber
  		&\leq \tfrac{(1-m(R^K) u_K  )b(R^K)}
 								 {(1-m(R^K) u_K  )b(R^K)+d(R^K)+{c(R^K,R^K)}{K^{-1}}(\lceil K\overline z(R^K)\rceil +i)}
			\\\nonumber&\hspace{5cm}
						\vee\tfrac{d(R^K)+{c(R^K,R^K)}{K^{-1}}(\lceil K\overline z(R^K)\rceil -i)+\bar c \lceil 3/\a \rceil\epsilon \sigma_K K}
										{(1-m(R^K) u_K  )b(R^K)+d(R^K)+{c(R^K,R^K)}{K^{-1}}(\lceil K\overline z(R^K)\rceil-i)}
		\\
\nonumber&\leq\tfrac{b(R^K)-m(R^K) u_K  b(R^K)}{2b(R^K)-m(R^K) u_K  b(R^K)+{c(R^K,R^K)}{K^{-1}}i}\vee
				\:\tfrac{b(R^K)-{c(R^K,R^K)}{K^{-1}}(i-1)+ \bar c\lceil 3/\a \rceil\epsilon \sigma_K K}{2b(R^K)-m(R^K) u_K  b(R^K)-{c(R^K,R^K)}{K^{-1}}i}
		\\&\leq \frac 1 2 -(\underline c/ 4\overline b )K^{-1}i+ (\overline c/ \underline b ) \lceil 3/\a \rceil \e\s_K. \nonumber
\end{align}
This proves the claim. 
Next we introduce a coupling, i.e. we define a discrete time process $Z_n$ with the following properties 
\begin{enumerate}[(i)]
\itemsep8pt
\item{$Z_0=Y_0$,}
\item{$\mathbb P\left[Z_{n+1}=i+1,Y_{n+1}=i+1\big|Y_n=Z_n=i, T_{n+1}<\th^{K}_{ \text{invasion}}\wedge \th^{K}_{ \text{diversity}}\wedge \th^{K}_{ \text{mut. of mut.}} \right]
				\\\hspace*{4cm}=\mathbb P\bigl[Y_{n+1}=i+1|Y_n=i, T_{n+1}<\th^{K}_{ \text{invasion}}\wedge \th^{K}_{ \text{diversity}}\wedge \th^{K}_{ \text{mut. of mut.}} \bigr]$,}
\item{$\mathbb P\bigl[Z_{n+1}=i+1,Y_{n+1}=i-1|Y_n=Z_n=i,T_{n+1}<\th^{K}_{ \text{invasion}}\wedge \th^{K}_{ \text{diversity}}\wedge \th^{K}_{ \text{mut. of mut.}} \bigr]$\\\hspace*{2.5cm}
		$=p_+^K(i)-\mathbb P\bigl[Y_{n+1}=i+1|Y_n=i, T_{n+1}<\th^{K}_{ \text{invasion}}\wedge \th^{K}_{ \text{diversity}}\wedge \th^{K}_{ \text{mut. of mut.}} \bigr]$},
\item{$\mathbb P\bigl[Z_{n+1}=i+1|Y_n<Z_n=i, T_{n+1}<\th^{K}_{ \text{invasion}}\wedge \th^{K}_{ \text{diversity}}\wedge \th^{K}_{ \text{mut. of mut.}} \bigr] 	=p_+^K(i)$,}
\item{$\mathbb P\bigl[Z_{n+1}=i-1|Y_n<Z_n=i, T_{n+1}<\th^{K}_{ \text{invasion}}\wedge \th^{K}_{ \text{diversity}}\wedge \th^{K}_{ \text{mut. of mut.}} \bigr] 		=1-p_+^K(i)$.}
\end{enumerate}
Note that by construction $Z_n\geq Y_n$ a.s. for all $n$ such that $T_{n}<\th^{K}_{ \text{invasion}}\wedge \th^{K}_{ \text{diversity}}\wedge \th^{K}_{ \text{mut. of mut.}} $ and the  marginal distribution of $Z_n$  is a Markov chain with transition probabilities 
\begin{equation}
\mathbb P\bigl[Z_{n+1} =j|Z_n=i\bigr]=\begin{cases}		
1 &\text{ for } i=0 \text{ and } j=1\\
p_+^K(i) &\text{ for }	i\geq1 	\text{ and }  j=i+1 \\
1-p_+^K(i)	& \text{ for }i\geq1 	\text{ and } j=i+1\\
0 	&\text{ else.}
\end{cases}
\end{equation}
Now we define a continuous time process, $\tilde Z$, associated to $Z_n$. To do this let $(\tilde T_j)_{j\in \mathbb N}$ be the sequence of jump times  of $\tilde Z$, i.e. $\tilde Z_t \equiv Z_n $  if $t\in [\tilde T_n,\tilde T_{n+1})$, defined for all $j\in \mathbb N$ as follows
\begin{align}  
\tilde T_j-\tilde T_{j-1} =
\begin{cases} 
T_j - T_{j-1} & \text{ if } T_j <\th^{K}_{ \text{invasion}}\wedge \th^{K}_{ \text{diversity}}\wedge \th^{K}_{ \text{mut. of mut.}} \\
W_j & \text{ else}, 
\end{cases}
\end{align}
where $W_j$ are independent exponential distributed random variables with mean $(C_{\text{total rate}}K)^{-1}$ where $C_{\text{total rate}}=4\overline b \underline c (\overline b+\overline d+\overline c(4\overline b \underline c))$. By Lemma \ref{upper_bound_total_mass}, $C_{\text{total rate}}K$ is an upper bound for the total event rate of $\langle \tilde \n^K_t,\mathds 1\rangle$ and therefore also for $\mfm^0_t$.

Define $\tau^Z_{M\e\s_K K}\equiv  \inf\{n\geq 0:  Z_n\geq M\e\s_K K\}$ and $\tau^Z_0\equiv \inf\{n \geq 0:Z_n =0\}$. Then,
since $\tilde Z_t \geq X_t$ a.s. for all $t<\th^{K}_{ \text{invasion}}\wedge \th^{K}_{ \text{diversity}}\wedge \th^{K}_{ \text{mut. of mut.}} $,
\begin{align}
\mathbb P\bigl[\tau^{}_{M\e\s_K K}<\tau^{}_0\wedge\th^{K}_{ \text{invasion}}\wedge \th^{K}_{ \text{diversity}}\wedge \th^{K}_{ \text{mut. of mut.}} \bigr]
\leq \mathbb P\big[\tau^Z_{M\e\s_K K}<\tau^Z_0\big].
\end{align}
Applying Proposition \ref{prop1} yields that, for all $M \geq 32 \lceil 3/\a \rceil (\overline c \:\overline b)/( \underline b\:  \underline c ) $ such that $Z_0 \leq \frac 1 3 M\e\s_K K$ and large $K$ large enough,
\begin{equation}
\mathbb P\big[\tau^Z_{M\e\s_K K}<\tau^Z_0\big]\leq \exp \lb - K^{2\a}\rb.
\end{equation}	
Next we prove that the process $X_t$ returns many times to zero before it reaches for the first time the value $M\e\s_K K$. More precisely, we obtain first a lower bound for the number of returns by the discrete time process $Z_n$. Then we calculate the time for a return to zero. From now on we assume that $M \geq 32  \lceil 3/\a \rceil(\overline c \:\overline b)/( \underline b\:  \underline c )$. 
Define stopping times with respect to the natural filtration of $Z$ which records the number of jumps the process $Z$ needs for $m$ zero-returns:
 \begin{equation}
\tau_{\text{$m$ returns}}^Z  \equiv \inf \left\{n\geq 1\;:\; \sum_{i=1}^{n}\mathds 1_{Z_{i}=0}=m\right\}.
\end{equation}
Let $Q^m\equiv \mathbb P\bigl[ \tau_{\text{$m$ returns}}^Z < \tau^Z_{M\e\s_K K}<\tau_{\text{$(m+1)$ returns}}^Z \bigr]$ 
be the probability that the Markov chain $Z_n$ returns exactly $m$ times to zero 
before it reaches the value $M\e\s_K K$.
We have
\begin{equation}
Q^0=\mathbb P\big[\tau^Z_{M\e\s_K K}<\tau^Z_0\big]\leq \exp  \lb - K^{2\a}\rb,
\end{equation} and, due to the Markov property, for $m\geq 1$ ,
\begin{align}
Q^m=\mathbb P& \left[\tau^Z_0<\tau^Z_{M\e\s_K K}\right] \left(1-\mathbb P^{}_1\left[\tau^Z_{M\e\s_K K}<\tau_{0}^Z \right]\right)^{m-1}
			\mathbb P^{}_1\left[\tau^Z_{M\e\s_K K}<\tau_{0}^Z\right], 
\end{align}
where the last term in the product is smaller than $\exp  \lb - K^{2\a}\rb$.
Thus,
\be 
Q^m\leq \exp \lb - K^{2\a}\rb \quad \text{ for all }m\geq 0. 
\ee
Let $B$ be the random variable which records the number of zero returns of $Z_n$ before $Z_n$ reaches $M\e\s_K K$.
With other words, $B=n$ if and only if $\tau_{\text{$n$ returns}}^Z < \tau^Z_{ M\e\s_K K}<\tau_{\text{$n+1$ returns}}^Z$ and we obtain that
\begin{align}\label{Q_a}
\mathbb P\bigl[B\leq n\bigr]
	= \sum_{i=0}^n Q^i\leq (n+1) \exp \lb - K^{2\a}\rb. 
	\end{align}
Set $I_1\equiv  \tilde T_{\tau_{\text{$1$ return}}^Z}$ and  $I_j\equiv  \tilde T_{\tau_{\text{$j$ returns}}^Z}-\tilde T_{\tau_{\text{$(j-1)$ returns}}^Z}$ for $j\geq 2$. 
For any $j$, $I_j$ is the random time between the ($j-1$)th and the $j$th zero return of the associated continuous time process $\tilde Z_t$  and 
\be
\label{1.interval}
 \sum_{i=1}^{B}I_{i}\leq \inf\{t\geq 0: \tilde Z_t\geq M\e\s_K K\}\leq \sum_{i=1}^{B+1}I_{i}.
\ee
We get an upper bound for the probability which we want to compute
\bea\label{prob_to_comp}
&&\mathbb P \Big[ \inf \Big\{ t\geq 0: |\mfm^{0}(\tilde \nu^K_t) - \lceil K\overline z(R^K) \rceil |>\e M  \s_K K \Big\} 
			<\hat\th^{K}\Big]\\\nonumber
	&&\quad\leq \sum_{l=n}^{\infty}\mathbb P\left[\inf\{t\geq 0: \tilde Z_t\geq M\e\s_K K\}<\exp(K^{\a})\:,\:B=l\right]
											+\mathbb P \bigl[B\leq n\bigr].\qquad
\eea
According to (\ref{1.interval}),
if $B=l$ and if in addition more than $l/2$ of the $l$ random times $I_j$ in the sum are larger than  $2l^{-1}  \exp( K^{\a})$,
then  $\inf\{t\geq 0: \tilde Z_t\geq M_x\e\s_K K\}$ is larger than $\exp( K^{\a})$. 
Therefore, for all $l\geq n$, 
\bea\TH(numbered.1)
&&	\mathbb P\left[\inf\{t\geq 0: \tilde Z_t\geq M\e\s_K K\}<\exp(K^{\a})\:,\:B=l\right]\\\nonumber
&&\quad		\leq \mathbb P\left[\sum_{i=1}^ l\mathds 1_{\{I_{j}<2l^{-1}  \exp(K^{\a})\}}>\nicefrac l 2 \;,\; B=l\right].
\eea
As mentioned before, $C_{\text{total rate}}K$ is a upper bound for the total event rate of $\langle \tilde \n^K_t,\mathds 1\rangle$. Thus we can bound the jump times  by a sequence of independent, exponential random variables $(V_j)_{j\in \mathbb N}$ with mean $(C_{\text{total rate}}K)^{-1}$. 
Namely,
\begin{equation}
		\tilde T_j-\tilde T_{j-1}\equiv  T_j-T_{j-1}\succcurlyeq V_j  \qquad \text{ if }\quad  T_j\leq \th^{K}_{ \text{invasion}}\wedge \th^{K}_{ \text{diversity}}\wedge \th^{K}_{ \text{mut. of mut.}}.
\end{equation}
Otherwise the random times $\tilde T_j-\tilde T_{j-1}$ are by definition independent and exponentially distributed with mean $(C_{\text{total rate}}K)^{-1}$.
The process $\tilde Z$ has to make at least two jumps to return to zero. Hence, \begin{align}\label{1.W-keit}
			I_i \succcurlyeq \tilde W_i, \quad  \text {for all $i\in \mathbb N$, } 
\end{align}
where $(\tilde W_i)_{i\in\mathbb N}$ are independent, exponential random variables with mean $(C_{\text{total rate}}K)^{-1}$. Thus
\begin{align}\label{2.W-keit}
	\mathbb P \left[ \sum_{i=1}^l \mathds 1_{\{I_{j}<2l^{-1}  \exp(K^{\a})\}}>\frac l 2 , B=l \right]
		\leq \mathbb P\left[\sum_{i=1}^ l\mathds 1_{\{\tilde W_{i}<2l^{-1}  \exp(K^{\a})\}}>\frac l 2 \right].		
\end{align}
Since $\mathbb P [\tilde W_{i}<2l^{-1}  \exp(K^{\a})]=1- \exp(-C_{\text{total rate}}K l^{-1} \exp(K^{\a}))$ and $(\tilde W_i)_{i\geq 1}$ are independent, we obtain that $\sum_{i=1}^ l\mathds 1_{\{\tilde W_{i}<2l^{-1}  \exp(K^{\a})\}}$ is binomially distributed with   $n=l$ and $p=1- $ $\exp(-C_{\text{total rate}}K$ $ l^{-1} \exp(K^{\a}))$. 
 Therefore, the right hand side of (\ref{2.W-keit}) is equal to
\begin{equation}
\sum_{i=\nicefrac l 2}^l \binom l i\left(1- \exp\left(-C_{\text{total rate}}K l^{-1} \exp(K^{\a})\right)\right)^i 
\left(\exp\left(-C_{\text{total rate}}K l^{-1} \exp(K^{\a}) \right)\right)^{l-i}	.		
\end{equation}
For the following two computations we use the elementary facts that  $\binom l i <2^l$  and $l<2^l $, for all $l\in \mathbb N$ and $i\leq l$. 
We obtain that, for large $K$ enough, the left hand side of (\ref{prob_to_comp}) is bounded from above by 
\begin{align}\nonumber
\sum_{l=n}^{\infty}\sum_{i=\nicefrac l 2}^l\binom l i &\left(1- \exp\left(-\tfrac{C_{\text{total rate}}K}{ l}\exp(K^{\a})\right)\right)^i 
\left(\exp\left(-\tfrac{C_{\text{total rate}}K}{ l } \exp(K^{\a}) \right)\right)^{l-i}+\mathbb P\left[B\leq n\right]
\\
&\leq\sum_{l=n}^{\infty}\frac l 2 \: 2^l\left(1- \exp\left(-\tfrac{C_{\text{total rate}}K}{ l } \exp(K^{\a})\right)\right)^{\nicefrac l 2}+\mathbb P\left[B\leq n\right].
\end{align}
By (\ref{Q_a}) we see that $\mathbb P\left[B\leq n\right]=o(\sigma_K)$ if the variable $n$ fulfills the following condition
\begin{equation} 
n\ll \exp\left(K^{2\a}\right)\s_K.
\end{equation} 
Therefore we  choose $n=\lceil\exp\left(2K^{\a}\right)\rceil$ and get, for large $K$ enough,
\bea\label{1.1.1.1}
&&\mathbb P \Big[ \inf \Big\{ t\geq 0: |\mfm^ {0}(\tilde \nu^K_t) - \lceil K\overline z(R^K) \rceil |>\e M  \s_K K \Big\} 
			<\hat \th^{K}\Big]\\\nonumber
			&&\qquad\leq\sum_{l=\lceil\exp\left(2K^{\a}\right)\rceil}^{\infty}\: 
			4^l\left(1- \exp\left(-C_{\text{total rate}}K l^{-1} \exp(K^{\a})\right)\right)^{\nicefrac l 2}+o(\s_K)\\\nonumber
			&&\qquad\leq\sum_{l=\lceil\exp\left(2K^{\a}\right)\rceil}^{\infty}
	\left (4\left(1-\exp\left(- C_{\text{total rate}} K \exp(-K^{\a})\right)\right)^{\nicefrac 1 2}\right)^l+o(\sigma_K)
\\\nonumber
&&\qquad\leq 2\left(4^2\left(1-\exp\left(- C_{\text{total rate}} K \exp(-K^{\a})\right)\right)\right)^{\frac 1 2\lceil\exp\left(2K^{\a}\right)\rceil}+o(\sigma_K)\\	\nonumber
&&\qquad\leq 2\left(4^2  C_{\text{total rate}} K \exp(-K^{\a})\right)^{\frac 1 2\lceil\exp\left(2K^{\a}\right)\rceil}+o(\sigma_K)\\
					&&\qquad\leq o(Ke^{-K^{\a}})+o(\sigma_K),\nonumber
\eea
where we used that $\exp(-x)\geq 1-x$ for $x\geq 0$ and $K\exp( K^{-\a})\ll\s_K$.
\end{proof}
%
%%%%%%%%%%%%%%%%%%%%%%%%%%%%%%%%%%%%%%%%%%%%%%%%%%%%%%%%%%%%%%%%%%%%%%%%%
In the following lemma we  bound $L^K_t$, the number of mutants up to time $t$,
 from above and below by Poisson counting processes.
\begin{lemma} \label{rv_A} 
Fix $\e>0$. Suppose that the assumptions of Theorem \ref{1.Phase} hold and let $M$ be the constant of Lemma \ref{exit_from_domain}. Then,
\begin{equation}\label{prob_rv_A}
		\lim_{K\to \infty}\;\s_K^{-1}\left(1-\mathbb P	\left [  \forall \:0\leq t\leq\hat \th^K\:: \:
																		\: A^{1, K}(t)\preccurlyeq L^K_{t}\preccurlyeq A^{2, K} (t)\:\right]\right)=0,
\end{equation} 
where $A^{1,K}$ and $A^{2,K}$ are Poisson counting processes with parameter $a^K_1u_K K$ and $a^K_2 u_K K$ with
\bea
 		a^K_1&\equiv&  \left(\overline z(R^K)-\e M  \s_K \right) b(R^K) m(R^K), \\
 		a^K_2 &\equiv & \left(\overline z(R^K)+\e \left(M +\lceil 3/ \a\rceil\right) \s_K \right) \left( b(R^K) m(R^K)+C^{b,m, M}_L A\sigma_K \right),
\eea
and $C^{b,m, M}_L$ is a constant depending only on the functions $b(.),m(.)$ and $M(.,h)$ for $h\in\{-A,\ldots,A\}$.
\end{lemma}
%
%%%%%%%%%%%%%%%%%%%%%%%%%%%%%%%%%%%%%%%%%%%%%%%%%%%%%%%%%%%%%%%%%%%%%%%%%
\begin{proof}We obtain from the last lemma that 
\begin{equation}
\mathbb P\left[\forall \:0\leq t\leq\hat \th^K:\: \overline z(R^K)-\e M  \s_K \leq \langle \tilde \nu_t,\mathds 1\rangle\leq\overline z(R^K)+\e \left(M +\lceil 3/\a\rceil\right) \s_K \right]=1-o(\s_K).
\end{equation}
Therefore, define
\begin{align}
A^{1, K}(t)=\int_{0}^{t}\int_{\mathbb N_0}\int_{\mathbb R_+}\int_{\{-A,\ldots, A\}}
					&\mathds 1_{\left\{i\leq K \left(\overline z(R^K)-\e M  \s_K \right),\; \theta \leq   b(R^K)
								u_K m(R^K)M(R^K,h)\right\}}\\\nonumber
					&\times N^{\text{mutation}}(ds,di,d\th, dh)	
\end{align}
and similarly
\begin{align}
A^{2, K}(t)=\int_{0}^{t}\int_{\mathbb N_0}\int_{\mathbb R_+}\int_{\{-A,\ldots, A\}}
&\mathds 1_{\left\{i\leq K \left(\overline z(R^K)+\e \left(M +\lceil\frac 3 \a\rceil\right) \s_K 
\right)\right\}}	\\&\nonumber
\times\mathds 1_{\left\{ \theta \leq u_K \left(b(R^K) m(R^K)M(R^K,h)+ C^{b,m, M}_L
  A\sigma_K \right)\right\}} 	\\&	\times N^{\text{mutation}}(ds,di,d\th, dh),\nonumber	
\end{align}  
Since $\hat \th^K\leq \th^K_{ \text{mut. of mut.}}$, any mutant trait  differs at most $A\s_K$ from the resident trait, $R^K$. Thus, we have that$u_K \big(b(R^K) m(R^K)M(R^K,h)+ C^{b,m, M}_L
 A\sigma_K \big)$ is a rough upper bound for the mutation rate per individual for an appropriate choice of $C^{b,m, M}_L$. Note that $A^{i, K}$ are Poisson counting process with parameter $a_i^K u_K K$. By construction, we obtain (\ref{prob_rv_A}).
\end{proof}
%%%%%%%%%%%%%%%%%%%%%%%%%%%%%%%%%%%%%%%%%%%%%%%%%%%%%%%%%%%%%%%%%%%%%%%%
Next we prove that  $\mfm^k(\tilde \nu_t)$, the number of offsprings of the $k$th mutant alive at time t, can be approximated by linear birth and death processes.
\begin{lemma}
\label{bound_mutants} 
Fix $\e>0$. Suppose that the assumptions of Theorem \ref{1.Phase} hold and let $M$ be the constant in Lemma \ref{exit_from_domain}. 
Then, 
\begin{equation}\label{stoch_dom}
\lim_{K\to \infty}\;\s_K^{-1}
\Big(1-\mathbb P\Big [   \forall \:1\leq k\leq L^K_{\hat \th^K},\: \forall  t\leq \hat \th^K\::Z^{K,1}_{k}(t)\preccurlyeq \mfm^k(\tilde \nu_t)\preccurlyeq Z^{K,2}_{k}(t) \:\Big]\:\Big)=0,
 \end{equation}
where $Z^{K,1}_{k}(t)$ resp.  $Z^{K,2}_{k}(t)$ are $\mathbb N_0$-valued processes, 
which are zero until time $\tau_k^K$, the first time  s.t. $\mfm^k(\tilde \nu_t)\neq 0$,
and afterwards linear, continuous time birth and death processes with initial state $1$ at time
 $\tau_k^K$ and birth rates per individual
\be
 b_k^{K,1}= b_k^{K,2}=b\big(Y^{K}_k\big)\lb1-u_K m(Y^{K}_k)\rb
 \ee
and death rate per individual 
\bea
d_k^{K,1}&=&d(Y^{K}_{k})+c( Y^{K}_k, R^K )
\lb \overline z(R^K)+ M \e\s_K\rb+ \overline c \lceil 3/ \a\rceil \e \s_K\\
\text{resp.}\qquad
d_k^{K,2}&=&d(Y^{K}_{k})+c( Y^{K}_k,R^K )
\lb \overline z(R^K)- M\e\s_K \rb.
\eea
Furthermore, define $\tilde Z^{K,1}_{k}(t)\equiv Z^{K,1}_{k}(\tau_k+t)$ and $\tilde Z^{K,2}_{k}(t)\equiv Z^{K,2}_{k}(\tau_k+t)$, then the processes $\{(\tilde Z^{K,1}_k,$ $ \tilde Z^{K,2}_k)\}_{k\geq 1}$ are independent and identically distributed. 
\end{lemma}
%
%%%%%%%%%%%%%%%%%%%%%%%%%%%%%%%%%%%%%%%%%%%%%%%%%%%%%%%%%%%%%%%%%%%%%%%%%%
%%
%
\begin{proof} For any $t\leq \hat\th^K$, any individual of $\mfm^k(\tilde \nu_t)$  gives birth to a new individual with the same trait with rate  $b\big(Y^{K}_k\big)\lb1-u_K m(Y^{K}_k)\rb$
and dies with rate $d(Y_k^K)+\int_{\mathbb N\times \mathcal X}c(Y^K_k ,\xi_2 )\tilde \nu^K_t(d\xi)$, which belongs to the following interval 
\begin{equation}
\Big[d(Y_k^K)+c(Y^K_k \!,R^K )(\bar z(R^K)\!-\!M \e\s_K) , d(Y_k^K)+c(Y^K_k \!,R^K )(\bar z(R^K)\!+\! M\e\s_K )+\bar c \lceil 3/ \a\rceil\e\s_K\Big].
\end{equation}
Thus,
let us define, for $k\leq L_{\hat\th^K}$,
\begin{align}
	\tilde Z^{K,1}_{k}(t)& \equiv 
		\int_{\tau_k}^{\tau_k+t}\!\!\int_{\mathbb N_0}\int_{\mathbb R_+}
			\!\mathds 1_{\left\{i\leq \tilde Z^{K,1}_{k}(s^-),\; \theta \leq b\left(Y_k\right)\lb1-u_K m\lb Y_k\rb\rb\right\}} N^{\text{birth}}_k(ds,di,d\th)\\
				&- \int_{\tau_k}^{\tau_k+t}\!\! \int_{\mathbb N_0}\int_{\mathbb R_+}
					\!\mathds 1_{\left\{i\leq \tilde Z^{K,1}_{k}(s^-), \theta \leq d(Y^{K}_{k})+c( Y^{K}_{k}, R^K )
								\lb \overline z(R^K)+
								 M\e\s_K\rb+ \bar c \lceil 3/ \a\rceil  \e \s_K\right\}}  N^{\text{death}}_k(ds,di,d\th)\nonumber
\end{align}
and similarly 
\begin{align}
	\tilde Z^{K,2}_{k}(t)\equiv &	\int_{\tau_k}^{\tau_k+t}\!\!\int_{\mathbb N_0}\int_{\mathbb R_+}
			\!\mathds 1_{\left\{i\leq \tilde Z^{K,1}_{k}(s^-),\; \theta \leq b\left(Y_k\right)\lb1-u_K m\lb Y_k\rb\rb\right\}} N^{\text{birth}}_k(ds,di,d\th)\\
				&- \int_{\tau_k}^{\tau_k+t}\!\! \int_{\mathbb N_0}\int_{\mathbb R_+}
					\!\mathds 1_{\left\{i\leq \tilde Z^{K,1}_{k}(s^-),\; \theta \leq d(Y^{K}_{k})+c( Y^{K}_{k}, R^K )
								\lb \overline z(R^K)-
								 M\e\s_K\rb\right\}}  N^{\text{death}}_k(ds,di,d\th),\nonumber
		\end{align}
and a similar construction for $k>L_{\hat\th^K}$, where the random variables  $Y^K_k$ are replaced by i.i.d.\  ones with distribution
$f_K*M(R^K,\cdot)$, independent of all the previously introduced random variables, where $f_K$
  is the homothety of ratio $\sigma_K$.
Note that, the Poisson point measures $N^{\text{birth}}_k$ and $N^{\text{death}}_k$ are independent of  $Y^K_k$ and
  $\tau_k$ and that the processes $\tilde Z^{K,1}_k$ and $\tilde Z^{K,2}_k$ only depend on $N^{\text{birth}}_k$, $N^{\text{death}}_k$, $Y^K_k$ and
  $\tau_k$.  By construction, conditionally on $Y^K_k=y$ and $\tau_k=s$, the process $\tilde Z^{K,1}_k$ is
  distributed as a linear birth and death processes with birth rate $b(y)(1-u_Km(y))$ and death rate
  $d(y)+c(y,R^K)(\bar{z}(R^K)+M\e \sigma_K)+\bar{c}\lceil 3/\alpha\rceil \e \sigma_K$, and similarly for $\tilde Z^{K,2}_k$.
  In particular, the law of $(\tilde Z^{K,1}_k,\tilde Z^{K,2}_k)$ does not depend on $\tau_k$. Therefore, defining ${\cal
    G}_k\equiv\sigma(\tilde \nu_t,t\leq\tau_k,\ Y^K_k, N^{\text{birth}}_\ell,N^{\text{death}}_\ell,1\leq\ell\leq k-1)$, for all bounded measurable functions
  $F_1,\ldots,F_k$ on $\mathbb{D}(\mathbb{R}_+,\mathbb{ Z}_+^2)$,
  \bea
   && \mathbb{E}\left[F_1(\tilde Z^{K,1}_1,\tilde Z^{K,2}_1)\ldots F_k(\tilde Z^{K,1}_k,\tilde Z^{K,2}_k)\right] 
    \\
\nonumber     && =\mathbb{E}\left[F_1(\tilde Z^{K,1}_1,\tilde Z^{K,2}_1)\ldots
      F_{k-1}(\tilde Z^{K,1}_{k-1},\tilde Z^{K,2}_{k-1})
      \mathbb{E}[F_k(\tilde Z^{K,1}_k,\tilde Z^{K,2}_k)\mid {\cal G}_k]\right] \\
    % & =\mathbb{E}\left[F_1(\tilde Z^{K,1}_1,\tilde Z^{K,2}_1)\ldots
    %   F_{k-1}(\tilde Z^{K,1}_{k-1},\tilde Z^{K,2}_{k-1})\mathbb{E}[F_k(\tilde Z^{K,1}_k,\tilde Z^{K,2}_k)\mid Y^K_k,\tau_k]\right] \\
\nonumber    && =\mathbb{E}\left[F_1(\tilde Z^{K,1}_1,\tilde Z^{K,2}_1)\ldots
      F_{k-1}(\tilde Z^{K,1}_{k-1},\tilde Z^{K,2}_{k-1})\mathbb{E}[F_k(\tilde Z^{K,1}_k,\tilde Z^{K,2}_k)\mid Y^K_k]\right] \\
  && =\mathbb{E}\left[F_1(\tilde Z^{K,1}_1,\tilde Z^{K,2}_1)\ldots
      F_{k-1}(\tilde Z^{K,1}_{k-1},\tilde Z^{K,2}_{k-1}) \right]\mathbb{E}[F_k(\tilde Z^{K,1}_k,\tilde Z^{K,2}_k)],\nonumber
    \eea
  where the last equality follows from the fact that the random variable \ $Y^K_k$ is independent of $(\tilde Z^{K,1}_\ell,\tilde Z^{K,2}_\ell)$ for
  $1\leq\ell\leq k-1$. Actually, \ $(Y^K_k)_{1\leq k\leq L_{\hat\th^K}}$ are i.i.d. random variables, with law
  $f_K*M(R^K,\cdot)$. This implies by induction that the processes $\{(\tilde Z^{K,1}_k,\tilde Z^{K,2}_k)\}_{k\geq 1}$ are i.i.d..
\end{proof}
%
%%%%%%%%%%%%%%%%%%%%%%%%%%%%%%%%%%%%%%%%%%%%%%%%%%%%%%%%%%%%%%%%%%%%%%%%%
%
Let us define $B_k^{K}\equiv \mathds 1_{\inf\{t\geq \tau_k\::\:\mfm^k(\tilde\nu_t) \geq \e\s_K K\}<\inf\{t\geq
  \tau_K\::\:\mfm^k(\tilde\nu_t) = 0\}}.$ This random variable indicates if the $k$-th mutant population, which appeared at time $\tau_k$,
invades or not, i.e.\ reaches $\e\s_K K$ individuals before dying out. The following lemma
introduces a sequence of i.i.d.\ random variables $(B^{1,K}_k,B^{2,K}_k)$ which are 2-tuples of Bernoulli random variables
constructed from the processes $Z^{K,1}_{k}(t)$ and $Z^{K,2}_{k}(t)$ defined in Lemma \ref{bound_mutants}, such that $(B^K_k)_{k\geq
  0}$ is stochastically dominated by the sequences $(B^{K,i}_k)_{k\geq 0}$.
\begin{lemma}\label{rv_B}
Fix $\e>0$. Suppose that the assumptions of Theorem \ref{1.Phase} hold and let $M$ be the constant of Lemma \ref{exit_from_domain}. 
Then, \begin{equation}
	\lim_{K\to \infty}\;\s_K^{-1}
			\Big(1-\mathbb P\Big [    \forall \:1\leq k\leq L^K_{\hat \th^K}\::\:  B_k^{1,K}\preccurlyeq B_k^K\preccurlyeq B_k^{2,K}	\Big]\Big)	= 0,
\end{equation}
where $((B_k^{1,K},B_k^{2,K}))_{k\geq 1}$ is a sequence of i.i.d.\ 2-tuples of Bernoulli random 
variables such that $B^{1,K}_k\leq B^{2,K}_k$ a.s. Its distribution is characterized by
\bea
  \label{eq:def-q1_K}
  \sigma_Kq_1^K(h)&\equiv&\mathbb{P}\big[B^{1,K}_k=1\mid Y^K_k=R^K+h\sigma_K\big]\\[0.2em]\nonumber
  &=&
  \begin{cases}
    \sigma_K \left(h \tfrac{\partial_1f(R^K,R^K) }{b(R^K)}-\e C^1_{\text{Bernoulli}}\right) & \text{if }1\leq h\leq A, \\
    0& \text{otherwise}
  \end{cases}
\eea
and
\bea
  \label{eq:def-q2_K}
  \sigma_K q_2^K(h)&\equiv&\mathbb{P}\big[B^{2,K}_k=1\mid Y^K_k=R^K+h\sigma_K\big]\\\nonumber
  &=&
  \begin{cases}
    \sigma_K\left(h\tfrac{\partial_1f(R^K,R^K) }{b(R^K)}+\e C^2_{\text{Bernoulli}}\right) & \text{if }1\leq h\leq A, \\
    0& \text{otherwise},
  \end{cases}
\eea
where $C^1_{\text{Bernoulli}}$ and $C^2_{\text{Bernoulli}} $ depend only on $\alpha$, $M$ and $C_L$ (the Lipschitz constant of our
parameters). Then, for $i=1,2$ and $k\geq 1$, $B^{i,K}_k$ is a Bernoulli random variable of parameter $\sigma_K p^K_i$, where
\begin{equation}
  \label{eq:def-pi_K}
  p^K_i\equiv\sum_{h=1}^A q_i^K(h)M(R^K,h).
\end{equation}
\end{lemma}
\begin{remark}
\begin{enumerate}[(i)]
		\setlength{\itemsep}{3pt}
\item For all $k\geq 1$, 
		$\;\mathbb P[B_k^{1,K}=0\big|B_k^{2,K}=1]=1-\frac{p^K_1}{p^K_2}$ and is thereby of order $\e$.
\item We use in here the assumption that $\partial_1 f(x,x)>0$ for all $x\in \mathcal X$. 
\end{enumerate}
\end{remark}
%
%%%%%%%%%%%%%%%%%%%%%%%%%%%%%%%%%%%%%%%%%%%%%%%%%%%%%%%%%%%%%%%%%%%%%%%%%%
%
\begin{proof} Let  $Z^{K,1}_{k}(t)$ resp.  $Z^{K,2}_{k}(t)$ as defined in Lemma \ref{bound_mutants} and define 
\begin{equation}
\tilde B^{i,K}_k\equiv \mathds 1_{\inf\{t\geq \tau_k\::\:Z^{K,i}_{k}(t) \geq \e\s_K K\}<\inf\{t\geq \tau_k \::\:Z^{K,i}_{k}(t) = 0\}} \quad \text{ for }i=1,2 .
\end{equation} Then, due to the last lemma
\begin{equation}
\mathbb P\Big[\forall 1\leq k\leq L^K_{\hat\th^K}: \:\tilde B^{1,K}_k \preccurlyeq B^K_{k}\preccurlyeq \tilde B^{2,K}_k\Big]
					=1-o(\s_K).
\end{equation}
For all $\forall 1\leq k\leq L^K_{\hat\th^K}$, we obtain with Proposition \ref{prop2}, that
\bea
\left | \mathbb P\Big[\inf\{t\geq \tau_k:Z^{K,i}_{k}(t) \geq \e\s_K K\}<\inf\{t\geq \tau_k : Z^{K,i}_{k}(t) = 0\}\big|\: Y^K_k\Big]-
\tfrac{[b_k^{i,K}-d_k^{i,K}]_+}{b_k^{i,K}}\right|\:&&\\\nonumber
=o\left( \exp(-K^{\a})\right),&&
\eea
where, using that $f(x,x)=0$ for all $x$, we have   
\bea
 b_k^{1,K}-d_k^{1,K}
&=&f(Y_k^{K}, R^K)-(c(Y^K_k,R^K) M+\bar c \lceil 3/\a\rceil )\e\s_K-u_K b(Y_k^{K})m(Y_k^{K}) \\\nonumber
&=&\partial_1 f(R^K,R^K)(Y_k^{K,1}-R^K)-(c(Y^K_k,R^K) M+\bar c \lceil 3/\a\rceil )\e\s_K+ O(\s_K^2),
\eea
and similarly 
\be
b_k^{2,K}-d_k^{2,K} =\partial_1 f(R^K,R^K)(Y_k^{K}-R^K)+c(Y^K_k,R^K) M\e\s_K+ O(\s_K^2) .
\ee
Recall that the sequence $(Y_k^{K})_{k\geq 1}$ used to construct the processes $Z^{K,1}_k$ and $Z^{K,2}_k$ is a sequence of i.i.d.\
random variables with distribution $M(R^K,\cdot)$. Since $b_k^{i,K}-d_k^{i,K}<0$ if $Y_k^{K}-R^K<0$, we obtain
\bea\label{Bernoulli1}
	\qquad \mathbb P\big[\tilde B_k^{1,K}=1\big]&=&\mathbb E\big[\mathbb P\big[\tilde B_k^{1,K}\big|Y^K_k\big]=1\big]\\\nonumber
				&\geq& \sum_{h\in \{1,\ldots ,A\}}\left( \tfrac{ \partial_1f(R^K,R^K)\s_K h -(c(Y^K_k,R^K) M+\bar c \lceil 3/\a\rceil )\e\s_K+ O(\s_K^2)}{b(R^K)} \right)M(R^K,h).
\eea
Therefore, there exists a constant $C^1_{\text{Bernoulli}}>0$ (which depends only on $\alpha$, $M$ and $C_L$) such that the sum in
the right hand side of (\ref{Bernoulli1})
is, term by term, bounded from below by
\bea
 \s_K\sum_{h\in \{1,\ldots ,A\}}\left ( h \tfrac{\partial_1f(R^K,R^K) }{b(R^K)}-\e C^1_{\text{Bernoulli}}\right)M(R^K,h)
 \eea
and similarly there exists a constant $C^2_{\text{Bernoulli}}>0$ such that
\be
	\mathbb P[\tilde B_k^{2,K}=1]  \leq  \s_K \sum_{h\in \{1,\ldots ,A\}}\left( h\tfrac{\partial_1f(R^K,R^K) }{b(R^K)}+\e C^2_{\text{Bernoulli}} \right)M(R^K,h).
\ee
Next, we introduces  two couplings, i.e.\ we define a sequences of i.i.d.\ 2-tuples of Bernoulli random variables  $((B_k^{1,K},B^{2,K}_k))_{k\geq 1}$ with the following properties
\begin{enumerate}[(i)]
		\setlength{\itemsep}{6pt}
\item $\mathbb P\big[B_k^{1,K}=0, \tilde B_k^{1,K}=0 \mid Y^K_k=R^K+h\sigma_K\big]= \mathbb P\big[\tilde B_k^{1,K}=0  \mid Y^K_k=R^K+h\sigma_K\big]$
\qquad and \\[0.2em]
$\mathbb P\big[B_k^{1,K}=1, \tilde B_k^{1,K}=1  \mid Y^K_k=R^K+h\sigma_K\big]= q_1^{K}(h)\s_K $
\item $\mathbb P\big[B_k^{2,K}=1, \tilde B_k^{2,K}=1  \mid Y^K_k=R^K+h\sigma_K\big]= \mathbb P\big[\tilde B_k^{2,K}=1  \mid Y^K_k=R^K+h\sigma_K\big]$
\qquad and \\[0.2em]
$\mathbb P\big[B_k^{2,K}=1, \tilde B_k^{2,K}=0  \mid Y^K_k=R^K+h\sigma_K\big]= 1- q_2^{K}(h)\s_K. $
\end{enumerate}
By construction, $B_k^{1,K}\leq \tilde B_k^{1,K}$ a.s.\ and $\tilde B_k^{2,K}\leq B_k^{2,K}$ a.s.\ for all $k\geq 1$ and these random
variables satisfy~\eqref{eq:def-q1_K} and~\eqref{eq:def-q2_K}.
\end{proof}
%
%%%%%%%%%%%%%%%%%%%%%%%%%%%%%%%%%%%%%%%%%%%%%%%%%%%%%%%%%%%%%%%%%%%%%%%%%%
%
\begin{notations}\begin{enumerate}[(a)]
		\setlength{\itemsep}{3pt}
\item For $i\in\{1,2\}$, define 
\be
T^{K,i}_k\equiv\inf\left\{t\geq 0: Z_k^{K,i}(\tau_k+t)=0\; \text { or }  \;Z_k^{K,i}(\tau_k+t) >\e\s_K K\right\}.
\ee  Obverse that $(T^{K,i}_k)_{k\geq 1}$ are i.i.d.\ random variables that are independent of $A^{K,i}$. 
\item Define $I^{K}\equiv\inf\{k\geq 1 :B^{K}_k=1\}$ and $I^{K,i}\equiv \inf\{k\geq 1 :B^{K,i}_k=1\}$.  
Then, $I^{K,i}$ are independent of $A^{K,i}$, and we have
\be
\mathbb P\left[ \left\{I^{K,2}\preccurlyeq I^K \preccurlyeq I^{K,1} \right\}\cap\bigl\{\tau_{I^K}\leq \hat \theta^K\bigr\}\right] =\mathbb P\left[\tau_{I^K}\leq \hat \theta^K\right]-o(\s_K).
\ee
\item  Define $R_1^K\equiv Y^{K}_{\inf\{k\geq 1 : B^K_k=1\}}$.											 
\end{enumerate}
\end{notations}
In fact, we prove at the end of this section that $\mathbb P\big[\tau_{I^K}\leq \hat \theta^K\big]=1-o(\s_K)$, 
i.e.  $R_1^K$ is with high probability the random variable which gives the value of the next resident trait and
%
%
%%%%%%%%%%%%%%%%%%%%%%%%%%%%%%%%%%%%%%%%%%%%%%%%%%%%%%%%%%%%%%%%%%%%%%%%%%
$\tau_{I^K}$, the first time where a successful mutant appears, is approximately exponential distributed as stated in lemma below. 
Note that this time is a random time, but not a stopping time.
\begin{lemma} \label{bounds_for_invasion_time}
Fix $\e>0$. Suppose that the assumptions of Theorem \ref{1.Phase} hold and let $M$ be the constant of Lemma \ref{exit_from_domain}. Then,
\begin{align}
	\lim_{K\to \infty}\:\s_K^{-1}
		\Big(\mathbb P\left[\tau_{I^K}\leq \hat \theta^K\right]-\mathbb P\Big [  \big\{E^{K,2}\preccurlyeq \tau^{}_{I^K}\preccurlyeq E^{K,1}\big\}\cap
													\big\{\tau_{I^K}\leq \hat \theta^K\big\}\Big] \Big)	=	0,
\end{align}
where $ E^{K,1}$ and $E^{K,2}$ are exponential random variables with mean
$a^K_1p^K_1 \s_K u_K K $ respectively  $a^K_2p^K_2$ $ \s_K u_K K $. 
\end{lemma}
With other words, we have $\mathbb P\big [  E^{K,2}\preccurlyeq \tau^{}_{I^K}\preccurlyeq E^{K,1}\big| \tau_{I^K}\leq \hat \theta^K\big]=1-o(\s_K)$, 
provided that $\liminf_{K\to\infty}\mathbb P\big[\tau_{I^K}\leq \hat \theta^K\big]>0$.
%
%
%%%%%%%%%%%%%%%%%%%%%%%%%%%%%%%%%%%%%%%%%%%%%%%%%%%%%%%%%%%%%%%%%%%%%%%%%%
%
%
\begin{proof} 
Let $A_t^{K,i}$ be defined as in Lemma \ref{rv_A} and observe that $\tau^{}_{I^K}=\inf\left\{t\geq 0: L^{K}_t=I^K\right\}$.
Then, we obtain by construction,
\bea
\mathbb P\bigg [ \Big\{ \inf\big\{t\geq 0\!: A^{K,2}_t=I^{K,2}\big\} \preccurlyeq \tau_{I^K}
					\preccurlyeq \inf\big\{ t\geq 0: A^{K,1}_t=I^{K,1}\big\} \Big\} \cap\left\{ \tau_{I^K}\leq \hat \theta^K\right\}
					\bigg]&&\\\nonumber
			=\mathbb P\left[\tau_{I^K}\leq \hat \theta^K\right]-o(\s_K).&&
\eea
By definition we have that $I^{K,1}$ and $I^{K,2}$ are geometrically distributed with parameter  $p^K_1 \:\s_K $ resp. 
$p^K_2 \:\s_K$ and $A^{K,1}$ and $A^{K,2}$ are Poisson counting processes with parameter $a^K_1 u_K K$ resp. $a^K_2u_K K$. 
Therefore, the time between each pair of successive events is exponential distributed with parameter $a^K_1 u_K K$ resp. $a^K_2u_K K$. 
Since the random variables $I^{K,i}$ are independent of $A^{K,i}$ and  the sum of geometrically distributed many independent,
 identically exponentially distributed random variables is exponentially distributed, we get that
$\quad \inf\{t\geq 0: A^{K,1}_t=I^{K,1}\}$\quad and \quad$\inf\{t\geq 0: A^{K,2}_t=I^{K,2}\}$\quad 
are exponentially distributed with parameter
$a^K_1u_K K p^K_1$ respectively  $a^K_2u_K K p^K_2$. \end{proof}
%
%
%%%%%%%%%%%%%%%%%%%%%%%%%%%%%%%%%%%%%%%%%%%%%%%%%%%%%%%%%%%%%%%%%%%%%%%%%%
%
In the next lemma we prove that a mutant invades with high probability before the resident population exits the neighborhood of this equilibrium, 
before too many different mutant traits are present and before a mutant of a mutant appears.
\begin{lemma} \label{diversity}
Fix $\e>0$. Suppose that the assumptions of Theorem \ref{1.Phase} hold and let $M$ be the constant of Lemma \ref{exit_from_domain}.
Then,
\begin{equation}
\lim_{K\to\infty}\s_K^{-1} \mathbb P \left[\th^{K}_{\text{invasion}}\geq \th^{K}_{ \text{diversity}}\wedge \exp({K^{\a}})\wedge \th^K_{ \text{mut. of mut.}}\right]=0.
\end{equation}
\end{lemma}
%
%%%%%%%%%%%%%%%%%%%%%%%%%%%%%%%%%%%%%%%%%%%%%%%%%%%%%%%%%%%%%%%%%%%%%%%%%%%%%%%%%%
%
\begin{proof}
We start with proving the following
\begin{equation}\label{bound_diversity}
\mathbb P\left[ \th^{K}_{ \text{diversity}}<(K u_K\s_K^{1+\a })^{-1}\wedge\th^{K}_{\text{invasion}}\wedge \th^K_{ \text{mut. of mut.}}\right]=o(\s_K).
\end{equation}
Define
\begin{align}
\hat Z_k^{K,2}(s)\equiv\begin{cases} 0 &\text{for }s< \inf\{t\geq 0\!:\!A^{K,2}_t=k\}\\
						Z_k^{K,2}\left(\tau_k +s-\inf\{t\geq 0\!:\!A^{K,2}_t=k\}\right)&\text{for } s\geq \inf\{t\geq 0\!:\!A^{K,2}_t=k\}.
\end{cases}
\end{align}
By construction of $A^{K,2}$ and $\hat Z^{K,2}$, the left hand side of (\ref{bound_diversity}) does not exceed
\begin{equation}\label{bound_div_2}
\mathbb P\bigg[ \inf\bigg\{t\geq 0:{\sum_{k=1}^{A^{K,2}_t}}\: \mathds 1_{\left \{1\leq \hat Z^{K,2}_k(t)\leq \e\s_K K\right\}}\geq\lceil 3/\a\rceil-1\bigg\}
<\lb K u_K\s_K^{1+\a}\rb^{-1} \bigg]
+o(\s_K).
\end{equation}
Next, we compute an upper bound for the mutation events that happen before $( K u_K\s_K^{1+\a })^{-1}$. 
Since $A^{K,2}$ is a Poisson counting process with parameter $a_2^K u_K K$,  Chebychev's inequality 
implies that
\begin{equation}\label{number_mutations}
				\mathbb P\left[A^{K,2}_{\lb K u_K\s_K^{1+\a}\rb^{-1}}\geq 2 a_2^K \s_K^{-1-\a} \right]
				\leq\text{Var}\bigg( A^{K,2}_{\lb K u_K\s_K^{1+\a}\rb^{-1}} \bigg) \lb 2 a_2^K\s_K^{-1-\a}\rb^{-2}=\frac{1}{a_2^K \s_K^{-1-\a}}.
\end{equation}
Next we need an upper bound  for the lifetimes of the mutants traits, $T_k^{K,2}$. 
First, observe that the probability that $Z_k^{K,2}$ goes extinct after it has reached the value $\lceil \e\s_K K\rceil$ converges to zero very fast.
More precisely, Proposition \ref{prop2} and \ref{prop2.1} (a) imply that   
\begin{align} 
&\mathbb P\left[\inf\bigl\{t\geq 0: Z_k^{K,2}=\lceil \e\s_K K\rceil\bigr\}<\inf\bigl\{t\geq \tau_k: Z_k^{K,2}=0\bigr\}<\infty\right]\\
&=\mathbb P\left[\inf\bigl\{t \!\geq \! \tau_k\! : Z_k^{K,2}\! =\! 0\bigr\}<\infty\right]-\mathbb P\left[\inf\bigl\{t \!\geq \! 0 \!: Z_k^{K,2} \!= \!\lceil \e\s_K K\rceil\bigr\}>\inf\bigl\{t \!\geq \! \tau_k \!: Z_k^{K,2}\! =\! 0\bigr\}\right]\nonumber\\
&=o(\exp(-K^{\a})).\nonumber
\end{align}
Note that, for each k, $Z_k^{K,2}$, conditioned on extinction, is a subcritical linear birth and death process (cf. \cite{J_BPCE}). Let $\check  Z_k^{K,2}$ denote the conditioned process. If $Z_k^{K,2}$ is subcritical, 
then conditioning  has no effect, otherwise the birth and death rates are exchanged. 
Denote by $\check  b_k^{K,2}$ the birth rate  and $\check  d_k^{K,2}$ the death rate of $\check  Z_k^{K,2}$. 
Then there exist uniform constants, $\check  C_1>0$ and $\check  C_2>0$, such that $\check  C_1 \s_K\leq \check  d_k^{K,2}-\check
b_k^{K,2}\leq \check  C_2 \s_K$, for all $k< I^{K,2}$.  Thus, \cite{A_BP} p.\ 109 entails, for all $k< I^{K,2}$,
\begin{equation}
\mathbb P\left[T_k^{K,2}\leq t\right] \geq \frac{\check  d_k^{K,2}-e^{(\check  d_k^{K,2}-\check  b_k^{K,2})t}\check  d_k^{K,2}}{\check  b_k^{K,2}-e^{(\check  d_k^{K,2}-\check  b_k^{K,2}) t}\check  d_k^{K,2}}-o(\exp(-K^{-\a})).
\end{equation} 
The error term $o(\exp(-K^{-\a}))$  appears since $Z_k^{K,2}$, for $k< I^{K,2}$,
 is conditioned on extinction before reaching the value $\lceil \e\s_K K\rceil$ and not only on extinction.
Choose $t=(\check  d_k^{K,2}-\check  b_k^{K,2})^{-1} \ln(K)$,  Then,
\begin{align} \nonumber
\mathbb P\left[T_k^{K,2}\leq (\check  d_k^{K,2}-\check  b_k^{K,2})^{-1} \ln(K)\right]&=\frac{\check  d_k^{K,2}(1-K)}{\check  b_k^{K,2}(1-K)-K(\check  d_k^{K,2}-\check  b_k^{K,2})}-o(\exp(-K^{-\a}))
\\\nonumber&=1+\frac{\check  d_k^{K,2}-\check  b_k^{K,2}}{\check  b_k^{K,2}(1-K)-K(\check  d_k^{K,2}-\check  b_k^{K,2})}-o(\exp(-K^{-\a}))
\\&=1-O(\s_K K^{-1})
\end{align} 
and hence
\begin{align}
\mathbb P\left[\forall 1\leq k< I^{K,2}: T_k^{K,2}\leq(\check  C_1 \s_K)^{-1} \ln(K)\right ]=1-o(\s_K).
\end{align} 
Therefore, we can bound the first summand of (\ref{bound_div_2}) by $2 a_2^K\s_K^{-1-\a}$ times the probability that more than
$\lceil 3/\a\rceil-1$ mutation events of $A^{K,2}$ take place in an interval of length $(\check  C_1 \s_K)^{-1} \ln(K)$. More precisely, (\ref{bound_div_2}) is smaller than
\begin{equation}
2 a_2^K\s_K^{-1-\a} \mathbb P\left[A^{K,2}_{(\check  C_1 \s_K)^{-1} \ln(K)}\geq \lceil 3/\a\rceil-1\right]+o(\s_K).
\end{equation} 
Thus, for $\a$ small enough,  the proof of (\ref{bound_diversity}) is concluded by the observation that
\bea
&&\mathbb P\left[A^{K,2}_{(\check  C_1 \s_K)^{-1} \ln(K)}\geq \lceil 3/\a\rceil-1\right]\\\nonumber
&&\qquad=e^{-a_2^K u_K K(\check  C_1 \s_K)^{-1} \ln(K)}\sum_{i=\lceil 3/\a\rceil-1}^{\infty}\frac{(a_2^K u_K K(\check  C_1 \s_K)^{-1} \ln(K))^i}{i!}\\\nonumber
&&\qquad\leq  \lb a_2^K u_K K(\check  C_1 \s_K)^{-1} \ln(K)\rb^{\lceil 3/\a\rceil-1}\\
&&\qquad=o(\s_K^{3- \a}),\nonumber
\eea  
where the last equality holds since $ u_K K \s_K^{-1} \ln(K)\ll (\s_K)^{\a}$. 
%%%%%%%%

Next, we want to prove that
\begin{equation}
\label{bound_mut_of_mut}
\mathbb P\left[ \th^K_{ \text{mut. of mut.}}<(K u_K\s_K^{1+\a })^{-1}\wedge\th^{K}_{\text{invasion}}\wedge\th^{K}_{ \text{diversity}} \right]=o(\s_K).
\end{equation}
%%%%%%%
Set, for all $\lambda\geq 0$,
\be
G(\lambda)=\mathbb E\left[\left.\exp\left(-\lambda\int_0^{\infty}Z_t\,dt\right)\,\right|\, Z_0=1\right],
\ee
where $(Z_t,t\geq 0)$ is a linear birth and death process
with individual birth rate $b$ and individual death rate $d$. Applying the strong Markov property and the branching property at the first jump time of $Z$ and using the facts that $G(\lambda)^2=\mathbb E\left[\left.\exp\left(-\lambda\int_0^{\infty}Z_t\,dt\right)\right| Z_0=2\right]$ and 
$\mathbb E\left[\left.\exp\left(-\lambda \tau^{}_{\text{first jump}}\right)\,\right|\, Z_0=1\right]=\frac {b+d}{b+d+\l}$, we obtain
\be
b G(\lambda)^2-(b+d+\lambda) G(\lambda) +d=0.
\ee
Thus, since 
\bea\textstyle
\lim_{\lambda \downarrow 0 }G(\lambda)&=&
\textstyle\lim_{\lambda \downarrow 0 }\mathbb E\left[\left.\exp\left(-\lambda\int_0^{\infty}Z_t\,dt\right)\mathds 1_{\{\tau^{}_{\text{extinction}}<\infty \}}\,\right|\, Z_0=1\right]\\\nonumber
&&\textstyle+\:\lim_{\lambda \downarrow 0 }\mathbb E\left[\left.\exp\left(-\lambda\int_0^{\infty}Z_t\,dt\right)\mathds 1_{\{\tau^{}_{\text{extinction}}=\infty \}}\,\right|\, Z_0=1\right]\\
\nonumber
&=&\mathbb P[\tau^{}_{\text{extinction}}< \infty]+0,
\eea
 which is 0 in the subcritical case and $1-d/b$
in the supercritical case, it follows that 
\be
G(\lambda)=
  \frac{b+d+\lambda-\sqrt{(b+d+\lambda)^2-4bd}}{2b}.
\ee
Let $\tilde Z^{K,2}_k(t)\equiv  Z^{K,2}_k(\tau_k+t)$, i.e. a linear birth and death process
with  birth rate $b^{K,2}_k$ and death rate $d^{K,2}_k$. Observe that $\int_0^{\infty}\tilde Z^{K,2}_k(t)\,dt$ 
gives an upper bound for the sum of the lifetimes of all individuals with label $k$. Since the mutation rate of any individual in the
population is smaller than $\bar b u_K$, the probability that an mutant appears, which was born from an unsuccessful mutant with label $k$,  is bounded from above by
\bea\label{prob_mut_of_mut}
  1- \mathbb E\left[\left.\exp\left(-u_K\bar b\int_0^{\infty}\tilde Z^{K,2}_k(t)\,dt\right)\right| \tau^{}_{\text{extinction}}<\inf\{t\geq 0:\tilde Z^{K,2}_k(t)>\e \sigma_K K\}\right]&& \qquad\\
    \leq1- \mathbb E\left[\left.\exp\left(-u_K\bar b\int_0^{\infty}\tilde Z^{K,2}_k(t)\,dt\right)\right| \tau^{}_{\text{extinction}}<\infty\right]+o(\exp(-K^{\a})).&&\nonumber
\eea
Since $\tilde Z^{K,2}_k(t)$, conditioned on extinction, is a sub-critcal linear birth and death process,
the right hand side of (\ref{prob_mut_of_mut}) is equal to $1-G_{\mathbb E[ \tilde Z^{K,2}_k| \tau^{}_{\text{extinction}}<\infty]}(u_K \bar b)+o(\exp(-K^{\a}))$
and
\begin{align}\nonumber
G_{\mathbb E[ \tilde Z^{K,2}_k| \tau^{}_{\text{extinction}}<\infty]}(u_K \bar b)
& =  
\begin{cases}
 \frac{b_k^{K,2}+d_k^{K,2} + u_K \bar b -\sqrt{(b_k^{K,2}+d_k^{K,2} + u_K \bar b)^2-4b_k^{K,2}d_k^{K,2}}}{2b_k^{K,2}}& 
 						\text{if }d_k^{K,2}>b_k^{K,2} \\[0.5em]
\frac{d_k^{K,2}+b_k^{K,2} + u_K \bar b -\sqrt{(d_k^{K,2}+b_k^{K,2} + u_K \bar b)^2-4d_k^{K,2}b_k^{K,2}}}{2d_k^{K,2}} & \text{if }b_k^{K,2}>d_k^{K,2}
  \end{cases}\\[0.2em]\nonumber
   &=  
\begin{cases}
 \frac{2 b_k^{K,2}+ u_K \bar b - O(u_K \s_K^{-1} )}{2b_k^{K,2}}& 
 						\text{if }d_k^{K,2}>b_k^{K,2} \\[0.5em]
\frac{2d_k^{K,2}+ u_K \bar b -O(u_K \s_K^{-1} )}{2d_k^{K,2}} & \text{if }b_k^{K,2}>d_k^{K,2}
  \end{cases}\\[0.2em]
  &=1-O(u_K\s_K^{-1})=1-o(\s_k^{2+\a}K^{-2\a}).
\end{align}
Note that we used for the second equality that $|b_k^{K,2}-d_k^{K,2}|=\xi\sigma_K$ for some $\xi>0$.
By~(\ref{number_mutations}), the total number of unsuccessful mutations until 
$(K u_K\s_K^{1+\a })^{-1}\wedge\th^{K}_{\text{invasion}}\wedge\th^{K}_{ \text{diversity}}$ is with probability $1-o(\s_K)$ smaller 
or equal  $2a^K_{2}\sigma^{-1-\a}_K$. Therefore, we finally obtain that the probability
to have one mutant of a unsuccessful mutant during that time is $o(\sigma_K)$. 
On the other hand, let $P^{K}_t$ be a Poisson counting process with parameter $\bar b u_K \e\s_K K $ 
and $(\tilde Z^{K,1}_t, t\geq 0)$  a linear birth and death process with initial state $1$ and
 birth rate $b^{K,1}(Y_{I_K}^K)$ and death rate $d^{K,1}(Y_{I_K}^K)$, then the probability to
 have one mutant of the successful mutant until the time 
 $(K u_K\s_K^{1+\a })^{-1}\wedge\th^{K}_{\text{invasion}}\wedge\th^{K}_{ \text{diversity}}$ is 
 bounded from above by
 \begin{align}\label{no_mutant_of_succ}
		  \mathbb P& \Big[P^{K}_{\tau^{\tilde Z^{K,1}}_{\e\s_K K}}\neq 0\Big|\tau^{\tilde Z^{K,1}}_{ \e\s_K K} < \tau^{\tilde Z^{K,1}}_{0} \Big]+o(\s_K) 													
		  	\\&  =\mathbb E \Big[\mathds 1_{\big\{P^{K}_{\tau^{\tilde Z^{K,1}}_{\e\s_K K}}\neq 0\big\}}\nonumber
			\left(\mathds 1_{\{\tau^{\tilde Z^{K,1}}_{\e\s_K K}\leq t_K\}}
			    +\mathds 1_{\{\tau^{\tilde Z^{K,1}}_{\e\s_K K}> t_K\}} \right)
		\Big|\tau^{\tilde Z^{K,1}}_{ \e\s_K K} < \tau^{\tilde Z^{K,1}}_{0} \Big]+o(\s_K) 	
		\\&  \leq(1-\exp( -\bar b u_K\e\s_K K t_K))
						    +\mathbb P\Big[\tau^{\tilde Z^{K,1}}_{\e\s_K K}> t_K 
		\Big|\tau^{\tilde Z^{K,1}}_{ \e\s_K K} < \tau^{\tilde Z^{K,1}}_{0} \Big]+o(\s_K),\nonumber 	
\end{align}
for each $t_K$, because the mutation rate per individual is bounded by $\bar b u_K$ and there are at most $\e\s_K K$
 successful mutant individuals alive until $\th^K_{\text{invasion}}$.
  If we choose $t_K=\ln(K)\s_K^{-1-\a/2}$, then by Proposition \ref{prop2.1}, all terms in the last line of (\ref{no_mutant_of_succ}) are $o(\s_K)$.
This implies (\ref{bound_mut_of_mut}).
 
Note that we have $\th^{K}_{\text{invasion}} = \tau^{}_{I^K} +\inf\big \{t\geq 0 : \mfm^{I^K}(\tilde \nu_{ \tau^{}_{I^K}+t})>\e\s_K K\big\}$.
Let $E^{K,1}$ be a exponential distributed random variable with mean $a^K_1p^K_1 \s_K u_K K $. Then, 
\bea
	 \mathbb P \left[ \tau^{}_{I^K} +\inf\big \{t\geq 0 : \mfm^{I^K}(\tilde \nu_{ \tau^{}_{I^K}+t})
				>\e\s_K K\big\} \geq \th^{K}_{ \text{diversity}} \wedge \exp({K^{\a}})\wedge \th^K_{ \text{mut. of mut.}}\right]&&\\
		\geq\mathbb P \left[E^{K,1}+T^{K,1}_{I^K}\geq (K u_K\s_K^{1+\a})^{-1} \right]-o(\s_K).&&\nonumber
\eea
Let $\tilde Z^{K,1}$ as defined before, then again by Proposition \ref{prop2.1} 
\begin{equation}\label{bound_T}
\mathbb P\left[T^{K,1}_{I^K}>\ln(K)\s_K^{-1-\a/2}\right] = \mathbb P\left [\tau^{\tilde Z^{K,1}}_{\e\s_K K}>\ln(K)\s_K^{-1-\a/2} \Big | \tau^{\tilde Z^{K,1}}_{\e\s_K K}<\tau^{\tilde Z^{K,1}}_{0}\right]=o(\s_K). 
\end{equation} 
Since $\ln(K)\s_K^{-1-\a/2} \ll (K u_K\s_K^{1+\a})^{-1}$, the Markov inequality for the function
 $f(x)=x^{n}$, where $n$ is smallest even number which is larger than $2/\a$, yields 
\bea
	\mathbb P \left[ E^{K,1}+T^{}_{I^K}> (K u_K\s_K^{1+\a})^{-1} \right]&\leq& \mathbb P \left[ E^{K,1}> (2 K u_K\s_K^{1+\a})^{-1} \right]+o(\s_K)\\
					&\leq&\frac{ (2 K u_K \s_K^{1+\a})^{n} n!} {(a^K_1 p^K_1u_K K\s_K)^{n}}= O(\s_K^{2})\nonumber.
\eea
\end{proof}
%
%%%%%%%%%%%%%%%%%%%%%%%%%%%%%%%%%%%%%%%%%%%%%%%%%%%%%%%%%%%%%%%%%%%%%%%%
%
The following lemma shows that there are no two successful mutants during the first phase of an invasion.
\begin{lemma} \label{2.succ_mutant}
Fix $\e>0$. Suppose that the assumptions of Theorem \ref{1.Phase} hold and let $M$ be the constant of Lemma \ref{exit_from_domain}.
Then,
\bea
	\lim_{K\to\infty} \s_K^{-1} \;
 			\mathbb P \Big[ \text{ There is a successful mutation in time }\text{interval } [\tau^{}_{I^K},\th^{K}_{\text{invasion}} ] \Big]&=&0.
\eea
\end{lemma}
%
%%%%%%%%%%%%%%%%%%%%%%%%%%%%%%%%%%%%%%%%%%%%%%%%%%%%%%%%%%%%%%%%%%%%%%%%
%
\begin{proof}
Let $P^{K}_{\text{suc. mut.}}(t)$ the process which recodes the number of successful mutants born after \: $\tau^{}_{I^K}$
\: until $\tau^{}_{I^K}\:+\: t$.
Then, 
\begin{equation}
	 \mathbb P \left[ \text {for all } t\geq 0 \text{ such that }  \tau^{}_{I^K}\:+\: t <\hat \th^{K}\::\: P^{K}_{\text{succ. mut.}}(t)\preccurlyeq P^K_t \right]=1-o(\s_K),
\end{equation}
where $P^K_t$ is Poisson process with parameter $a^K_2p^K_2\s_K u_K K $. 
Define $Z_{I^K}^{K,2}(t)$  as in Lemma \ref{bound_mutants}. Then 
$\mathbb P[\forall t\leq\hat \th^{K} : \mfm^{I^K}(\tilde\nu_t) \preccurlyeq Z_{I^K}^{K,2}(t)]\geq1-o(\s_K)$.
Note that $P^{K}_t$ and $Z^{K,2}$ are independent by construction.  
 Therefore, as in the last lemma, or each $t_K$,
\begin{align}\label{no_2_succ}
&\mathbb P \Big[ \text{ There is a}\text{ successful mutation in } [\tau^{}_{I^K},\th^{K}_{\text{invasion}} ] \:\Big]
	\\\nonumber&\:\leq \:\mathbb P \Big[P^{K}_{\tau^{Z^{K,2}}_{\e\s_K K}}\neq 0\Big|\tau^{Z^{K,2}}_{ \e\s_K K} < \tau^{Z^{K,2}}_{0} \Big]+o(\s_K) 								\\\nonumber&\: \leq\:(1-\exp( -a^K_2p^K_2\s_K u_K K  t_K))
						    +\mathbb P\Big[\tau^{Z^{K,2}}_{\e\s_K K}> t_K
		\Big|\tau^{Z^{K,2}}_{ \e\s_K K} < \tau^{Z^{K,2}}_{0} \Big]+o(\s_K).	
\end{align}
With $t_K=\ln(K)\s_K^{-1-\a/2}$, by  Proposition \ref{prop2.1}, all terms in the last line of (\ref{no_2_succ}) are $o(\s_K)$.
\end{proof}
%%%%%%%%%%%%%%%%%%%%%%%%%%%%%%%%%%%%%%%%%%%%%%%%%%%%%%%%%%%%%%%%%%%%%%%%%%
%
The following corollary gives an approximation for the distribution of the next resident trait. 
\begin{corollary} \label{new_resident_trait}
Fix $\e>0$. Suppose that the assumptions of Theorem \ref{1.Phase} hold and let $M$ be the constant of Lemma \ref{exit_from_domain}.
Then, there exist two $\mathcal X$-valued random variables $R_1^{K,1}$ and $R_1^{K,2}$ with distribution
\be
  \mathbb P[ R_1^{K,1}= R^K+\s_K h]
  =\begin{cases}
    \frac{M(R^K,1)q_1^K(1)}{p_2^K}\:+\:1-\frac{p^K_1}{p^K_2} 
    & \text{ if }h=1\\[0.5em]
    \frac{M(R^K,h)q_1^K(h)}{p_2^K} 
    & \text{ if }h\in\{2,...,A\},\qquad
  \end{cases} \label{eq:law-R1_K}
  \ee
  {and }
  \be \mathbb P[ R_1^{K,2}= R^K+\s_K h]
  =\begin{cases}
    \frac{M(R^K,h)q_1^K(h)}{p_2^K} ,
    & \text{ if }h\in\{1,...,A-1\}\\[0.5em]
     \frac{M(R^K,A)q_1^K(A)}{p_2^K}+1-\frac{p^K_1}{p^K_2} 
     & \text{ if }h=A,
      \label{eq:law-R2_K}
  \end{cases}
\ee
such that 
\begin{align}\label{R_i}
	\lim_{K\to \infty}\s_K^{-1}
			\left(1-\mathbb P\left[ R_1^{K,1}\preccurlyeq R_1^K \preccurlyeq R_1^{K,2} \Big| 
												 \th^{K}_{\text{invasion}}< \th^{K}_{ \text{diversity}}\wedge\th^K_{ \text{mut. of mut.}}\wedge \exp({K^{\a}})\right] \right)=0.
\end{align}
\end{corollary} 
%
%%%%%%%%%%%%%%%%%%%%%%%%%%%%%%%%%%%%%%%%%%%%%%%%%%%%%%%%%%%%%%%%%%%%%%%%%%
%
\begin{proof} Define 
\be
	R_1^{K,1}\equiv \begin{cases} 
					Y_{I^{K}}^{K}, & \text{ if }I^{K,1}=I^{K,2},\\
					R^K+\s_K,	& \text { otherwise},
				\end{cases} 		
	\quad\text{ and }\quad
	R_1^{K,2}\equiv \begin{cases} 
					Y_{I^K}^{K}, & \text{ if }I^{K,1}=I^{K,2},\\
					R^K+A \s_K, 		& \text { otherwise.}
				\end{cases}
\ee
By construction of $B^{K,i}_k$ and $Y_{k}^{K,i}$, we have that (\ref{R_i}) holds. Next, we compute
\bea\nonumber
  \mathbb{P}[Y^K_{I^{K,2}}=R^K+\sigma_K h,\ I^{K,1}=I^{K,2}] & 
  =&\mathbb{P}[Y^K_1=R^K+\sigma_K h,\ B^{K,1}_1=1\mid B^{K,2}_1=1] \\\nonumber
%  & =& \frac{\mathbb{P}[Y^K_1=R^K+\sigma_K h,\ B^{K,1}_1=1,\ B^{K,2}_1=1]}{\mathbb{P}[B^{K,2}_1=1]} \\\nonumber
  & =& \frac{\mathbb{P}[Y^K_1=R^K+\sigma_K h,\ B^{K,1}_1=1]}{\mathbb{P}[B^{K,2}_1=1]} \\
  & =&\frac{M(R^K,h)q^K_1(h)}{p^K_2}
\eea
and $\mathbb{P}[I^{K,1}\not= I^{K,2}]=1-\sum_{h=1}^A \frac{M(R^K,h)q^K_1(h)}{p^K_2}=1-p^{K}_1/p^K_2$.
Since $\mathbb{P}[R_1^{K,1}=R^{K}+\sigma_K h]=\mathbb{P}[Y^K_{I^{K,2}}=R^K+\sigma_K h,\
I^{K,1}=I^{K,2}]+\mathds{1}_{\{h=1\}}\mathbb{P}[I^{K,1}\not= I^{K,2}]$ and similarly for $R^{K,2}_1$, we deduce~\eqref{eq:law-R1_K} and~\eqref{eq:law-R2_K}.
\end{proof}
%

%%%%%%%%%%%%%%%%%%%%%%%%%%%%%%%%%%%%%%%%%%%%%%%%%%%%%%%%%%%%%%%%%%%%%
%%Second Phase
%%%%%%%%%%%%%%%%%%%%%%%%%%%%%%%%%%%%%%%%%%%%%%%%%%%%%%%%%%%%%%%%%%%%%

\section{The Second Phase of an Invasion}\label{2.Phase}
\begin{notations}
Let us denote 
\be 
\theta^K_{\text{fixation}}=\inf\left\{t\geq \theta^K_{\text{invasion}}: |\text{Supp}(\tilde \nu^K_{t})|=1\text{ and } |\langle \tilde \nu_t,\mathds 1\rangle-\bar z(R_1^K)|<(M/3)\e\s_K\right\}
\ee
i.e. the first time after $\theta^K_{\text{invasion}}$ such that the population is 	
	monomorphic and in the $(M/3)\e\s_K$-neighborhood of the corresponding equilibrium.
\end{notations}

Again we start with a theorem, which summarizes several of the results of this section.

\begin{theorem}\label{Thm_2.Phase}
Fix $\e>0.$ Under the Assumptions \ref{ass}, \ref{ass3} and \ref{ini_con},  there exists a constant, 
$M>0$, such that, for all $K$ large enough, 
\begin{enumerate}[(i)]
		\setlength{\itemsep}{3pt}
 \item $\tilde\nu^K_{0}=N_{R^K}^K K^{-1}\delta_{(0,R^K)}$, where  $\left|\overline z(R^K)-N^K_{R^K}K^{-1}\right|<(M/3) \e \sigma_K$ a.s..	
 \item At the first time of invasion, $\theta^K_{\text{invasion}}$,
  the resident density is in an $\e  M\sigma_K$-neighborhood of $\bar{z}(R^K)$,
the number of different living mutant traits is bounded by $\lceil \alpha/3\rceil$ and there is no 
mutant of a mutant, with probability $1-o(\sigma_K)$. 
(cf. Theorem \ref{1.Phase})
\item The time between $\theta^K_{\text{invasion}}$ and  $\th^K_{\text{fixation}}$ is smaller than $5 \ln(K)\s_K^{-1-\a/2}$, with probability $1-o(\s_K)$.
\item The trait of the population at time $\th^K_{\text{fixation}}$ is the trait of the 
mutant  whose density was larger than $\e\s_K$ at time $\theta^K_{\text{invasion}}$, 
i.e. $\text{Supp}(\tilde \nu^K_{\th^K_{\text{fixation}}})=\left(I^K, R^{K}_1\right)$, 
with probability $1-o(\s_K)$. The distribution of $R^{K}_1$ can be approximated as in Corollary \ref{new_resident_trait}.
\end{enumerate}
Moreover, until  time $\th^K_{\text{fixation}}$, the total mass of the population stays in the 
$O(\s_K)$-neighborhood of $\bar z(R^K)$, the number of different living mutant traits is bounded 
by $\lceil\a/3\rceil$, and there is no
second successful mutant, with probability $1-o(\s_K)$.
\end{theorem}

%%%%%%%%%%%%%%%%%%%%%%%%%%%%%%%%%%%%%%%%%%%%%%%%%%%%%%%%%%%%%%%%%%%%%%%%
%

To prove this, we will divide this phase into five steps, illustrated in  Figure \ref{fig2}. 

\begin{description}
\item[Step 1]
		{ From $\th^K_{\text{invasion}}$ to $\th^K_{\text{mut. size }\e}$,
		 the first time  a mutants density reaches the value $\e$.
		 During this period we approximate the mutant density by a  continuous time branching process, which is supercritical 
		 (of order $\s_K$). Thus we obtain that $\th^K_{\text{mut. size }\e}-\th^K_{\text{invasion}}$ is of order $(\ln(K)\s_K^{-1})$.}
\item[Step 2]
		{ From $\th^K_{\text{mut. size }\e}$ to $\th^K_{\text{mut. size }C^{\e}_{\text{cross}}}$, the first time
		 the mutant density reaches a  value $C^{\e}_{\text{cross}}$ (defined in Eq. \eqv(cross.1) below).
		 This step can be seen as the "stochastic Euler scheme". The idea is that the total
		 mass of the population stays close to a function which depends only on the density of
		  the successful mutant.
		 This allows to approximate
		 the number of mutants by a discrete time 				
		 Markov chain until the mutant density has increased by $\e$. Furthermore we control the number of jumps needed to increase by $\e$ 
		 using upper and lower bounds for one jump time of the associated continuous time process. Then we recompute the parameters 
		 and  start again. 
		 Iterating, we obtain that $\th^K_{\text{mut. size }C^{\e}_{\text{cross}}} - \th^K_{\text{mut. size }\e}$ 
		 is also of order $\ln(K)\s_K^{-1}$.}
\item[Step 3] 
		{From $\th^K_{\text{mut. size }C^{\e}_{\text{cross}}}$ until $\th^K_{\text{res. size }\e}$, 
		the first time such that the density of the resident trait $R^K$
		 decreases to the value $\e$. The proof is very similar to the proof of Step 2, the only 	
		 difference is that we approximate the number of resident individuals by a discrete Markov chain, which decreases slowly. }
\item[Step 4] 
		{From  $\th^K_{\text{res. size }\e}$ until  $\th^K_{\text{res. size }0}$,
		 the first time such that the resident trait $R^K$ goes extinct.
		We  approximate the dynamics of the resident trait by a continuous time branching process 
		which is subcritical (of order $\s_K$) and therefore
		 goes extinct, a.s., after a time of order $\ln(K)\s_K^{-1}$. }
\item[Step 5]
		{From  $\th^K_{\text{res. size } 0}$ until $\th^K_{\text{fixation}}$.
		Even if it is unlikely that  this  time period is larger than 0, 
		we have to obtain an upper bound for this time.}		 
\end{description}
%%%%%%%%%%%%%%%%%%%%%%%%%%%%%%%%%%%%%%%%%%%%%%%%%%%%%%%%%%%%%
 \begin{figure}[ht]
\centering\includegraphics[scale=0.24]{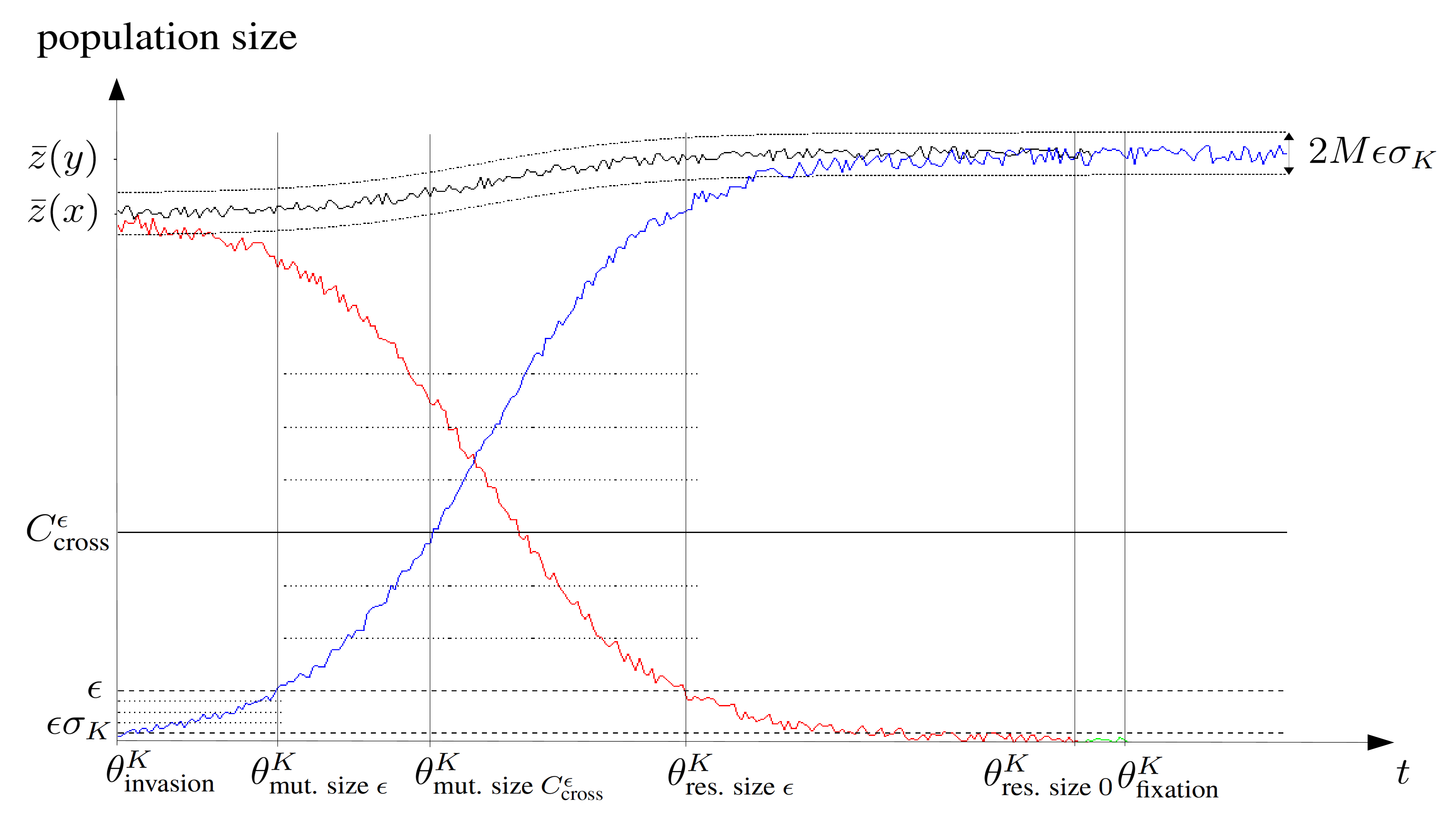}
  \caption{\label{fig2} Evolution of the population after the destiny of the successful mutant has reached the value $\e_K$. }
\end{figure}
%%%%%%%%%%%%%%%%%%%%%%%%%%%%%%%%%%%%%%%%%%%%%%%%%%%%%%%%%%%%%
\begin{notations} 
Fix $\e>0$. Suppose that the assumptions of Theorem \ref{Thm_2.Phase} hold. 
Set
\bea\Eq(cross.1)
C^{\e}_{\text{cross}}	&\equiv&	\left\lceil\left(\inf_{x\in\mathcal X}\tfrac{b(x)-d(x)}{c(x,x)}\right)\e^{-1}\right\rceil\frac \e 2,
\Eq(cross.2)\\
\text{and}\qquad \th^K_{2\text{ succ. mut.}}&\equiv&
\inf\left\{t\geq 0:\sum_{k=0}^{\infty}\mathds 1_{\mfm^k(\tilde\nu_t) \geq \e\s_KK}\geq 3 \right\}.
\eea
Moreover, for any $\xi \geq 0$,
\bea\Eq(cross.3)
\th^K_{\text{mut. size }\xi}	&\equiv&\inf\left\{t\geq 0:\exists k\geq 1: \mfm^{k}(\tilde\nu_t) =\lceil \xi K\rceil\right\},\\
\Eq(cross.4)
\th^K_{\text{res. size } \xi}&\equiv&\inf\left\{t\geq 0: \mfm^{0}(\tilde\nu_t) =  \lceil \xi K\rceil\right\},\eea
and let $S_K$ be a sequence in $K$ such that $1\ll S_K\ll \e\s_K^{-1}$. 
\end{notations}
%
%%%%%%%%%%%%%%%%%%%%%%%%%%%%%%%%%%%%%%%%%%%%%%%%%%%%%%%%%%%%%
%
\begin{remark}\label{Remark_2Phase} Using similar arguments as in the proofs of Lemma 
\ref{exit_from_domain}, \ref{diversity} and \ref{2.succ_mutant}, we obtain 
\be
\lim_{K\to\infty}\s_K^{-1} \mathbb P \left[\th^{K}_{\text{invasion}}+5\s_K^{-1-\a/2}\ln(K)>\th^{K}_{ \text{diversity}}\wedge\th^K_{2\text{ succ. mut.}} \wedge \exp({K^{\a}})\right]=0.
\ee
More precisely, until the time $\th^{K}_{ \text{diversity}}\wedge\th^K_{2\text{ succ. mut.}} \wedge \exp({K^{\a}})$ the total mass of the population stays with high probability in the $O(\s_K)$ neighborhood of $\bar z(R^K)$. This can be proved similarly as  Lemma 
\ref{exit_from_domain} or  \ref{exit_from_domain_2}. Since we have only a approximation of oder $\s_K$ (not $\e\s_K$), 
we have less precise bounds for the rates of the mutants and for their success probability. Nevertheless, 
we can bound the mutant subpopulations from above by linear branching processes which are  slightly supercritical of order $\s_K$.
\end{remark}
%
%%%%%%%%%%%%%%%%%%%%%%%%%%%%%%%%%%%%%%%%%%%%%%%%%%%%%%%%%%%%%
\subsection{Step 1}
The following lemma shows that the total mass stays from the beginning (including the first phase) until $\th^K_{\text{mut. size }\e}$ in the $M\e\s_K$ neighborhood of $\overline z(x)$.
\begin{lemma}
\label{exit_from_domain_2} Fix $\e>0.$ Suppose that the assumptions of Theorem \ref{Thm_2.Phase} hold. Then, there exists a constant $M>0$ (independent of $\e$ and $K$) such that
\bea\label{eq_exit_from_domain_2}
&&	\lim_{K\to\infty}\s_K^{-1}\;
\mathbb P\Big[\:\inf\left \{t\geq 0: |\langle \tilde \nu_t,\mathds 1\rangle-\overline z(R^K)|> M\e\s_K\right\} \\
&&\hspace{3cm}	< \th^K_{\text{mut. size }\e}\wedge\th^K_{2\text{ succ. mut.}}  \wedge \th^{K}_{ \text{diversity}}\wedge \exp({K^{\a}})\Big] = 0.
\nonumber
\eea
\end{lemma}
%
%%%%%%%%%%%%%%%%%%%%%%%%%%%%%%%%%%%%%%%%%%%%%%%%%%%%%%%%%%%%%%
%
\begin{proof} The proof of this lemma is very similar to the one of Lemma \ref{exit_from_domain},
 therefore we omit some  details.
Define 
\be 
X_t\equiv|\langle \tilde \nu_t,\mathds 1\rangle K-\lceil K\overline z(R^K) \rceil |.
\ee			
We associate with the continuous time process $X_t$ a discrete time (non-Markov) process $Y_n$ 
which records the sequence of values that $X_t$ takes on.\\[0.5em]
\noindent\textbf{Claim:}
 \textit{For  $1\leq i\leq \e K$, and $K$ large enough,}
  \begin{align}
 \label{1.cond_pro_2}\mathbb P\bigl[Y_{n+1}=i+1|Y_n=i, T_{n+1}&<\th^K_{\text{mut. size }\e}\wedge\th^K_{2\text{ succ. mut.}}\wedge \th^{K}_{ \text{diversity}}\bigr] \\&\leq \frac 1 2 -(\underline c/ 4\overline b )K^{-1}i+ (2C^{b,d,c}_L A/ \underline b ) \epsilon \sigma_K \equiv p_+^K(i),\nonumber
\end{align}
 \textit{where $C^{b,d,c}_L $ is the sum of the Lipschitz constants for the birth, death and competition rate.}\\[0.5em]
This can be proven exactly as in Lemma \ref{exit_from_domain}, using that $b(R^K)=d(R^K)+c(R^K,R^K)\bar z(R^K)$ and that all mutant traits are at a distance of at most $2A \s_K$ from $R^K$, and hence, $|b(x)-b(R^K)|<C^{b}_L\s_K 2A$, $|d(x)-d(R^K)|<C^{d}_L\s_K 2A$ and $|c(x,y)-c(R^K,R^K)|<C^{c}_L\s_K 2A$ for all traits $x$ and $y$ alive in the population. By continuing as in Lemma \ref{exit_from_domain} we obtain (\ref{eq_exit_from_domain_2}). 
\end{proof}
%
%%%%%%%%%%%%%%%%%%%%%%%%%%%%%%%%%%%%%%%%%%%%%%%%%%%%%%%%%%%%%%
%
Next we prove that $\th^K_{\text{invasion}}-\th^K_{\text{mut. size }\e}$ is smaller than $\ln(K)\s_K^{-1-\a/2}$ and we use the following notation.
%
%%%%%%%%%%%%%%%%%%%%%%%%%%%%%%%%%%%%%%%%%%%%%%%%%%%%%%%%%%%%%
\begin{notations}$
\tilde \th^K\equiv \inf\left \{t\geq 0: |\langle \tilde \nu_t,\mathds 1\rangle-\overline z(R^K)|> M\e\s_K\right\} 
														\wedge\th^K_{2\text{ succ. mut.}} \wedge \th^{K}_{ \text{diversity}}.
$
\end{notations}
\begin{lemma}\label{Step1}
 Fix $\e>0$. Suppose that the assumptions of Theorem  \ref{Thm_2.Phase} hold. Let $M$ be the constant from
  Lemma \ref{exit_from_domain_2}.
Then,  
\be
	\lim_{K\to\infty} \s_K^{-1}\;
				\mathbb P\Big[\:\th^K_{\text{mut. size }\e}
									>\big(\th^K_{\text{invasion}}+\ln(K)\s_K^{-1-\a/2}\big)\wedge\tilde \th^K \Big]=0.
\ee
\end{lemma}
%
%%%%%%%%%%%%%%%%%%%%%%%%%%%%%%%%%%%%%%%%%%%%%%%%%%%%%%%%%%%%%%
%
\begin{proof} 
To prove this lemma we use  a coupling with an linear continuous time birth and death process. 
From Phase 1 and the last lemma we know that $\th^{K}_{ \text{invasion}}$ is  
with probability $1-o(\s_K)$ smaller than
$\tilde \th^K$.
Define $k_1\equiv I^K$ the label of the 
first successful mutation.
For any $t\in(\th^K_{\text{invasion}},\tilde \th^K]$,  any individual of $\mfm^{k_1}(\tilde\nu_t)$ gives birth to a new individual with the same trait,  $R_1^K$, with rate 
\begin{align}
\bigl(1-u_K \:m(R_1^K)\bigr)b(R_1^K)\in \bigl[b(R_1^K)-u_K\:\overline b\:, b(R_1^K)\bigr],
 \end{align}
 and dies with rate 
\begin{align}
 d(R_1^K)+\int_{\mathcal X\times \mathbb N_0}c(R_1^K,\xi) d\tilde\nu_t(\xi),
 \end{align}
which is smaller than $d_Z\equiv d(R_1^K)+c(R_1^K,R^K)(\overline z(R^K)+M\e\s_K)+\overline c (\e+\lceil 3/\a\rceil \s_K )A\s_K$.
Similarly as in Lemma \ref{bound_mutants} we construct, by using a standard coupling argument, a processes $Z_t$ such that 
\begin{align} Z_t\leq  \mfm^{k_1}(\tilde\nu_{\th^K_{\text{invasion}}+t})\end{align}
for all $t $ such that $\th^K_{\text{invasion}}+t\leq \tilde \th^K \wedge\inf \{t\geq 0: \mfm^{k_1}(\tilde\nu_t)\geq \e K\}$.
The processes $Z_t$ is a branching process starting at $\lceil \e \s_K K\rceil$, with birth rate per individual  $b_Z=b(R_1^K)-\bar b
u_K$ and with death rate per individual
 $d_Z$. For all $\e< \inf_{x\in\mathcal X}\tfrac{\partial_1 f(x,x)}{2(M+A+1)}$, we have 
   \begin{align}
 b_Z-d_Z&\geq f(R_1^K,R^K)-\overline c \s_K (M\e+A(\e+\lceil 3/\a\rceil\s_K))
 \geq \s_K \inf_{x\in\mathcal X}\tfrac{\partial_1 f(x,x)}{2}.
 \end{align} 
Thus $Z_t$ is super-critical of order $\sigma_K$. Let $\tau^Z_i$ be the first hitting time of level $i$ by $Z_t$, then
by Proposition \ref{prop2.1}
\be
\mathbb P[\tau^Z_{\lceil \e K\rceil }>\tau^Z_{0}]\leq \exp(- K^{\a}).
\ee Furthermore, we have the following exponential tail bound, see \cite{A_MC} page 41,
 \be
\mathbb P\left[\tau^Z_{\lceil \e K\rceil }\geq\ln(K)\s_K^{-1-\a/2}\big| \tau^Z_{\lceil \e K\rceil }<\tau^Z_{0}\right]\leq 
\exp \left( -\left \lfloor \frac{\ln(K)\s_K^{-1-\a/2}}{e \max_{n\leq \lceil \e K\rceil }\mathbb E_n [ \tau^Z_{\lceil \e K\rceil }|\tau^Z_{\lceil \e K\rceil }<\tau^Z_{0}]}\right\rfloor\right)
\ee
and $ \max_{n\leq \lceil \e K\rceil }\mathbb E_n [ \tau^Z_{\lceil \e K\rceil }|\tau^Z_{\lceil \e K\rceil }<\tau^Z_{0}]\leq O(\ln(K)\s_K)$ (compare with Proposition \ref{prop2}).  
Therefore, 
\be
\mathbb P[\tau^Z_{\lceil \e K\rceil }<\ln(K)\s_K^{-1-\a/2}]\geq (1-\exp(-\s_K^{-\a/3}))(1-\exp(- K^{\a}))=1-o(\s_K),
\ee
 which implies the claim.
\end{proof}
%
%%%%%%%%%%%%%%%%%%%%%%%%%%%%%%%%%%%%%%%%%%%%%%%%%%%%%%%%%%%%%%%%%
%
\subsection{Step 2}
Recall that  the trait of  the successful mutant is 
$R^K+\s_K h$ where $h\in \{1,\ldots,A\}$. Due to the regularity assumptions (iv) in Assumption 1,  
we have the following estimates:
\begin{align}
	b(R^K+\s_K h)	=\:&	b(R^K)+b'(R^K)\s_K h+O((\s_K h)^2)\\
	d(R^K+\s_K h)	=\:&	d(R^K)+d'(R^K)\s_K h+O((\s_K h)^2)\\
	r(R^K+\s_K h)	=\:&	r(R^K)+r'(R^K)\s_K h+O((\s_K h)^2)\\
	c(R^K+\s_K h,R^K)		=\:&	c(R^K,R^K)+\partial_1 c(R^K,R^K)\s_K h+O((\s_K h)^2)\\
	c(R^K,R^K+\s_K h)		=\:&	c(R^K,R^K)+\partial_2 c(R^K,R^K)\s_K h+O((\s_K h)^2)\\
	c(R^K\!+\!\s_K h,R^K\!+\!\s_K h)		=\:&	c(R^K,R^K)+\left(\partial_1 c(R^K,R^K)+\partial_2 c(R^K,R^K)\right)\s_K h \\\nonumber&+O((\s_K h)^2).
\end{align}

\emph{The deterministic system:}  Although we cannot use a law of large numbers, to understand the 
behavior of the stochastic system it is useful to look at the properties of the corresponding deterministic 
Lotka-Volterra system.  The limiting system when $K\to \infty$, with $\s_K=0$, takes the simple form 
\begin{eqnarray}
		\frac{d m_t^0}{dt}		&=& m_t^0 \left(r(R^k)-c(R^K,R^K)(m_t^0+m_t^{k_1})\right),\\
		\frac{d m_t^{k_1}}{dt}	&=&  m_t^{k_1} 	\left(r(R^K)-c(R^K,R^K)(m_t^0
		+m_t^{k_1})\right).
\end{eqnarray}
The corresponding vector field is depicted in Figure \thv(inv-man-fig.0).
This system has an invariant manifold made of fixed points given by the roots of the equation
\be\Eq(inv-man.1)
m^0+m^{k_1} =r(R^K)/c(R^K,R^K)=\bar z(R^K),
\ee
with $m^0,m^{k_1}\geq 0$. This manifold connects the fixed points of the monomorphic
equations, $(\bar z(R^K),0)$ and $(0,\bar z(R^K))$. Note that $\bar z(R^K)$ has the interpretation of the total mass of the population 
in equilibrium.
 A simple computation shows that the  Hessian matrix on the invariant manifold is given by
 \be\Eq(inv-man.2)
 H(m^0,m^{k_1})=-c(R^K,R^K)\left(\begin{matrix}  m^0&m^0\\
 m^{k_1}&m^{k_1}\end{matrix}\right).
 \ee
 The corresponding eigenvectors are $(1,-1)$ with eigenvalue $0$, and 
 $(m^0,\bar z(R^K)-m^0)$ with eigenvalue $-c(R^K,R^K) \bar z(R^K)$. 
 
 \begin{figure} 
\includegraphics[width=5cm]{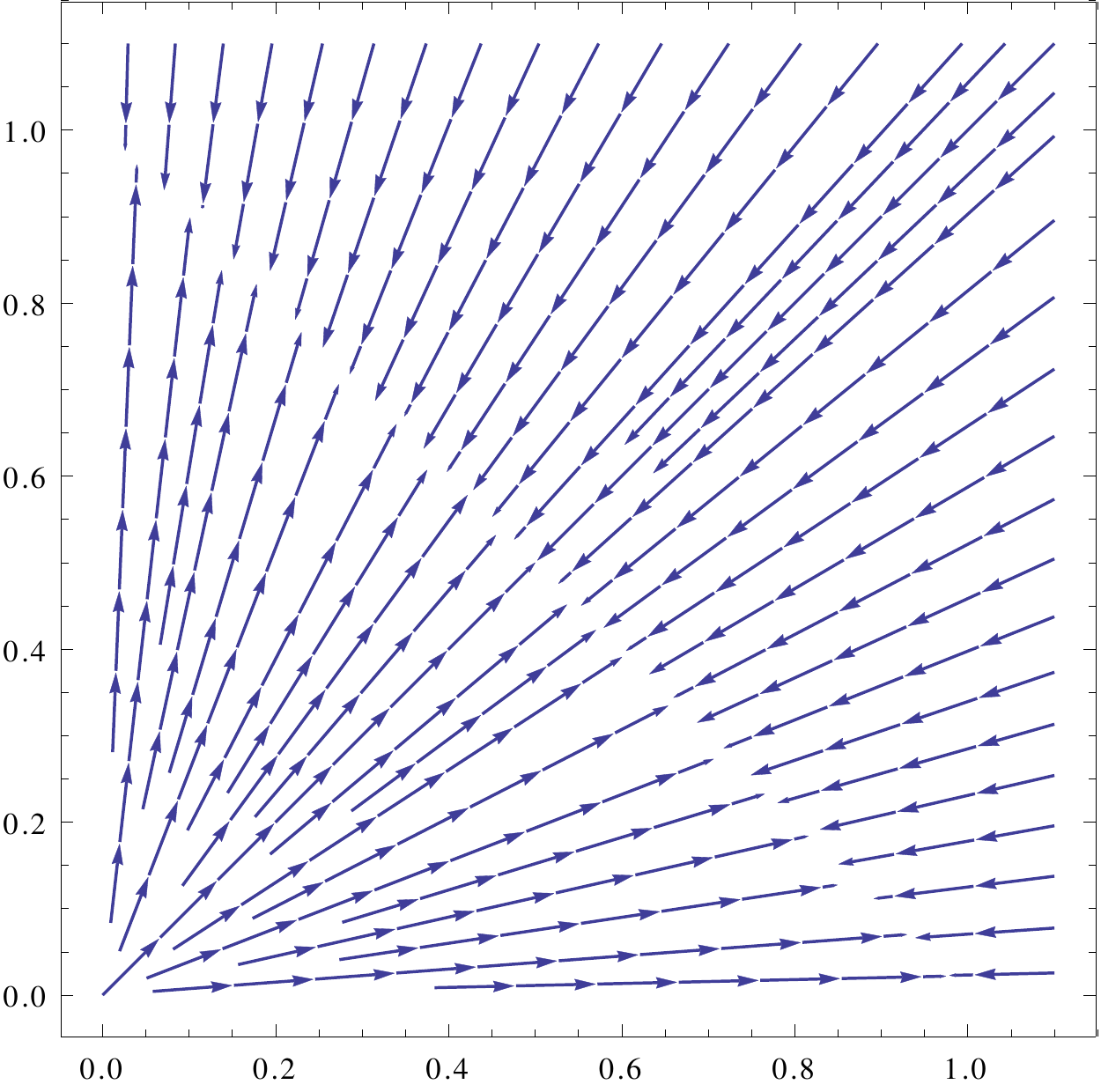}\includegraphics[width=5cm]{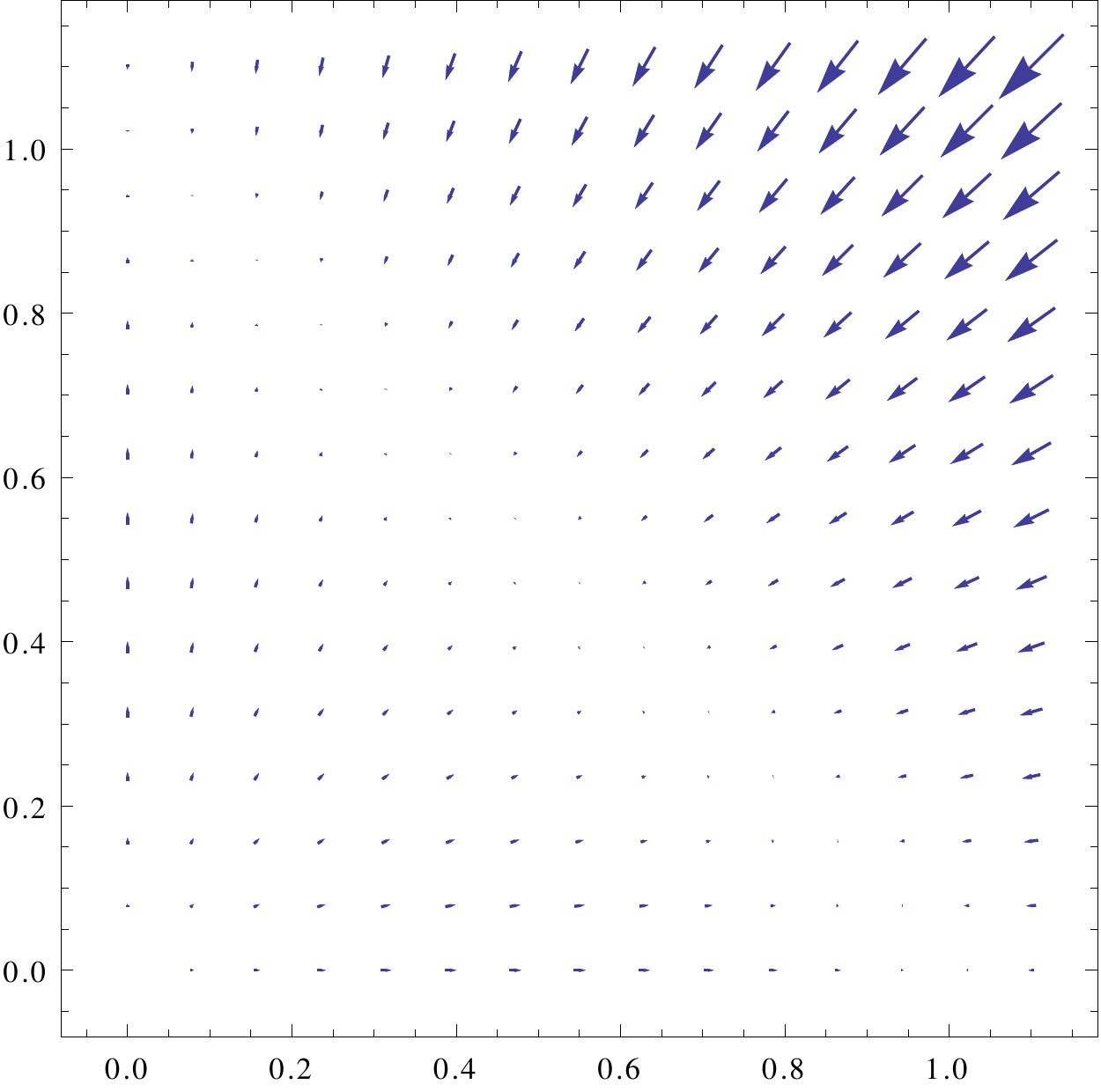}
\caption{Vector field of the unperturbed system}\label{inv-man-fig.0}
\end{figure}
 It follows that the perturbed system
\begin{align}\Eq(inv-man.4)
\frac{d m_t^0}{dt}&=  m_t^0\left(r(R^k)-c(R^K\!,R^K)m_t^0-
c(R^K\!,R^K\!+\s_K h)m_t^{k_1}\right)\\\nonumber
\frac{d m_t^{k_1}}{dt}&=  m_t^{k_1} \left(r(R^K\!+\!\s_K h)-c(R^K\!+\!\s_K h,R^K)m_t^0-
c(R^K\!+\!\s_K h,R^K\!+\!\s_K h)m_t^{k_1}\right)\!,
\end{align}
has an invariant manifold connecting its fix points 
$(\bar z(R^K),0)$ and $(0,\bar{z}(R^K+\s_Kh))$, where $\bar{z}(R^K+\s_Kh)=r(R^K+\s_Kh) / c(R^K+\s_Kh,R^K+\s_Kh)$ in an 
$\s_K$-neighborhood
of the unperturbed invariant manifold (see Figure \thv(inv-man-fig.1)). Thus the perturbed deterministic system will move quickly towards 
a small neighborhood of this invariant manifold and then move slowly with speed $O(\s_K)$ along it. Since the invariant manifold is close to the curve $m^0+m^{k_1}=\bar z(R^K)$, it is reasonable to 
choose as variables $M_t=m^0_t+m^{k_1}_t$.  The motion of the system will then be close to the curve 
$\tilde \phi(m^{k_1})$ defined by the condition that the derivative of $M_t$ vanishes for $M_t=
\tilde\phi(m^{k_1})$. 
 \begin{figure} 
\includegraphics[width=5cm]{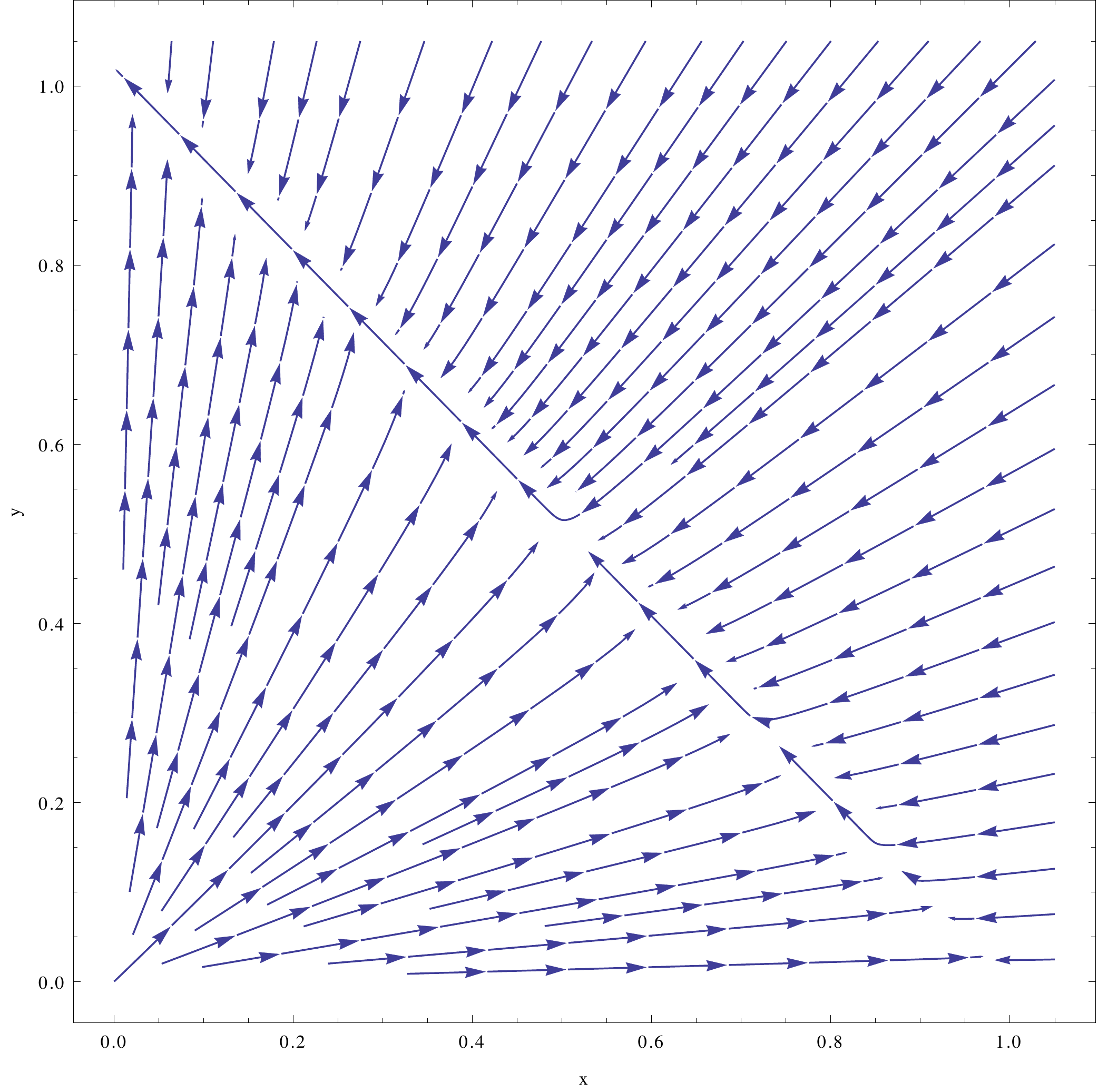}\includegraphics[width=5cm]{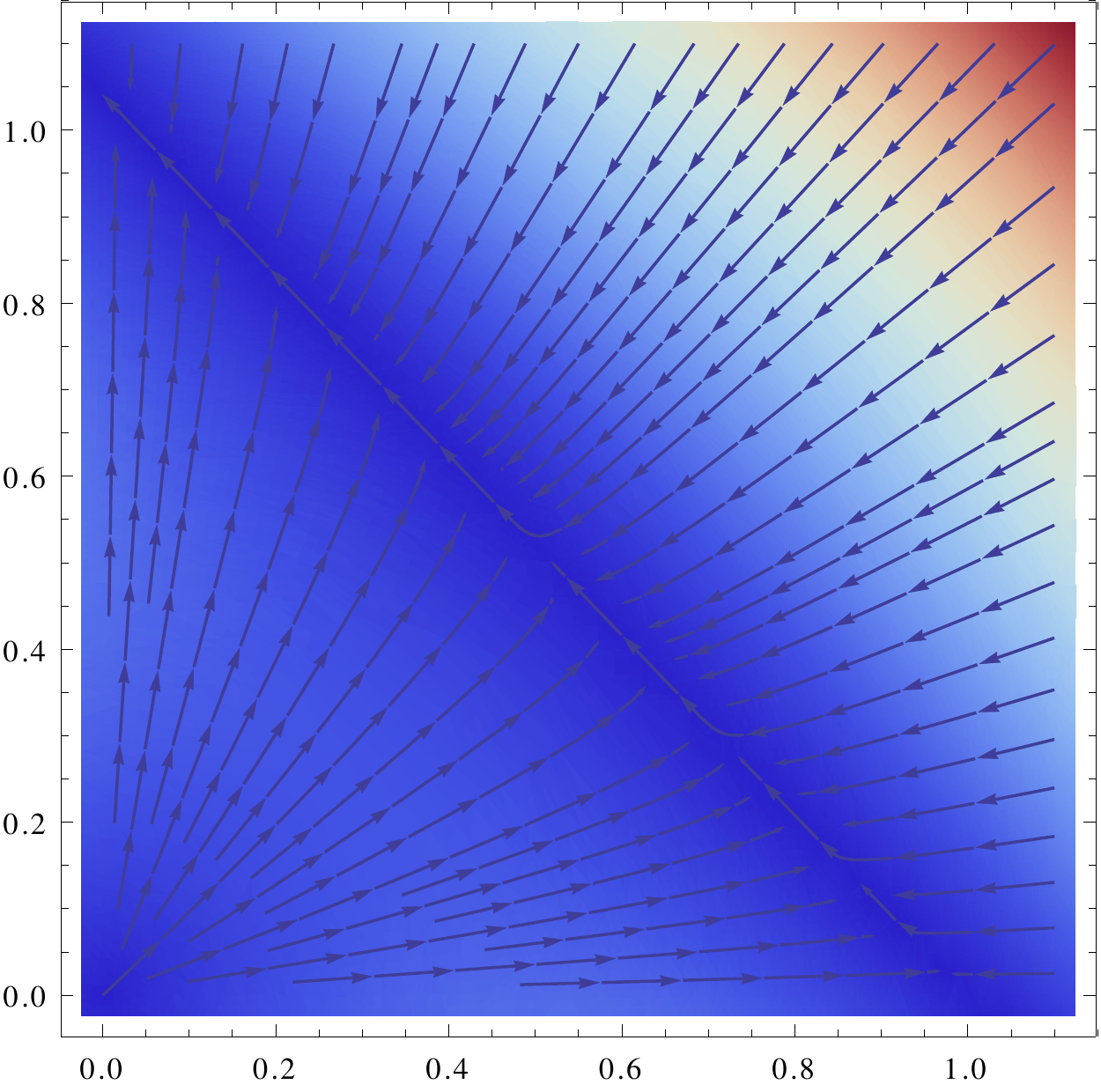}
\caption{Vector field of the perturbed system}\label{inv-man-fig.1}
\end{figure}
Since 
\bea\Eq(inv-man.3)
 \frac{d M_t}{dt} 
& =& 
M_t\left(r(R^k)-c(R^K,R^K)M_t\right)\\\nonumber
&&
-\left[\left(\partial_1 c(R^K,R^K)+\partial_2 c(R^K,R^K)\right)M_t- r'(R^K)  \right]\s_K h m_t^{k_1}+O(\s_K^2).
\eea
Setting the right hand side to zero yields the leading orders in $\s_K$
\be\Eq(inv-man.4)
\tilde \phi(m_t^{k_1}) 
=  \overline z(R^K) +\s_K h m_t^{k_1}\left(\frac{r'(R^K)}{ r(R^K)}
-\frac{\left(\partial_1 c(R^K,R^K)+\partial_2 c(R^K,R^K)\right)}{c(R^k,R^k)}\right)+O(\s_K^2).
\ee
We expect that the stochastic system also evolves along this curve. I.e., we will show that 
$m^{k_1}$ increases while the total mass stays close to the curve defined in \eqv(inv-man.4).

Define the function 
\be\Eq(phi.1)					
\phi(y)\equiv \overline z(R^K) +\s_K h y\left( \frac{r'(R^K)}{ r(R^K)}-\frac{\partial_1 c(R^K,R^K)+\partial_2 c(R^K,R^K)}{c(R^k,R^k)}\right),
\ee
and the stopping time 
\be\Eq(pi.2)		
\th_{\text{near } \phi(i\frac \e 2)}^K\equiv \inf\left\{t\geq \th^K_{\text{mut. size }i (\e/ 2)}:
|\langle\tilde \nu_t,\mathds 1\rangle-\phi(i(\e/2))|<(M/3)\e\s_K\right\}.
\ee
The dependence of $\phi$ with respect to the mutant density allows us to decompose the increase of the mutant density into successive steps during which the total mass does not move more than $M\e\s_K$. 
%
%%%%%%%%%%%%%%%%%%%%%%%%%%%%%%%%%%%%%%%%%%%%%%%%%%%%%%%%%%%%%%%%%%%
%
\begin{lemma}
\label{Step2}
Fix $\e>0$. Suppose that the assumptions of Theorem  \ref{Thm_2.Phase} hold. Then, there exists a constant $M>0$ (independent of $\e$, $K$ and $i$) such that
and for all $2\leq i\leq 2\e^{-1}C^{\e}_{\text{cross}}$,
\begin{enumerate}[(a)]
		\setlength{\itemsep}{3pt}
\item	Soon after $\th^K_{\text{mut. size }i (\e/ 2)}$, the total population size is close to $\phi(i\frac \e 2)$:
	 \begin{align}
		\lim_{K\to\infty}\:\s_K^{-1}\:
		\mathbb P\Big[ &\th_{\text{near } \phi(i\frac \e 2)}^K	
			>\left(\th^K_{\text{mut. size }i(\e/ 2)}+S_K \right)\wedge\th^K_{2\text{ succ. mut.}} \wedge\th^{K}_{ \text{diversity}}	\\&\quad
		\wedge\inf\left\{t\geq \th^K_{\text{mut. size }i(\e/2)}:\exists k\geq 1: \mfm^{k}(\tilde\nu_t) =\lceil (i\pm \tfrac12)(\e/ 2) K\rceil\right\}	\Big]=0.\nonumber
		\end{align}	
\item A change of order $\e $ for the mutant density takes more than $o(\s_K^{-1})$ time: 		
	\begin{align}		
		\lim_{K\to\infty}\:\s_K^{-1}\:
 			\mathbb P\Big[&\inf\left\{t\geq \th^K_{\text{mut. size }i(\e/2)}:\exists k\geq 1: \mfm^{k}(\tilde\nu_t) 
 				=\lceil (i\pm \tfrac12)(\e/2) K\rceil\right\}\\&
					\qquad\quad<\left(\th^K_{\text{mut. size }i (\e/ 2)}+S_K \right)\wedge \th_{\text{near } \phi(i\frac \e 2)}^K\wedge\th^K_{2\text{ succ. mut.}} 
					\wedge\th^{K}_{ \text{diversity}}	\Big]	=	0.	\nonumber													
	\end{align}	
\item At the time when the mutant density has changed of order $\e$	 the total population size is still close to $\phi(i\frac \e 2)$:
	\begin{align}																								
		\lim_{K\to\infty} \s_K^{-1}\:
			\mathbb P\Big[&\inf\left\{t\!\geq \!\th_{\text{near } \phi(i\frac \e 2)}^K\!:\! |\langle \tilde \nu_t,\mathds 1\rangle-\phi(i(\e/2))|> M\e\s_K\right\}
					< \th^K_{2\text{ succ. mut.}}\!\wedge\!\th^{K}_{ \text{diversity}} \\&\qquad
						\wedge \inf\left\{t\geq  \th^K_{\text{mut. size }i(\e/2)}\!:\!\exists k\geq 1: \mfm^{k}(\tilde\nu_t) \!=\!\lceil (i\pm 1)(\e/ 2) K\rceil\right\} 
											\Big]	=0.			\nonumber
	\end{align}
\item  A change of order $\e $ for the mutant density takes no more than $(i \s_K)^{-1-\a/2}$ time: 			
	\begin{align}									
		\lim_{K\to\infty} \s_K^{-1}\;
			\mathbb P&\Big[\:\th^K_{\text{mut. size }(i+1) (\e/ 2)}
							>\bigl(\th_{\text{near } \phi(i\frac \e 2)}^K+(i \s_K)^{-1-\a/2}\bigr)\wedge\th^K_{2\text{ succ. mut.}} 
								\;\:\\&\quad							
								\wedge\th^{K}_{ \text{diversity}}\wedge\inf\left \{t\geq \th_{\text{near } \phi(i\frac \e 2)}^K: |\langle \tilde \nu_t,\mathds 1\rangle-\phi(i(\e/2))|> M\e\s_K\right\} \Big]
												=0.	\nonumber																																				
	\end{align}
\end{enumerate}
\end{lemma} 
\begin{remark}
For each $\e>0$, Lemma \ref{Step2} implies that the  mutant density reaches the value  $C^{\e}_{\text{cross}}$ with high probability, since $\e$ is independent of $K$. Moreover, for all $\e>0$,  
\begin{align}
			\mathbb P\left[\th^K_{\text{mut. size } C^{\e}_{\text{cross}}}>(\th^K_{\text{mut. size } \e}+\ln(K)\s_K^{-1-\a/2})\wedge\th^K_{2\text{ succ. mut.}}
										 \wedge\th^{K}_{ \text{diversity}}\right]&=o(\s_K)\\
\text{and} \qquad 
			\mathbb P\left[\:|\langle \tilde \nu_{\th^K_{\text{mut. size } C^{\e}_{\text{cross}}}},\mathds 1\rangle-\phi(C^{\e}_{\text{cross}})|>M\e\s_K\right]&=o(\s_K).
\end{align}
\end{remark}
%
%%%%%%%%%%%%%%%%%%%%%%%%%%%%%%%%%%%%%%%%%%%%%%%%%%%%%%%%%%%%%%%%%%%%%%%%%%%%%%%%%
%
\begin{proof}
We will prove the lemma by induction over $i$. 
Base clause: Compare with Lemma \ref{exit_from_domain_2} and \ref{Step1} that there exists a constant $M>0$ such that
$|\langle\tilde \nu_{\th^K_{\text{mut. size }\e}},\mathds 1\rangle-\phi(0)|$ is smaller than $M\e\s_K$  
and that $\th^K_{\text{mut. size }\e}<\th^K_{2\text{ succ. mut.}} \wedge\th^{K}_{ \text{diversity}}$ both with probability $1-o(\s_K)$. 

Induction step form $i-1$ to $i$: Assume that the lemma holds true for $i-1$, then be prove separately that (a)-(d) are true for $i$, as long as $i<2\e^{-1}C^{\e}_{\text{cross}}$
%
%%%%%%%%%%%%%%%%%%%%%%%%%%%%%%%%%%%%%%%%%%%%%%%%%%%%%%%%%%%%%%%%%%%%%%%%%%%%%%%%%
\begin{proof}\textit{of (a) for $i$ by assuming that the lemma holds for $i-1$.}
In the proof we use the following notation 
\be
		\tilde \th_{i}^K\equiv  \th^K_{2\text{ succ. mut.}}\wedge\th^{K}_{ \text{diversity}}%\wedge\th^K_{\text{mut. size }(i+\frac14)\e}
						\wedge\inf\{t\geq \th^K_{\text{mut. size }i(\e/2)}:\exists k\geq 1: \mfm^{k}(\tilde\nu_t) =\lceil (i\pm \tfrac12)(\e/2) K\rceil \}.
\ee
Note that ${\tilde \th^K}_{i}$ differs from ${\tilde \th^K}$ defined in Lemma \ref{Step1}. We will prove $(a)$ provided 
it happens before  ${\tilde \th^K}_{i}$ and we use the estimates of step $(b)$ for $i$ to prove that it indeed happens before  ${\tilde \th^K}_{i}$ with high probability.

If the Lemma is true for $i-1$, we know that (with (d))
\be
\P\left[|\langle\tilde \nu_{\th^K_{\text{mut. size }i(\e/2)}},\mathds 1\rangle-\phi( (i-1)(\e/2))|<M\e\s_K\right]
=1-o(\s_K).
\ee
Since $\phi(x)-\phi(y)=O( h(x-y) \s_K)$, we have with probability $1-o(\s_K)$ either 
\be
 \inf\big\{t\geq \th^K_{\text{mut. size }i(\e/2)}:|\langle\tilde \nu_t,\mathds 1\rangle-\phi(i(\e/2))|<(M/3)\e\s_K\big\}=\th^K_{\text{mut. size }i(\e/2)},
 \ee
 which implies (a) for $i$, or at least
\be
	|\langle\tilde \nu_{\th^K_{\text{mut. size }i(\e/2)}},\mathds 1\rangle-\phi(i(\e/2))|
				<\left(M+\left| h\left(\tfrac{ r'(R^K)}{ r(R^K)}-\tfrac{(\partial_1 c(R^K,R^K)+\partial_2 c(R^K,R^K))}{c(R^k,R^k)}\right)\right|\right)
							\e\s_K.
\ee
Similarly as in many previous lemmata we 
want to couple $K\langle\tilde \nu_{t},\mathds 1\rangle$ with a discrete time Markov Chain. 
Therefore, let 
\be
	X^{i}_t=|K\langle\tilde \nu_{t},\mathds 1\rangle-\lceil \phi({i}(\e/2))K\rceil |,
\ee
and   $T^{i}_0=\th^K_{\text{mut. size }i (\e/2)}$  and $(T^{i}_k)_{k\geq1}$ be the sequences of the jump times of $\langle\tilde \nu_{t},\mathds 1\rangle$ 
after $\th^K_{\text{mut. size }i(\e/2)}$. Then let $Y^{i}_k$ be the associated discrete time process  
which records the values that $X^{i}_t$ takes after time $\th^K_{\text{mut. size }i(\e/2)}$.
\\[0.2em]
\textbf{Claim:}
 \textit{There exists a constant $C_{\text{derivative}}^{b,d,c}>0$ such that for all
 $\lceil C_{\text{derivative}}^{b,d,c}  \e\s_KK\rceil\leq j< \lceil \e K\rceil$ and $K$ large enough, 
  \begin{align}\label{prob_Step2}
 \mathbb P\bigl[Y^i_{n+1}=j+1|Y^i_n=j, T_{n+1}&<\tilde \th_i^K\bigr]
 					 \leq \:\frac 1 2 -\e\s_K=:p_+^K.
\end{align}
\textit{Moreover, we can choose 
\be
C_{\text{derivative}}^{b,d,c}=\sup_{x\in\mathcal X}\tfrac{1}{c(x,x)}(4b(x)+A|\tfrac{r'(x)c(x,x)}{r(x)}-\partial_1 c(x,x)-\partial_2 c(x,x)|).
\ee
}} 
\\[0.2em]
If $\langle\tilde \nu_t,\mathds 1\rangle K>\lceil \phi(i(\e/2))K\rceil$ at time $t=T^i_n$, then 
$\langle\tilde \nu_{T^i_n},\mathds 1\rangle K= \lceil \phi(i(\e/2))K\rceil+Y^i_{n}$ and conditionally on $\mathcal F_{T^i_n}$ the left hand side of (\ref{prob_Step2})
is equal to the probability that the next event is a birth. Namely,
\begin{align} \label{Step2_birth}
&\frac{ \sum_{k\geq 0}b(h_{k,1}(\tilde\nu_{T^i_n}))\mfm^k(\tilde\nu_{T^i_n})}
					{ \sum_{k\geq 0}\left(b(h_{k,1}(\tilde\nu_{T^i_n}))+d(h_{k,1}(\tilde\nu_{T^i_n}))
								+\int_{\mathbb N\times \mathcal X}c(h_{k,1}(\tilde\nu_{T^i_n}),\xi)d\tilde\nu_{T^i_n}(\xi)\right)\mfm^k(\tilde\nu_{T^i_n})}
\\&\leq \nonumber
\left[b(R^K) \textstyle {\sum_{k\geq 0}}\mfm^k(\tilde\nu_{T^i_n})+\s_K hb'(R^K)\mfm^{k_1}(\tilde\nu_{T^i_n})
									+C^b_L 2A\sigma_K \lceil 3/\a\rceil \s_K\e K+O(\s_K^2 K)\right]\\\nonumber&\hspace{0.3cm}
	\times\!\Big[  \textstyle {\sum_{k\geq 0}} 
										\left(b(R^K)+d(R^K)+
												\sum_{k\geq 0}\tfrac {c(R^K,R^K)}{K}\mfm^k (\tilde\nu_{T^i_n})\right) \mfm^k (\tilde\nu_{T^i_n}) \\\nonumber&\hspace{0.6cm}
 	+\s_Kh\mfm^{k_1}(\tilde\nu_{T^i_n})
				\Big(b'(R^K)+d'(R^K)
						+\tfrac{\left(\partial_1 c(R^K,R^K)+\partial_2 c(R^K,R^K)\right)\left(\mfm^{0}(\tilde\nu_{T^i_n})+\mfm^{k_1}(\tilde\nu_{T^i_n})\right)}{K}						\Big)\\\nonumber&\hspace{7.2cm}
			-(C^{b,d,c}_L) 2A\sigma_K\lceil 3/\a\rceil \s_K\e K-O(\s_K^2 K)\:
	\Big]^{-1}.
\end{align}
For the inequality we have used the fact that, conditioned  on $T_n<\tilde{\th^K_i}$, there at most $\s_K \e \lceil3/\a\rceil$ many unsuccessful mutant individuals which differ at most $2A\s_K$ from the resident trait $R^K$.  
Since  $\textstyle {\sum_{k\geq 0}}\mfm^k(\tilde\nu_{T^i_n})=\langle\tilde \nu_{T^i_n},\mathds 1\rangle K$ which equals $\lceil \phi(i(\e/2))K\rceil+j$ conditioned on $j=Y^i_{n}$,
the right hand side of the last inequality is smaller or equals
\begin{align}\nonumber
& \left[b(R^K)+\s_K hb'(R^K)\tfrac{\mfm^{k_1}(\tilde\nu_{T^i_n})}{ \lceil \phi(i(\e/2))K\rceil+j}+O(\s_K^2)\right]
  \!\times\!\bigg[ b(R^K)+d(R^K)+c(R^K,R^K)\tfrac{\lceil \phi(i(\e/2))K\rceil+j} K\\&\hspace{0.2cm}
						+\s_K\tfrac{h\mfm^{k_1}(\tilde\nu_{T^i_n})}{\lceil \phi(i(\e/2))K\rceil+j}\Big(b'(R^K)+d'(R^K)
			+\tfrac{\left(\partial_1 c(R^K,R^K)+\partial_2 c(R^K,R^K)\right)\left(\mfm^{0}(\tilde\nu_{T^i_n})+\mfm^{k_1}(\tilde\nu_{T^i_n})\right)}{K}\Big)
		-O(\s_K^2 )\bigg]^{-1}
		\end{align}
and by definition of $\phi$ the denominator equals
\begin{align}	
 &2b(R^K)\s_K+2\s_K hb'(R^K)\tfrac{\mfm^{k_1}(\tilde\nu_{T^i_n})}{\lceil \phi(i(\e/2))K\rceil+j}+c(R^K,R^K)\tfrac{j}{K}-O(\s_K^2 )
		\\&\nonumber+\s_K h\Big[	i(\e/2)\left(\tfrac {r'(R^K)}{\overline z (R^K)}+\partial_1c(R^K,R^K)+\partial_2c(R^K,R^K)\right)+\tfrac{\mfm^{k_1}(\tilde\nu_{T^i_n})}{\lceil \phi(i(\e/2))K\rceil+j}
			\\&\nonumber
			\qquad\qquad\times\Big(-b'(R^K)+d'(R^K)
				+\tfrac{\left(\partial_1 c(R^K,R^K)+\partial_2 c(R^K,R^K)\right)\left(\mfm^{0}(\tilde\nu_{T^i_n})+\mfm^{k_1}(\tilde\nu_{T^i_n})\right)}{K}\Big)
			\Big].	
\end{align}
Thus, we obtain that the right hand side of (\ref{Step2_birth}) is bounded from above by
\begin{align}
&\frac 1 2 -\tfrac{c(R^K,R^K)}{3b(R^K)}j K^{-1}-  \tfrac{\s_Kh }{4b(R^K)} \bigg[i(\e/2)\left(\tfrac{r'(R^K)}{\overline z(R^K)}-{\partial_1 c(R^K,R^K)-\partial_2 c(R^K,R^K)} \right)
\\&\nonumber\quad +\tfrac{\mfm^{k_1}(\tilde\nu_{T^i_n})}{\lceil \phi(i(\e/2))K\rceil+j}\Big(-r'(R^K)
+\tfrac{\left(\partial_1 c(R^K,R^K)+\partial_2 c(R^K,R^K)\right)\left(\mfm^{0}(\tilde\nu_{T^i_n})+\mfm^{k_1}(\tilde\nu_{T^i_n})\right)}{K}\Big)\bigg]+O(\s_K^2).
\end{align}	
In the case where $\langle\tilde \nu_t,\mathds 1\rangle K<\lceil \phi(i(\e/2))K\rceil$ at time $t=T^i_n$, we obtain the same
inequality but with an opposite sign in front of the third term. Since
\begin{align}\nonumber
		&\Big|\tfrac{ i\e} 2 \tfrac{r'(R^K)}{\overline z(R^K)}-\tfrac{\mfm^{k_1}(\tilde\nu_{T^i_n})}{K}\tfrac{r'(R^K) K}{\lceil \phi((\e/2))K\rceil\pm j}
			-\left( \partial_1 c(R^K,R^K)+\partial_2 c(R^K,R^K)\right)\left(\tfrac{i \e} 2-\tfrac{\mfm^{k_1}(\tilde\nu_{T^i_n})}{K}\right)\Big|\\&
			< (\e/2) \:\left|\tfrac{r'(R^K)}{\overline z(R^K)}-\partial_1 c(R^K,R^K)-\partial_2 c(R^K,R^K)\right|,
\end{align}	
we deduce the claim.
Since we choose $M$ such that $M\geq 3 C_{\text{derivative}}^{b,d,c}$, we can construct a Markov chain $Z^{i}_n$ such that $Z^{i}_n\geq Y^{i}_n$, a.s., for all $n$ 
such that $T^{i}_{n}<\tilde\th^K_i\wedge \inf\{t\geq \th^K_{\text{mut. size }i (\e/ 2)}:|\langle\tilde \nu_t,\mathds 1\rangle-\phi(i(\e/2))|<\tfrac 1 3 M\e\s_K\}$ and
 the  marginal distribution of $Z_n$  is a Markov chain with $Z^i_0=Y^i_0$ and transition probabilities \vspace{-1em}
\be \mathbb P\bigl[Z^i _{n+1} =j_2|Z^i_n=j_1\bigr]
			=\begin{cases} p_+^K &\text{ for }j_1\geq1 \text{ and }  j_2=j_1+1, \\
				1-p_+^K& \text{ for }j_1\geq1 \text{ and } j_2=j_1+1,\\
					0 &\text{ else.}
				\end{cases}
\ee
Let $C_{\text{exit}} =\sup_{x\in\mathcal X} 2A \big|\tfrac{ r'(x)}{ r(x)}-\tfrac{(\partial_1 c(x,x)+\partial_2 c(x,x))}{c(x,x)}\big|$. Then,by applying  Proposition \ref{prop2.2} (b), we obtain, 
for all $a \leq  (M+C_{\text{exit}}) \e\s_K K$ and large $K$ large enough,
\bea\label{Prob_for_Claim}
&&\mathbb P_a \Big[ \inf\{n\geq 0:  Z^{i}_n \geq  2(M+C_{\text{exit}}) \e \s_K K\}<\inf\{n\geq 0:  Z^i_n \leq  (M/3) \e \s_K K\} \Big]\\&&\qquad\leq \exp \lb - K^{\a}\rb.
\nonumber
\eea
Next define $B^i\equiv \inf\{n\geq 0:Z^i_n\leq \tfrac 1 3 M \e \s_K K \}$. This is the random variable, which counts the number of jumps $Z^i$ makes until it is  smaller than $\e \s_K K$.
Note that $(T^i_{n+1}-T_{n}^i)$, the times between two jumps of $X^i_t$, 
are exponential distributed with a parameter $(b(R^K)+d(R^K)+c(R^K,R^K)\overline z(R^K))\overline z(R^K)K+O(\s_K K)$, if $T^i_{n+1}$ is smaller than $\tilde \th^K_i$. Thus, 
\be
(T^i_{l+1}-T_{l}^i)\preccurlyeq E^i_l,
\ee 
where $(E^i_l)_{l\geq 0}$ are i.i.d.\ exponential random variables with parameter
 $\inf_{x\in \mathcal X} b(x)\bar z(x) K$.   
Therefore,
\bea\label{prob_Step2.1}
&&\mathbb P\Big[
				\th_{\text{near } \phi(i\frac \e 2)}^K >\th^K_{\text{mut. size }i (\e/ 2)}+S_K\wedge \tilde\th^K_i    \Big]\\\nonumber
		&&\quad\leq\mathbb P\bigg[\:\sum_{l=0}^{B^i}E_{l}^i>S_K\bigg] 
					+\mathbb P\left[\tilde\th^K_i<\th^K_{\text{mut. size }i (\e/ 2)}+S_K\wedge \th_{\text{near } \phi(i\frac \e 2)}^K\right].
\eea
Our next goal is to find a number, $n_i$, such that $\mathbb P[B^i>n_i ]$ is $o(\s_K)$. Since the transition probabilities of $Z^i$ do not depend on the present  state, we have that
$Z_n^{i}-Z^{i}_{0}$ has the same law as $\sum_{k=1}^{n} V^i_k$, where
 $(V^i_k)_{ k\in\mathbb N}$ is a sequence of i.i.d. random variables with
\begin{align}
\mathbb P[V^i_k=1]=p^K_+\quad \text{ and }\quad\mathbb P[V^i_k=-1]=1-p^K_+
\end{align} 
and 
$\mathbb E\bigl[V^i_k\bigr]=-2\e\s_K$ and $|V^i_k|=1$.
Furthermore, we get
\begin{align}
\mathbb P\left[B^{i}\leq n_i\right] 
&\geq\mathbb P\left[\inf\left\{j\geq 0: Z_{j}-Z_{0}\leq -\lceil(\tfrac 3 2 M+C_{\text{exit}})\e\s_K K\rceil \right\}\leq n_i\right] \\
&\geq \mathbb P\left[\sum_{k=1}^{n_i}V^i_k\leq-\lceil(\tfrac 3 2 M+C_{\text{exit}})\e\s_K K\rceil\right]\nonumber
\end{align}
and by applying the\\[0.5em]
\textbf{Hoeffding's Inequality: }(Appendix 2 in \cite{P_CSP}):
\textit{Let $Y_1,\ldots,Y_n$ be independent random variables, $x>0$ and $a_j\leq Y_j-\mathbb E[Y_j]\leq b_j$ are bounded for all $j$. Then} 
\begin{align}\mathbb P\biggl[\:\sum_{j=1}^n Y_j-\mathbb E[Y_j]\geq x\biggr]\leq \exp\Bigl(-2x^2\Bigl(\sum_{j=1}^n (a_j-b_j)^2\Bigr)^{-1}\Bigr).
\end{align}
we obatin
\be
 \mathbb P\left[
 \sum_{k=1}^{n_i} V^i_k	\geq -2\e\s_K n_i + (n_i)^{\nicefrac 1 2+{\a/2}}\right ]
				\leq 2\exp(- (n_i){}^{\a}).
\ee
With $n_i \equiv \lceil K(\tfrac 3 2 M+C_{\text{exit}})\rceil$, we get $ -2\e\s_K n_i + (n_i)^{\nicefrac 1 2+{\a/2}}\leq -\lceil(\tfrac 3 2 M+C_{\text{exit}})\e\s_K K\rceil$, since $K^{-\frac 1 2+\a}\ll\s_K$.
Applying the exponential Chebychev inequality  (with $\l=K^{\a}$)
\begin{align}
 \mathbb P &\Bigg[\sum_{l=0}^{\lceil K(\frac 3 2 M+C_{\text{exit}})\rceil}E_{l}^i>S_K\Bigg]
\leq\exp(-\l S_K)\mathbb E \left[\exp\left(\l \sum_{l=0}^{\lceil K(\frac 3 2 M+C_{\text{exit}})\rceil}E_{l}^i\right)\right]\\
&\qquad\qquad\nonumber\leq\exp(-\l S_K)\Bigg(\frac{\inf_{x\in \mathcal X}{ b(x)\bar z(x)}K}{\inf_{x\in \mathcal X}{ b(x)\bar z(x)}K-\l}\Bigg)^{\lceil K(\frac 3 2 M+C_{\text{exit}})\rceil+1}\\
&\qquad\qquad\nonumber\leq\exp\left(-\l S_K+(\lceil K(\tfrac 3 2 M+C_{\text{exit}})\rceil +1)\ln\left(1+\tfrac{\l}{\inf_{x\in \mathcal X} { b(x)\bar z(x)}K-\l }\right)\right)\\
&\qquad\qquad\nonumber\leq\exp\Bigg(-\l S_K+\l \frac{\tfrac 3 2 M+C_{\text{exit}}+1}{\inf_{x\in \mathcal X}{ b(x)\bar z(x)}}+O(\l^2 K^{-1})\Bigg)
\leq\exp\left(-K^{\a}\right). 
\end{align}
Hence, the left hand side of (\ref{prob_Step2.1}) is bounded from above by
 \be
				\exp\left(-K^{\a}\right)+ 2\exp(- ( K(\tfrac 3 2 M+C_{\text{exit}}))^{\a})
+\mathbb P\left[\tilde\th^K_i<(\th^K_{\text{mut. size }i (\e/ 2)}+S_K)\wedge  \th_{\text{near } \phi(i\frac \e 2)}^K\right].
\ee
This proves the lemma, if we can show that 
\be
\mathbb P\big[\tilde\th^K_i<(\th^K_{\text{mut. size }i (\e/ 2)}+S_K)\wedge  \th_{\text{near } \phi(i\frac \e 2)}^K\big]=o(\s_K).
\ee
According to Remark \ref{Remark_2Phase} and Lemma \ref{Step1}, we have that 
\be
\mathbb P\big[ \th^K_{2\text{ succ. mut.}}\wedge\th^{K}_{ \text{diversity}}<\th^K_{\text{mut. size }i (\e/ 2)}+S_K\big]=o(\s_K).
\ee
Therefore, the following proof of (b) for $i$ implies (a) for $i$.
\end{proof}
%
%%%%%%%%%%%%%%%%%%%%%%%%%%%%%%%%%%%%%%%%%%%%%%%%%%%%%%%%%%%%%%%%%%%%%%%%
%
\begin{proof}\textit{of (b) for $i$ by assuming that the lemma holds for $i-1$.}
Note that the random elements $B^i,\:T^i,\:V^i,\:W^i,\: X^i,\:Y^i$ and $Z^i$ are not the ones of the last proof. 
They will be defined during this proof. In fact, the structure of the proof is similar to the one of (a), except that we prove a lower bound for the time of a change of oder $\e$ for the mutant density instead of upper bound for the time of a change of oder $\e\s_K$ of the total mass.
We couple $\mfm_t^{k_1}$, for $t\geq\th^K_{\text{mut. size }i (\e/ 2)}$, with a discrete time Markov chain (depending on $i$). Therefore,
let $T^i_0 =\th^K_{\text{mut. size }i (\e/ 2)}$ and $(T^i_k)^{}_{k\geq 1}$ be the sequences of  jump times of $\mfm_t^{k_1}$ after $\th^K_{\text{mut. size }i (\e/ 2)}$.
Furthermore, let $(Y_n^i)^{}_{n\geq 0}$ be the discrete time process which records the values that $\mfm_t^{k_1}$ takes i.e. 
$Y^i_0 = \mfm^{k_1}(\tilde \nu_{T^i_0}) = \lceil K i (\e/ 2)\rceil $ and $Y^i_n =\mfm^{k_1}(\tilde\nu_{T_n^i}).$
Observe that if 
\be\tilde\th^K_i> \th_{\text{near } \phi(i\frac \e 2)}^K\wedge \inf\{t\geq \th^K_{\text{mut. size }i (\e/ 2)}:|\langle\tilde \nu_t,\mathds 1\rangle-\phi(i(\e/2))|\geq 2(M+C_{\text{exit}})\e\s_KK\},
\ee
we know from the inequality (\ref{Prob_for_Claim})  that  the probability that
$\th_{\text{near } \phi(i\frac \e 2)}^K$ is larger than $\inf\{t\geq \th^K_{\text{mut. size }i (\e/ 2)}:
 |\langle\tilde \nu_t,\mathds 1\rangle -\phi(i(\e/2))|\geq 2(M+C_{\text{exit}})\e\s_KK\}$  is smaller than $\exp(-K^{\a})$. 
Define
 \bea
\hat \th^K_i&\equiv&  \inf\{t\geq \th^K_{\text{mut. size }i (\e/ 2)}:|\langle\tilde \nu_t,\mathds 1\rangle-\phi(i(\e/2))|\geq 2(M+C_{\text{exit}})\e\s_KK\}
		 				\\&&\wedge\:\th_{\text{near } \phi(i\frac \e 2)}^K\wedge\th^K_{2\text{ succ. mut.}} \wedge\th^{K}_{ \text{diversity}}	\nonumber
 \eea
and 
$
  \tilde C_{\text{fitness} }\equiv\inf_{x\in\mathcal X}\partial_1f(x,x)/\:\overline b.$
Then, for all  $-\lceil \frac \e 4 K\rceil \leq  j \leq  \lceil\frac \e 4 K\rceil,$ for K large enough and
   for $\e$ small enough,  we have that
\begin{align}\label{cond_pro_Step2(b)}
	\mathbb P \left[Y^i_{n+1} = \lceil i\tfrac \e 2K\rceil + j+1\big|Y^i_n =\lceil i\tfrac \e 2K\rceil +j,T^i_{n+1} < \hat \th^K_i \right]\qquad\qquad&\\
				\in\left[\tfrac 1 2 + \tfrac 1 2  \tilde C_{\text{fitness} }\:\s_K ,\: \tfrac 1 2 + 2A    \tilde C_{\text{fitness} }\: \s_K \right]&,\nonumber
\end{align}
since the left hand side of (\ref{cond_pro_Step2(b)})  is equal to the  expectation of the probability that the next event is a birth without mutation conditioned on $\mathcal F_{T^i_n}$. Namely,
\begin{align} \label{abcder}
	&\frac{ b(R^K+\s_K h)(1-u_K m(R^K-\s_K h))}
			{ \left(b(R^K+\s_K h)+d(R^K+\s_K h)+\int_{\mathbb N\times \mathcal X}c(R^K+\s_K h,\xi)d\tilde\nu_{T_n}(\xi)\right)}\\[0.1em]&\nonumber
		=b(R^K+\s_K h)\Big[ b(R^K+\s_K h)+d(R^K+\s_K h)+c(R^K+\s_K h,R^K)\Big(\phi (i\tfrac \e 2)-\tfrac {\lceil i\tfrac \e 2K\rceil +j} K\Big)\\&\nonumber\hspace{0.3cm}
					+c(R^K+\s_K h,R^K+h\s_K)\Big(\tfrac {\lceil i\tfrac \e 2K\rceil +j} K\Big)+\xi_1(\e\s_KC^{c}_L (\lceil \tfrac 3 \a\rceil +2(M+C_{\text{exit}})) )\Big]^{-1}+O(u_K)
					\\[0.1em]&\nonumber					
		= b(R^K\!+\s_K h)
				\Big[ 2b(R^K\!+\s_K h)-f(R^K\!+\s_K h,R^K)
					+c(R^K\!+\s_K h,R^K)\Big(\phi (i\tfrac \e 2)-\tfrac{r(R^K)}{c(R^K,R^K)}\Big)\\&\nonumber\hspace{1.5cm}	
					+\s_K h \partial_2 c(R^K,R^K)\Big(\tfrac {\lceil i\tfrac \e 2K\rceil +j} K\Big)+\xi_1\big(\e\s_KC^{c}_L \big(\lceil \tfrac 3 \a\rceil +2(M+C_{\text{exit}})\big) \big)\Big]^{-1}
								+O(u_K).
\end{align}
for some $\xi_1\in(-1,1)$. By definition of $\phi$ of (\ref{abcder}) is equal to
\begin{align}								
		&b(R^K+\s_K h)
				\Big[ 2b(R^K+\s_K h)-\partial_1 f(R^K,R^K)\s_K h+c(R^K+\s_K h,R^K)\s_Kh \\\nonumber&\hspace{3cm}
				\times(i\tfrac \e 2)\Big(\tfrac{r'(R^K)}{r(R^K)}-\tfrac{\partial_1c(R^K,R^K)}{c(R^K,R^K)}\Big)
									+\xi_1\left(\e\s_KC^{c}_L \left(\lceil \tfrac 3 \a\rceil +2(M+C_{\text{exit}})\right) \right)\Big]^{-1}	\\&\hspace{12cm}\nonumber
									    		+O(\tfrac{\s_K  j} K+\s_K^2+u_K)\\[0.1em]&	\nonumber
		= b(R^K+\s_K h)
				\Big[ 2b(R^K+\s_K h)-\s_Kh\left(1-i\tfrac \e 2\tfrac {c(R^K,R^K)}{r(R^K)} \right)	\partial_1f(R^K,R^K)			\\&\nonumber\hspace{4.5cm}
					+\xi_1(\e\s_KC^{c}_L (\lceil \tfrac 3 \a\rceil +2(M+C_{\text{exit}})) )\Big]^{-1}+O(\tfrac{\s_K  j} K+\s_K^2+u_K) \\[0.1em]&	\nonumber
		=\frac{1}{2} +\s_Kh \left(1-i\tfrac \e 2\tfrac {c(R^K,R^K)}{r(R^K)}\right )\tfrac{\partial_1f(R^K,R^K)}{b(R^K)}
					+\e\s_K\xi_1\tfrac{(C^{c}_L (\lceil \frac 3 \a\rceil +2(M+C_{\text{exit}}))}{b(R^K)} 
					+O(\tfrac{\s_K  j} K+\s_K^2+u_K).
\end{align}	
Then, because $i<2\e^{-1}C^{\e}_{\text{cross}}$ implies that $1-i\tfrac \e 2\tfrac {c(R^K,R^K)}{r(R^K)}>0$, 
we obtain (\ref{cond_pro_Step2(b)}).  Thus we can construct a Markov chain $Z^i_n$ such that $Z^i_n\geq Y^i_n$ a.s. for all $n$ 
such that $T^i_{n}<\hat \th^K$ and the marginal distribution of $Z^i_n$  is a Markov chain with transition probabilities 
\be\Eq(cross.10)
\mathbb P\bigl[Z^i_{n+1} =j_2|Z^i_n=j_1\bigr]
=\begin{cases} 
\frac 1 2 + 2A \tilde C_{\text{fitness}} \s_K&\text{ for } j_2=j_1+1, \\
\frac 1 2 -  2A  \tilde C_{\text{fitness}} \s_K& \text{ for }j_2=j_1-1,\\
0 &\text{ else.}
\end{cases}
\ee
We define a continuous time process, $\tilde Z^i$, associate to $Z^i_n$. To do this, we define first $(\tilde T^i_j)_{j\in \mathbb N}$, the sequence of jump times, by $\tilde T^i_0=0$ and
\be\Eq(cross.11)  
\tilde T^i_j-\tilde T^i_{j-1} =
\begin{cases} 
		T^i_j - T^i_{j-1} & \text{ if } T^i_j <\tilde \th^K,\\
			W^i_j & \text{ else} ,
\end{cases}  
\ee
where $W^i_j$ are exponential random variables with mean 
$(\lceil K (i+\tfrac 1 2) (\e/ 2)\rceil(\overline b+\overline d +\overline c(4\overline b/\underline c)) )^{-1}$.
We set  $\tilde Z^i_t =Z^i_n $  if $t\in [\tilde T^i_n,\tilde T^i_{n+1})$. Obverse that we obtain by construction 
$\tilde Z^i_t \geq \mfm^{k_1}(\tilde \nu_{\th^K_{\text{mut. size }i (\e/ 2)}+t})$, for all $t$ such that 
$ \th^K_{\text{mut. size }i (\e/ 2)}+t\leq\hat \th^K_i$.
Next we want to show that
\be\label{Step2.2.0}
\mathbb P\left[ \inf\left\{t\geq 0:  \tilde Z^i_t \geq \lceil K (i+\tfrac 1 2) (\e/ 2)\rceil \right\}> S_K\right]=1-o(\s_K).
\ee
Therefore, let $B^Z_i=\inf\{n\geq 0:Z^i_n= \lceil K (i+\tfrac 1 2) (\e/ 2)\rceil\}$. We can construct  $(X^i_j)_{ j \geq 1}$ a
sequence of independent, exponential random variables with parameter $x^K_i\equiv \lceil K (i+\tfrac 1 2) (\e/ 2)\rceil(\overline b+\overline
d +\overline c(4\overline b/\underline c))$ such that
\be\Eq(cross.12)
(\tilde T^i_{j+1}-\tilde T^i_{j}) \succcurlyeq X^i _j\quad \text{ for all }1\leq j\leq B^Z_i.
\ee
Our next goal is to find a barrier, $n_i$, such that 
$B^Z_i$ is smaller than $n_i$ only with very small probability. 
Since the transition probabilities of $Z^i$ do not depend on the present  state, 
$Z^i_{B^Z_i}-Z_{0}$ is  stochastically equivalent to  $\sum_{k=1}^{j} V^i_k$, where $(V^i_k)_{ k\in\mathbb N}$ are
 i.i.d. random variables taking values $\pm 1$ with probabilities
\be\Eq(cross.13)
			\mathbb P[V^i_k=1]=\tfrac 1 2 + 2A \tilde C_{\text{fitness}} \s_K \quad
		 \text{ and }		\quad\mathbb P[V^i_k=-1]=\tfrac 1 2 -  2A \tilde C_{\text{fitness}} \s_K.
\ee
Note that 
$\mathbb E\bigl[V^i_k\bigr]= 4A \tilde C_{\text{fitness}} \s_K$ and $|V^i_k|=1$.
Furthermore, we get 
\be\Eq(cross.14)
\mathbb P\left[B^Z_{i}\leq n_i\right] 
=\mathbb P\left[\exists {\lceil(\e/4) K\rceil\leq j\leq n_i} : \sum_{k=1}^j V^i_k\geq \lceil(\e/4)K\rceil \right ].
\ee
Hoeffding's inequality implies that, for  $j\geq\lceil(\e/4) K\rceil$,
\be\Eq(cross.15)
 	\mathbb P\left[ \sum_{k=1}^{j} V^i_k	\geq  4A \tilde C_{\text{fitness}} \s_K j + j^{\nicefrac 1 2+
	{\a/2}}\right ]
\leq 2\exp(- j {}^{\a}).
\ee
We take $n_i \equiv \e K(8 A\tilde C_{\text{fitness}}\s_K)^{-1}$ and get for all $\lceil(\e/4)K\rceil\leq j\leq n_i$,
\be\Eq(cross.16)
4A \tilde C_{\text{fitness}} \s_K  j + j^{\nicefrac 1 2+{\a/2}}\leq \lceil(\e/4) K\rceil ,
 \ee
since $K^{-\frac 1 2+\a}\ll\s_K$.
%%%%
Then, the probability that 
$B^Z_{i}\leq \e K(8 A\tilde C_{\text{fitness}}\s_K)^{-1} $ is bounded from above by $2\exp(-K^{\a})$.
Therefore, the left  hand side of equation (\ref{Step2.2.0}) is larger than
\begin{align}
\mathbb  P\left[\textstyle \sum_{j=1}^{\e K(8 A\tilde C_{\text{fitness}}\s_K)^{-1} }X^i_j> S_K\right]-2\exp(-K^{\a}),
\end{align}
%and since $(X^i_j)_{j\geq 0}$ are exponential  distributed with parameter
%$x^K_i=(\lceil K (i+\tfrac 1 2) (\e/ 2)\rceil(\overline b+\overline d +\overline c(4\overline b/\underline c)))$,
%we have 
%\be
% 		\e K(8 A\tilde C_{\text{fitness}}\s_K)^{-1} \mathbb E[X^i_1]\gg S_K.
%\ee
By applying the exponential  Chebychev inequality we get, similarly as in (a),
\begin{align}\label{Cheby_for_(b)}
	\:\mathbb P\Bigg[\sum_{j=1}^{\e K(8 A\tilde C_{\text{fitness}}\s_K)^{-1}}\!\!&X^i_j\leq S_K \Bigg]
	\:=\:\mathbb P\Bigg[- \sum_{j=1}^{\e K(8 A\tilde C_{\text{fitness}}\s_K)^{-1}}\!\!X^i_j\geq-S_K\Bigg]\\[0.2em]\nonumber
	&\:\leq\: \exp(K^{\a} S_K )\mathbb E\left[\exp(-K^{\a} X^i_j)\right]^{\e K(8 A\tilde C_{\text{fitness}}\s_K)^{-1}}\\\nonumber
	&\:\leq\: \exp(K^{\a} S_K )\exp\left(\e K(8 A\tilde C_{\text{fitness}}\s_K)^{-1}\ln\left(\tfrac{x^K_i}{x^K_i+K^{\a}}\right)\right)\\\nonumber
	&\:\leq\: \exp(K^{\a} S_K-\e K(8 A\tilde C_{\text{fitness}}\s_K)^{-1} C K^{-1+\a})),\quad \text{ for some small $C>0$,}\\\nonumber
	&\:\leq \:\exp\left(-K^{\a}\right).
\end{align}
This proves that $\mathbb P\big[ \inf\{t\geq 0:  \tilde Z^i_t \geq \lceil K (i+\tfrac 1 2) (\e/ 2)\rceil \}> S_K\big]\geq 1-3\exp\left(-K^{\a}
			\right)$, and therefore (b) and  (a) for $i$, provided that the lemma holds for $i-1$. 			
\end{proof}
%
%%%%%%%%%%%%%%%%%%%%%%%%%%%%%%%%%%%%%%%%%%%%%%%%%%%%%%%%%%%%%%%%%%%%%%%%%%%%%%%%%%
%
\begin{proof}
\textit{of (c) for $i$ by assuming that the lemma holds for $i-1$.}
Note that the random elements $\:T^i,\: X^i$ and $Y^i$are not the ones of the last proof. 
As in (a) we 
 couple $K\langle\tilde \nu_{t},\mathds 1\rangle$ with a discrete time Markov Chain. 
Therefore, let 
\be
					X^{i}_t=|K\langle\tilde \nu_{t},\mathds 1\rangle-\lceil \phi({i}(\e/2))K\rceil |
\ee
and   $T^{i}_0=\th^K_{\text{mut. size }i (\e/2)}$  and $(T^{i}_k)_{k\geq1}$ be the sequences of the jump times of $\langle\tilde \nu_{t},\mathds 1\rangle$ 
after $\th^K_{\text{mut. size }i(\e/2)}$. Then, let $Y^{i}_k$ be the associated discrete time process  
which records the values that $X^{i}_t$ takes after time $\th^K_{\text{mut. size }i(\e/2)}$.
\\[0.2em]
\textbf{Claim:}
 \textit{There exists a constant $\tilde C^{b,d,c}_{\text{derivative} }$ such that for all
 $ j< \lceil \e K\rceil$ and $K$ large enough, 
  \begin{align}\label{prob_Step2.3}
 \mathbb P\bigl[Y^i_{n+1}=j+1|Y^i_n=j, T_{n+1}&<\tilde \th_i^K\bigr]
 					 \leq \:\frac 1 2 -\tfrac{\underline c}{3\overline b}j K^{-1}+\e\s_K\tilde C^{b,d,c}_{\text{derivative} }=:p_+^K(j),
\end{align}
\textit{Moreover, we can choose $\tilde C^{b,d,c}_{\text{derivative} }\equiv \sup_{x\in\mathcal X}\tfrac{A}{4b(x)} \:|\tfrac{r'(x)}{\overline z(x)}-\partial_1 c(x,x)-\partial_2 c(x,x)|$}.} 
\\[0.2em]
From (a) we know that the left hand side of (\ref{prob_Step2.3})
is smaller or equals
\begin{align} 
 \frac 1 2 -\tfrac{c(R^K,R^K)}{3b(R^K)}j K^{-1}+\tfrac{\e\s_K h}{8b(R^K)} \:\left|\tfrac{r'(R^K)}{\overline z(R^K)}-\partial_1 c(R^K,R^K)-\partial_2 c(R^K,R^K)\right|+O(\s_K^2).
\end{align}	
This proves the Claim. Note that $p_+^K(j)$ depends on $j$. Since we can choose $M\geq (8\tilde C_{\text{derivative}}^{b,d,c})(\tfrac{3\overline b}{\underline c})$, 
continuing as in Lemma \ref{exit_from_domain} implies that (c) is true for $i$,
provided that the lemma holds for $i-1$. 
\end{proof}
%
%%%%%%%%%%%%%%%%%%%%%%%%%%%%%%%%%%%%%%%%%%%%%%%%%%%%%%%%%%%%%%%%%%%%%%%%%%%%%%%%%
\begin{proof}
\textit{of (d) for $i$ by assuming that the lemma holds for $i-1$.}
Again we couple $\mfm_t^{k_1}$, for $t\geq\th_{\text{near } \phi(i\frac \e 2)}^K$, with a discrete time Markov chain. Therefore, let $T^i_0 =\th_{\text{near } \phi(i\frac \e 2)}^K$
and $(T^i_k)^{}_{k\geq 1}$ be the sequences of the jump times of $\mfm_t^{k_1}$ after $\th_{\text{near } \phi(i\frac \e 2)}^K$.
Then, let $(Y_n^i)^{}_{n\geq 0}$ be the 
discrete time process which records the values that $\mfm_t^{k_1}$, i.e. 
\be
Y^i_0 = \mfm^{k_1}(\tilde \nu_{T^i_0}) \in [K (\frac {i\e}{ 2}-\frac{\e} 4)-1,  K (\frac {i\e}{ 2}+\frac{\e} 4)+1],
\ee
 and $Y^i_n =\mfm^{k_1}(\tilde\nu_{T_n^i})$.
Define
 \be
 \hat \th^K_i\equiv  \inf\left \{t\geq \th_{\text{near } \phi(i\frac \e 2)}^K: |\langle \tilde \nu_t,\mathds 1\rangle-\phi(i(\e/2))|> M\e\s_K\right\}\wedge\th^K_{2\text{ succ. mut.}} \wedge\th^{K}_{ \text{diversity}}	.
 \ee
 Note that this $\hat \th^K_i$ differs only a bit from the one defined in (b).
From the proof of  (b), we know that the density of  the mutant trait has the tendency to increase. 
More precisely, since $i\leq C^{\e}_{\text{cross}} (2/\e)$, we have,  for all  $-\lceil \frac \e 4 K\rceil \leq  j \leq  \lceil\frac \e 2 K\rceil,$ for K large enough and
   for $\e$ small enough,
\begin{align}\label{cond_pro_Step2}
\mathbb P \Big[Y^i_{n+1} = \lceil i\tfrac \e 2K\rceil + j+1\Big|&Y^i_n =\lceil i\tfrac \e 2K\rceil +j,T^i_{n+1} < \hat \th^K_i \Big]\geq \tfrac 1 2 + \s_K\tfrac{\inf_{x\in\mathcal X}\partial_1f(x,x)}{2\overline b}\end{align}
By Continuing in a similar way as in (b) with bounding the random variables in the  in the other direction (as in (a)),  implies that (d) is true for $i$,
provided that the lemma holds for $i-1$.  			
\end{proof}\vspace{-0.5em}

\end{proof}
%In fact, the arguments of this proof allow us  to approximate the mutant density until it reaches $\bar z(R^K)-\delta (\epsilon)$, where $\epsilon\ll\delta (\epsilon)\ll0$  as $\epsilon\to0$. This implies that the resident density deceased to $\delta (\epsilon)+O(\s_K)$ in the same time. Thus it may be possible to skip the Step 3, but then one has to be more careful in Step 4, since the total mass will only stay in a $M\delta(\epsilon)\s_K$ neighborhood of the new equilibrium.   
%
%
%%%%%%%%%%%%%%%%%%%%%%%%%%%%%%%%%%%%%%%%%%%%%%%%%%%%%%%%%%%%%%%%%%%%%%%%%%%%
%
\subsection{Step 3}
Similarly as in Step 2 we define a function which allows us to approximate the total mass of the population for a given 
density of the resident trait. 
\begin{notations}
Let us define
\begin{align}
		\psi(x)\equiv \overline z(R^K) &+\s_K h (\bar z(R^K)-x) \left(\tfrac{ r'(R^K)}{r(R^K)} 
					+\tfrac{\partial_1 c(R^K,R^K)+\partial_2 c(R^K,R^K)}{c(R^K,R^K)}
                                        \right).\label{psi.1}
\end{align}
\end{notations}%
Note that $\phi(y)=\psi\left(\phi(y)-y\right)+O(\s_K^2)$. Therefore and since  $ |\langle \tilde \nu_{\th^K_{\text{mut. size }C^{\e}_{\text{cross}}}},\mathds 1\rangle-\phi(C^{\e}_{\text{cross}})|< M\e\s_K$  
with probability $1-o(\s_K)$, we get
 that at time $\th^K_{\text{mut. size }C^{\e}_{\text{cross}}}$ the  density of the resident population  belongs to an interval centered at $\phi(C^{\e}_{\text{cross}})-C^{\e}_{\text{cross}}$ with diameter 
 $2(M+\lceil 3/\a\rceil)\e\s_K$ with probability $1-o(\s_K)$ and hence 
\begin{align}
		\psi(\mfm^0(\tilde \nu_{\th^K_{\text{mut. size }C^{\e}_{\text{cross}}}})K^{-1})
						=\psi(\phi(C^{\e}_{\text{cross}})-C^{\e}_{\text{cross}}) +O(\e\s_K^2)=\phi(C^{\e}_{\text{cross}})+O(\s_K^2)
\end{align}
with probability $1-o(\s_K)$.
Thus, the total mass of the population also belongs to an interval centered at $\psi(\phi(C^{\e}_{\text{cross}})-C^{\e}_{\text{cross}})$ with diameter 
$2(M\e\s_K+O(\s_K^{2}))<2(M+1)\e\s_K$.
%
%%%%%%%%%%%%%%%%%%%%%%%%%%%%%%%%%%%%%%%%%%%%%%%%%%%%%%%%%%%%%%%%%%%%%%%%%%%%
\begin{notations}
Let us define
\bea
	\tilde C^{K}_{\text{cross}}&\equiv& \lceil(\phi(C^{\e}_{\text{cross}})-C^{\e}_{\text{cross}}-\e)2/\e\rceil(\e/2) \qquad \qquad \text{ and }\\[0.5em]
 	\th_{\text{near } \psi(\tilde C^{\e}_{\text{cross}}-\frac \e 2)} &\equiv&  
		 \inf\{t\geq \th^K_{\text{mut. size } C^{\e}_{\text{cross}}}: 
		 	|\langle\tilde \nu_t,\mathds 1\rangle-\psi(\tilde C^{\e}_{\text{cross}}-\tfrac \e 2)|<(M/3)\e\s_K\}.\quad
\eea
Note that the term $-\e$ in the definition of  $\tilde C^{K}_{\text{cross}}$ ensures that resident population is larger than $\tilde C^{K}_{\text{cross}}$ at time $\th^K_{\text{mut. size }C^{\e}_{\text{cross}}}$.
\end{notations}
First, we need a lemma for the interface between Step 2 and Step 3.
%
%
%%%%%%%%%%%%%%%%%%%%%%%%%%%%%%%%%%%%%%%%%%%%%%%%%%%%%%%%%%%%%%%%%%%%%%%%%%%%
\begin{lemma}\label{interface}
Fix $\e>0.$ Suppose that the assumptions of Theorem  \ref{Thm_2.Phase} hold. Then, there exists a constant $M>0$ (independent of $\e$ and $K$) such that,
\begin{enumerate}[(a)]
		\setlength{\itemsep}{3pt}
\item Soon after $\th^K_{\text{mut. size } C^{\e}_{\text{cross}}}$, the total population size is close to $\psi(\tilde C^{\e}_{\text{cross}}-\frac \e 2)$:
\begin{align}
		\lim_{K\to\infty}\:\s_K^{-1}\:
				\mathbb P&\Big[ \th_{\text{near } \psi(\tilde C^{\e}_{\text{cross}}-\frac \e 2)}
								>\th^K_{\text{mut. size } C^{\e}_{\text{cross}}}+S_K	\wedge\th^K_{2\text{ succ. mut.}} \wedge\th^{K}_{ \text{diversity}} \\&\quad									\wedge\inf\left\{t\geq \th^K_{\text{mut. size } C^{\e}_{\text{cross}}}: 
															\mfm^{0}(\tilde\nu_t) =\lceil (\tilde C^{\e}_{\text{cross}}\pm 3\e/4) K\rceil\right\}	\Big]	=	0.\nonumber
\end{align}	
\item A change of order $\e $ for the resident density takes more than $o(\s_K^{-1})$ time: 	
\begin{align}																
 \lim_{K\to\infty}\:\s_K^{-1}\:
			 \mathbb P&\Big[\inf\left\{t\geq \th^K_{\text{mut. size } C^{\e}_{\text{cross}}}:\mfm^{0}(\tilde\nu_t)  =\lceil (\tilde C^{\e}_{\text{cross}}\pm 3\e/4) K\rceil\right\}\\&\quad
										<\th^K_{\text{mut. size } C^{\e}_{\text{cross}}}+S_K \wedge \th_{\text{near } \psi(\tilde C^{\e}_{\text{cross}}-\frac \e 2)}
														\wedge\th^K_{2\text{ succ. mut.}} 
															\wedge\th^{K}_{ \text{diversity}}		\Big]		=	0.	\nonumber													
\end{align}	
\item At the time when the resident density has changed of order $\e$ the total population size is still close to $\psi(\tilde C^{\e}_{\text{cross}}-\frac \e 2)$: 	
\begin{align}		
		\lim_{K\to\infty}	\s_K^{-1}\;
			\mathbb P&\Big[\inf\left \{t\geq \th_{\text{near } \psi(\tilde C^{\e}_{\text{cross}}-\frac \e 2)}^K\!:	
														|\langle \tilde \nu_t,\mathds 1\rangle-\psi(\tilde C^{\e}_{\text{cross}}-\tfrac \e 2)|\!>\! M\e\s_K\right\}
	<\th^K_{2\text{ succ. mut.}} \\\nonumber
	&\quad\wedge\th^{K}_{ \text{diversity}}
	\wedge \inf\left\{t\geq \th^K_{\text{mut. size } C^{\e}_{\text{cross}}}: 
					\mfm^{0}(\tilde\nu_t) =\lceil (\tilde C^{\e}_{\text{cross}}\pm \e )K\rceil\right\}	\Big]	=0.
\end{align}	
\item A change of order $\e $ for the resident density takes no more than $(i \s_K)^{-1-\a/2}$ time: 	
\begin{align}
		\lim_{K\to\infty}	\s_K^{-1}\;
			\mathbb P&\Big[\:\th^K_{\text{res. size }\tilde C^{\e}_{\text{cross}}-\e}
								>\th^K_{\text{mut. size } C^{\e}_{\text{cross}}}+(i \s_K)^{-1-\a/2}\big)\wedge\th^K_{2\text{ succ. mut.}} \wedge\th^{K}_{ \text{diversity}}\\
										&\quad\wedge\inf\left \{t\geq \th_{\text{near } \psi(\tilde C^{\e}_{\text{cross}}-\frac \e 2)}^K:
				 									|\langle \tilde \nu_t,\mathds 1\rangle-\psi(\tilde C^{\e}_{\text{cross}}-\tfrac \e 2)|> M\e\s_K\right\} \Big] = 0.\nonumber
\end{align}
\end{enumerate}
\end{lemma}
%
%%%%%%%%%%%%%%%%%%%%%%%%%%%%%%%%%%%%%%%%%%%%%%%%%%%%%%%%%%%%%%%%%%%%%%%%%%%%
\begin{proof} 
Apply the methods of of (a) to (d) from Lemma \ref{Step2}.  
\end{proof}
Next, we have the following similar lemmata as in Step 2, for them  let us define 
\be \th_{\text{near } \psi(i\frac \e 2)}^K\equiv\inf\{t\geq \th^K_{\text{res. size }i (\e/ 2)}:|\langle\tilde \nu_t,\mathds 1\rangle-\psi(i(\e/2))|<(M/3)\e\s_K\}.\ee
%%
%%%%%%%%%%%%%%%%%%%%%%%%%%%%%%%%%%%%%%%%%%%%%%%%%%%%%%%%%%%%%%%%%%%%%%%%%%%%
\begin{lemma}
\label{Step3}
Suppose that the assumptions of Theorem  \ref{Thm_2.Phase} hold. 
Then, there exists a constant $M>0$ (independent of $\e$, $K$ and $i$) such that, for all $\e>0$ and for all $(\tilde C^{\e}_{\text{cross}}-\e)(2/\e)\geq i\geq 2$,
\begin{enumerate}[(a)]
		\setlength{\itemsep}{3pt}
\item Soon after $\th^K_{\text{res. size }i(\e/ 2)}$, the total population size is close to $\psi(i\frac \e 2)$:
\begin{align}
\lim_{K\to\infty}\:\s_K^{-1}\:
			\mathbb P&\Big[ \th_{\text{near } \psi(i\frac \e 2)}^K
									>\th^K_{\text{res. size }i(\e/ 2)}+S_K \wedge\th^K_{2\text{ succ. mut.}} \wedge\th^{K}_{ \text{diversity}}
												\\&\quad\wedge\inf\left\{t\geq \th^K_{\text{res. size }i(\e/2)}: \mfm^{0}(\tilde\nu_t) =\lceil (i\pm \tfrac12)(\e/ 2) K\rceil\right\}\Big]=0.	\nonumber	
\end{align}
\item A change of order $\e $ for the resident density takes more than $o(\s_K^{-1})$ time: 	
\begin{align}	
	\lim_{K\to\infty}\:\s_K^{-1}\:
			 \mathbb P&\Big[\inf\left\{t\geq \th^K_{\text{res. size }i(\e/2)} : \mfm^{0}(\tilde\nu_t) =\lceil (i\pm \tfrac12)(\e/2) K\rceil\right\}\\&\quad
	<\th^K_{\text{res. size }i (\e/ 2)}+S_K \wedge \th_{\text{near } \psi(i\frac \e 2)}^K\wedge\th^K_{2\text{ succ. mut.}} 
				\wedge\th^{K}_{ \text{diversity}}		\Big]		=	0.\nonumber
\end{align}	
\item At the time when the resident density has changed of order $\e$ the total population size is still close to $\psi(i\frac \e 2)$: 	
\begin{align}	\lim_{K\to\infty}	\s_K^{-1}\;
			\mathbb P&\Big[	\inf\left \{t\geq \th_{\text{near } \psi(i\frac \e 2)}^K: |\langle \tilde \nu_t,\mathds 1\rangle-\psi(i(\e/2))|> M\e\s_K\right\}
									< \th^K_{2\text{ succ. mut.}} \\&\quad
										\wedge\th^{K}_{ \text{diversity}}\wedge \inf\left\{t\geq  \th^K_{\text{res. size }i(\e/2)}:\mfm^{0}(\tilde\nu_t) =\lceil (i\pm 1)(\e/ 2) K\rceil\right\}	\Big]	=0.\nonumber
\end{align}	
\item A change of order $\e $ for the resident density takes no more than $(i \s_K)^{-1-\a/2}$ time: 	
\begin{align}
\lim_{K\to\infty}	\s_K^{-1}\;
			\mathbb P&\Big[\:\th^K_{\text{res. size }(i-1) (\e/ 2)}
								>\big(\th_{\text{near } \psi(i\frac \e 2)}^K+(i \s_K)^{-1-\a/2}\big)\wedge\th^K_{2\text{ succ. mut.}} \wedge\th^{K}_{ \text{diversity}}
\\
&	\quad									\wedge\inf\left \{t\geq \th_{\text{near } \psi(i\frac \e 2)}^K: |\langle \tilde \nu_t,\mathds 1\rangle-\psi(i(\e/2))|> M\e\s_K\right\} \Big]
												=0.\nonumber
\end{align}
\end{enumerate}
\end{lemma}
\begin{proof} Apply the methods of of (a) to (d) from Lemma \ref{Step2}.  
\end{proof}
\begin{remark}
Lemma \ref{interface} and \ref{Step3} imply that the density of the resident trait decreases to the value  $\e$. Moreover, 
\begin{align}
			\mathbb P\left[\th^K_{\text{res. size } \e}>\th^K_{\text{mut. size }  C^{\e}_{\text{cross}}}+\ln(K)\s_K^{-1-\a/2}\wedge\th^K_{2\text{ succ. mut.}}
										 \wedge\th^{K}_{ \text{diversity}}\right]&=o(\s_K)\\
\text{and} \qquad 
			\mathbb P\left[\:|\langle \tilde \nu_{\th^K_{\text{res. size } \e}},\mathds 1\rangle-\psi(\e)|>M\e\s_K\right]&=o(\s_K).
\end{align}
\end{remark}
\subsection{Step 4} After the time $\th^K_{\text{res. size } \e}$ we have to wait less than $\ln(K)\s_K^{1+\a/2}$ time to know that the resident trait is extinct with high probability.
\begin{notations}
 Define 
 \:$\th_{\text{near } \psi(0)}^K\equiv \inf\{t\geq \th^K_{\text{res. size } \e }:|\langle\tilde \nu_t,\mathds 1\rangle-\psi(0)|<(M/3)\e\s_K\}$.
\end{notations}
\begin{lemma}
\label{Step4.1}
Suppose that the assumptions of Theorem  \ref{Thm_2.Phase} hold. Then, there exists a constant $M>0$ (independent of $\e$ and $K$) such that, for all $\e>0$ 
\begin{enumerate}[(a)]
		\setlength{\itemsep}{3pt}
\item Soon after $\th^K_{\text{res. size } \e }$, the total population size is close to $\psi(0)$:
\begin{align}
\lim_{K\to\infty}\:\s_K^{-1}\:
			&	\mathbb P\Big[ 	\th_{\text{near } \psi( 0 )}^K
									>\th^K_{\text{res. size }\e}+S_K \wedge\th^K_{2\text{ succ. mut.}} \wedge\th^{K}_{ \text{diversity}}
	\\&\qquad\wedge\inf\left\{t\geq \th^K_{\text{res. size }\e)}: \mfm^{0}(\tilde\nu_t) =\lceil (1\pm \tfrac14)\e K\rceil\right\}	
												\Big]=0.\nonumber
\end{align}												
\item A change of order $\e $ for the resident density takes more than $o(\s_K^{-1})$ time: 	
\begin{align}  \lim_{K\to\infty}\:\s_K^{-1}\: 
		&	\qquad \mathbb P\Big[\inf\left\{t\geq \th^K_{\text{res. size } \e} :\mfm^{0}(\tilde\nu_t) =\lceil (1\pm  \tfrac14) \e K\rceil\right\}\\&
										<\th^K_{\text{res. size } \e}+S_K \wedge \th_{\text{near } \psi(0)}
														\wedge\th^K_{2\text{ succ. mut.}} 
															\wedge\th^{K}_{ \text{diversity}}		\Big]		=	0.\nonumber
\end{align}
\end{enumerate}
\end{lemma} 
\begin{proof} See proof of Lemma \ref{Step2}
\end{proof}
\begin{lemma}\label{Step4.4}
Suppose that the assumptions of Theorem  \ref{Thm_2.Phase} hold. Then, there exists a constant $M>0$ (independent of $\e$ and $K$) such that, for all $\e>0$ 
\begin{align}
\lim_{K\to\infty}
\s_K^{-1}\;
&\mathbb P\Big[\:\th^K_{\text{res. size }0}
>\big(\th_{\text{near } \psi(0)}^K+\ln(K)\s_K^{-1-\a/2}\big)\wedge\th^K_{2\text{ succ. mut.}} \wedge\th^{K}_{ \text{diversity}}\\	
&\qquad\qquad\qquad\wedge\inf\left \{t\geq \th_{\text{near } \psi(0)}^K: |\langle \tilde \nu_t,\mathds 1\rangle-\psi(0))|> M\e\s_K\right\} \Big]
												=0.\nonumber
\end{align}
\end{lemma}

\begin{proof} To prove this lemma we use  a coupling with an continuous time branching process as in the proof of lemma~\ref{Step1}.
For any $\th_{\text{near } \psi(0)}^K\leq t \leq \th^K_{2\text{ succ. mut.}} \wedge\th^{K}_{ \text{diversity}}
										\wedge\inf \{t\geq \th_{\text{near } \phi(0)}^K: |\langle \tilde \nu_t,\mathds 1\rangle-\psi(0))|> M\e\s_K\} $, any individual of $\mfm^0(\tilde\nu_t)$ gives birth to a new individual with trait  $R^K$ with rate 
\begin{align}
\bigl(1-u_K \:m(R^K)\bigr)b(R^K)\in \bigl[b(R^K)-u_K\:\overline b\:, b(R^K)\bigr],
 \end{align}
 and dies with rate 
\begin{align}
 d(R^K)+c(R^K,R^K)\mfm^0(\tilde\nu_t)+\int_{\mathcal X\times \mathbb N}c(R^K,\xi)\tilde d\nu_t(\xi),
 \end{align}
which is larger than $d_Z\equiv d(R^K)+c(R^K,R^K+\s_K h )\overline z(R^K+\s_K h)-C_{\text{total death}}^{M}\e\s_K$ where $C_{\text{total death}}^{M}\equiv M+\overline c \lceil 3 /\a\rceil - 2 h \partial_2c(R^K,R^K)$.
Therefore, we construct, by using a standard coupling argument, a process $Z_t$ such that 
\begin{align} Z_t\geq  \mfm^0(\tilde\nu_t)\end{align}
for all $\th_{\text{near } \psi(0)}^K\leq t \leq \th^K_{2\text{ succ. mut.}} \wedge\th^{K}_{ \text{diversity}}
										\wedge\inf \{t\geq \th_{\text{near } \phi(0)}^K: |\langle \tilde \nu_t,\mathds 1\rangle-\psi(0))|> M\e\s_K\} $.
The process $Z_t$ is a linear birth and death process starting at $\lceil \frac 5 4 \e K\rceil$, with birth rate per individual  $b_Z=b(R^K)$ and with death rate per individual
 $d_Z$. 
 Since 
 \begin{align}
 b_Z-d_Z&=f(R^K,R^K+\s_K h)+ C_{\text{total death}}^{M}\e\s_K\\&\nonumber
 =-\s_K h \partial_1f(R^K+\s_K h,R^K+\s_K h) +C_{\text{total death}}^{M}\e\s_K+O((\s_K h)^2)
 \equiv-\s_K \xi_K
 \end{align}
 is negative and of order $\sigma_K$, the process $Z_t$ is sub-critical. Note that  $\xi_K\geq \inf_{x\in\mathcal X}\tfrac{\partial_1 f(x,x)}{2}>0$. Let $\tau^Z_i$ be the first hitting time of level $i$ by $Z_t$, then we have
\be
\mathbb P[\tau^Z_{\lceil 2\e K\rceil }<\tau^Z_{0}]\leq \exp(-K^{\a})
\ee
 compare with the proof of Proposition \ref{prop2.2}. Since $Z_t\geq  \mfm^0(\tilde\nu_t)$, we obtain also that, with high probability, $\mfm^0(\tilde\nu_t)$ stays smaller than $\lceil 2\e K\rceil$ before it dies out. For any $t\geq 0$ and $n\in\mathbb N$, the distribution of the extinction time of $Z_t$ for $b_Z\neq d_Z$  is given by: 
\begin{align}
\mathbb P_n(\tau^Z_0\leq t)=\biggl(\frac {d_Z-d_Z \exp({(d_Z-b_Z)t})}{b_Z-d_Z \exp({(d_Z-b_Z)t})}\biggr)^n.
\end{align}
(cf. \cite{A_BP} p. 109 and \cite{C_TSS}).
Therefore, we can compute in our case where $d_Z-b_Z=\s_K\xi_K$ with $\xi_K$ uniformly positive
\bea\Eq(cross.20)
\mathbb P\Bigr[\tau^Z_{0}\leq\ln(K)\s_K^{-1-\a/2 }\Bigr]
&=&\left(\frac {d_Z-d_Z\exp{\left((d_Z-b_Z)\ln(K)\s_K^{-1-\a/2 }\right)}}{b_Z-d_Z \exp\left((d_Z-b_Z)\ln(K)\s_K^{-1-\a/2 }\right)}\right)^{\frac 5 4\e  K}
\quad\\\nonumber&=&\left(\frac {d_Z-d_Z K^{\xi_K \s_K^{-\a/2 }}}{d_Z-\s_K \xi_K -d_Z K^{\xi_K\s_K^{-\a/2 }}}\right)^{\frac 5 4\e  K}\\
\nonumber&=&\left(1-\frac {\xi_K \s_K}{d_Z(K^{\xi_K\s_K^{-\a/2 }}-1)+\s_K\xi_K}\right)^{\frac 5 4\e  K}\\
\nonumber&\geq& \left(1-\s_K ({\tfrac 5 4\e  K})^{-1}K^{-1}\right)^{\tfrac 5 4\e  K}\ 
\\&\geq& 1-O(\s_KK^{-1})\geq 1-o(\s_K),\nonumber
\eea
which proves the lemma.
\end{proof}
\subsection{Step 5}
After the extinction time of the resident trait, we have to wait at most $\ln(K)\s_K^{-1-\a/2}$ time until the population is monomorphic with trait $R^K+\s_K h$.
\begin{lemma}\label{Step5}
 Suppose that the assumptions of Theorem  \ref{Thm_2.Phase} hold. Then, there exists a constant $M>0$ (independent of $\e$ and $K$) such that, for all $\e>0$
\begin{align}
\lim_{K\to\infty}
\s_K^{-1}\;
\mathbb P\Big[
\th^K_{\text{fixation}} >(\th^K_{\text{res. size }0}+\ln(K)\s_K^{-1-\a/2})\wedge\th^K_{2\text{ succ. mut.}}
\wedge\th^{K}_{ \text{diversity}}\qquad
&\\\wedge\inf\left \{t\geq \th_{\text{near } \phi(0)}^K: |\langle \tilde \nu_t,\mathds 1\rangle-\psi(0)|> M\e\s_K\right\} 
&\Big]=0. \nonumber
\end{align}
\end{lemma}
\begin{proof}
By the last lemmata, we have $\th^K_{\text{fixation}}=\inf\{t\geq \th^K_{\text{res. size }0}:|\mathrm{Supp}(\tilde \nu^K_t)|=1, |\langle \tilde \nu_t,\mathds 1\rangle-\psi(0)|< (M/3)\e\s_K\}$ with probability $1-o(\s_K)$. Set $D\equiv \{k\in \mathbb N: 1\leq \mfm^k(\tilde \nu_{\th^K_{\text{res. size }0}})<\e\s_K K\}$. Then $|D|\leq\lceil 3/\a \rceil$ and 
none of these traits are successful since we have seen that $\th^K_{\text{res. size }0}$ is smaller than 
$\th^K_{2\text{ succ. mut.}}$ and $\th^{K}_{ \text{diversity}}$ 
with a probability of order $1-o(\s_K)$. By applying Proposition \ref{prop2} and using the Markov inequality, we obtain that the life 
time of each of these subpopulations is with probability $1-o(\s_K)$ smaller than $\ln(K)\s_K^{-1-\a/4}$. 
Therefore, if no new mutant is born between $\th^K_{\text{res. size }0}$ and $\th^K_{\text{res. size }0}+\ln(K)\s_K^{-1-\a/4}$, we 
obtain the claim.
 On the other hand, as in  Lemma \ref{rv_A},   the number of mutants born in the time interval 
 $[\th^K_{\text{res. size }0},\th^K_{\text{res. size }0}+\ln(K)\s_K^{-1-\a/2}]$ is stochastically dominated by a  
 Poisson point process, 
 $A^K(t)$,  with parameter $a\: u_K K$, where $a\equiv \sup_{x\in\mathcal X}\overline z(x)b(x)m(x)+1$. Hence, the probability to have no new mutant in this interval is
\be\Eq(cross.19)
\mathbb P\big[A^K(\ln(K)\s_K^{-1-\a/2})=0\big]=\exp(-\ln(K)\s_K^{-1-\a/2}a\: u_K K)\geq \exp(-\s_K^{\a/2})\geq 1-o(1).
\ee
Because the probability that a mutant is successful is of order $\s_K$, the probability that a successful mutant is born between times 
 $\th^K_{\text{res. size }0}$ and  $\th^K_{\text{res. size }0}+\ln(K)\s_K^{-1-\a/2}$ is $o(\s_K)$.
Since
\bea\Eq(cross.20)
&&\mathbb P\big[A^K(\ln(K)\s_K^{-1-\a/2})\leq\lceil 3/\a\rceil \big]\\\nonumber
&&=\exp\left(-\ln(K)\s_K^{-1-\a/2} \: a u_K K\right)\sum_{i=0}^{\lceil 3/\a\rceil }\frac{\ln(K)\s_K^{-1-\a/2}a\: u_K K}{i}\\\nonumber
&&\geq 1-(\ln(K)\s_K^{-1-\a/2}a\: u_K K)^{\lceil 3/\a\rceil +1}\\
&&\geq 1-\s_K^{3/2}=1-o(\s_K),\nonumber
\eea
there are maximal $\lceil 3/\a\rceil$ unsuccessful mutations in this interval. With the same argument as before the life time of each 
of these subpopulations is with probability $1-o(\s_K)$ smaller than $\ln(K)\s_K^{-1-\a/4}$. 
Therefore, with probability $1-o(\s_K)$ the maximal possible time interval where at least one mutant individual is alive
is smaller or equal $\ln(K)\s_K^{-1-\a/4}+ \lceil 3/\a\rceil \ln(K)\s_K^{-1-\a/4}$ $\ll $ $\ln(K)\s_K^{-1-\a/2}$. 
Recall from Lemma \ref{Step4.1} that if   
$|\langle \tilde \nu_t,\mathds 1\rangle-\psi(0)|> (M/3)\e\s_K$ at the first time when the population is again monomorphic, 
then the time the process needs to enter the $(M/3)\e\s_K$-neighborhood of $\psi(0)$ is smaller than $S_K$, which can be chosen 
smaller than $\s_K^{1+\a}/(Ku_K)$. This proves the lemma.
\end{proof}
This ends up Step 5 and the second invasion phase. Note that the estimates of the two phases do not depend on 
the exact trait value of the resident trait, especially the a priori different constants $M$. In fact, we can use in all lemmata the same constant $M$, namely the largest.  
Therefore, we can apply our results for the successful mutant trait $R^K_1=R^K+\s_K h$, 
which is the next resident trait by using the strong Markov property for $(\tilde\nu, L)$ at the stopping time $\th^K_{\text{fixation}}$.
% $\inf\{t\geq\th^K_{\text{res. size }0}:|\mathrm{Supp}(\tilde \nu^K_t)|=1, |\langle \tilde \nu_t,\mathds 1\rangle-\psi(0)|< (M/3)\e\s_K\}$.
 %

%%%%%%%%%%%%%%%%%%%%%%%%%%%%%%%%%%%%%%%%%%%%%%%%%%%%%%%%%%%%%%%%%%%%%%%%%%%%%%%%%%%%%%%%%%%%%
%%Convergence
%%%%%%%%%%%%%%%%%%%%%%%%%%%%%%%%%%%%%%%%%%%%%%%%%%%%%%%%%%%%%%%%%%%%%%%%%%%%%%%%%%%%%%%%%%%%%

\section{Convergence to the CEAD }\label{Convergence}
Our goal is to find $T_0>0$ and to construct, for all $\e>0$, two measure valued processes, $(\mu^{1,K,\e}_t,t\geq 0)$ and
$(\mu^{2,K,\e}_t,t\geq 0)$, in $\mathbb D([0,\infty),\mathcal M(\mathcal X))$ such that
\be
\lim_{K\rightarrow \infty} \mathbb P\left [\forall\: t\leq \tfrac {T_0}{Ku_K\sigma_K^2}:\quad \mu_t^{1,K,\e} \preccurlyeq \nu_t^{K}
  \preccurlyeq \mu_t^{2,K,\e} \:\right]  =1, \label{eq:conv-1}
\ee
and for $j\in\{1,2\}$
\begin{align}
\lim_{K\rightarrow \infty}
		 \mathbb P\left [\:\sup_{0\leq t\leq T_0} \Big\Vert \: \mu^{i,K,\e}_{t/(K u_K \sigma_K{}^2)}-\overline
                   z(x_t)\delta_{x^{}_t}\:\Big\Vert^{}_0>\delta(\e) \:\right] & = 0. \label{eq:conv-2},
\end{align}
for some function $\delta$ independent of $x,K$ such that $\delta(\e)\rightarrow 0$ when $\e\rightarrow 0$. This easily
implies~\eqref{conv_in_proba_first} for all $T\leq T_0$. 

The result for all $T>0$ then follows from the strong Markov property. Indeed, the construction below implies that there exists a
stopping time $\tau\in[T_0/2Ku_K\sigma_K^2,T_0/Ku_K\sigma_K^2]$ (a fixation time) such that, with probability converging to 1,
$\nu_\tau^{K}$ has a unique (random) point $Y$ as support and a total mass belonging to
$[\bar{z}(Y)-M\sigma_K,\bar{z}(Y)+M\sigma_K]$. Hence~\eqref{eq:conv-1} and~\eqref{eq:conv-2} also hold for the process
$(\nu_{\tau+t}^{K},t\geq 0)$, and~\eqref{conv_in_proba_first} is thus true for all $T\leq 3T_0/2$. We obtain~\eqref{conv_in_proba_first} for any
fixed $T>0$ by induction.

%%%%%%%%%%%%%%%%%%%%%%%%%%%%%%%%%%%%%%%%%%%%%%%%%%%%%%%%%%%%%%%%%%%%%%

\subsection{Construction of two processes $\mu^{K,1}$ and $\mu^{K,2}$ such that $\mu_t^{1,K} \preccurlyeq \nu_t^{K}
  \preccurlyeq \mu_t^{2,K}$}
\label{sec:step-1-conv}

Fix $T>0$. Let $\theta_{i}^{K}$ denote the random time of $i$-th invasion (i.e. $\theta_{i}^{K}=\theta_{i, \text{invasion}}^{K}$),
$\theta_{i,\text{fixation}}^{K}$ the time of $i$-th fixation %(as defined in Lemma~\ref{Step5})
and $R_{i}^{K}$ the trait of the $i$th successful mutant. 
Let us fix the following initial conditions $R_{0}^{K,1}=R^K_0-A\s_K$, $R_{0}^{K,2}=R^K_0+A\s_K$ and
$\theta_{0}^{K,1}=\theta_{0}^{K,2}=0$. Assume that we have constructed $\theta_{i}^{K,1}$ and $\theta_{i}^{K,2}$, and $R_{i}^{K,1}$
and $R_{i}^{K,2}$. By Theorem \ref{1.Phase} and Markov's property, we can construct two random variables $R_{i+1}^{K,1}$ and
$R_{i+1}^{K,2}$ such that
\be
R_{i+1}^{K,1}-R^{K,1}_i\leq R_{i+1}^K-R^K_i\leq R^{K,2}_{i+1}-R^{K,2}_i
\ee
with probability $1-o(\sigma_K)$. Moreover, $R^{K,1}_{i+1}-R^{K,1}_i=R^K_{i+1}-R^K_{i}=R^{K,2}_{i+1}-R^{K,2}_i$ with probability $1-O(\e)$ and
$R^{K,2}_{i+1}-R^{K,1}_{i+1}\leq A \s_K $. The distributions of $R_{i+1}^{K,1}-R^{K,1}_i$ and $R_{i+1}^{K,2}-R^{K,2}_i$ are
(cf. Corollary \ref{new_resident_trait})
\begin{equation}
  % \label{eq:conv-3}
  r^\e_1(R^K_i,h)\equiv \mathbb P[ R_{i+1}^{K,1}\!=\! R_i^K\!+\!\s_k h]
  =\begin{cases}
    \frac{M(R^K_i,1)q^\e_1(R^K_i,1)}{p^\e_2(R^K_i)}+1-\frac{p^\e_1(R^K_i)}{p^\e_2(R^K_i)} 
    & \text{if }h=1\\[0.5em]
    \frac{M(R^K_i,h)q^\e_1(R^K_i,h)}{p^\e_2(R^K_i)} 
    & \text{if }h\in\{2,...,A\}\
  \end{cases}
\end{equation}
and
\begin{equation}
  % \label{eq:conv-4}
  r^\e_2(R^K_i,h)\equiv \mathbb P[ R_{i+1}^{K,2}\!=\! R_i^K\!+\!\s_k h]
  =\begin{cases}
    \frac{M(R^K_i,h)q^\e_1(R^K_i,h)}{p^\e_2(R^K_i)} 
    & \text{if }h\in\{1,...,A\!-\!1\}\\[0.5em]
    \frac{M(R^K_i,A)q^\e_1(R^K_i,A)}{p^\e_2(R^K_i)}+1-\frac{p^\e_1(R^K_i)}{p^\e_2(R^K_i)} 
    & \text{if }h=A
  \end{cases},
\end{equation}
where 
\begin{equation}
  \label{eq:def-q}
  q^\e_1(x,h)=h\frac{\partial_1 f(x,x)}{b(x)}-C^1_{\text{Bernoulli}}\e,\qquad q^\e_2(x,h)=h\frac{\partial_1
    f(x,x)}{b(x)}+C^2_{\text{Bernoulli}}\e
\end{equation}
and $p^\e_j(x)=\sum_{h=1}^A q^\e_j(x,h)M(x,h)$ for $j=1,2$. (Note that we changed a bit the notations of
Corollary~\ref{new_resident_trait} to make explicit the dependence on $\e$ and $R^K_i$.) Since we assumed that the fitness gradient
$\partial_1 f(x,x)$ is positive and uniformly lower bounded on ${\cal X}$, the transition probabilities $r^\e_j(x,h)$, $j=1,2$ are
uniformly Lipschitz-continuous functions of $x$ with some Lipschitz constant $C^r_{\text{Lip}}$.
%where, for all $h\in\{1,\ldots A\}$, $ q_1^K(h)\equiv \big(h \tfrac{\partial_1f(R^K,R^K) }{b(R^K)}-C^1_{\text{Bernoulli}}\e \big)M(R^K,h)$  and $q_2^K(h)\equiv \big(h \tfrac{\partial_1f(R^K,R^K) }{b(R^K)}+C^2_{\text{Bernoulli}}\e \big)M(R^K,h)$ 
%and $p_i^K\equiv\sum_{h=1}^{A}q_i^K(h)$.

By Theorem \ref{1.Phase} and Lemmata~\ref{bounds_for_invasion_time} and~\ref{rv_A}, we can construct two exponential
random variables, $E^{K,1}_{i+1}$ and $E^{K,2}_{i+1}$ ,with parameters $a^{K,\e}_1(R^K_i)p^\e_1(R^K_i) \s_K u_K K $ and
$a^{K,\e}_2(R^K_i)p^\e_2(R^K_i) \s_K u_K K $ given by
\begin{align}
  a_1^{K,\e}(x) & =(\bar z(x)-\e\s_K M)b(x)m(x) \\
  a_2^{K,\e}(x) & =(\bar{z}(x)+\e\s_K(M+\lceil 3/\alpha\rceil))(b(x)m(x)+C^{b,m,M}_L A\sigma_K),
\end{align}
such that
\be
\P\left(
E^{K,2}_{i+1}\leq \th^{K}_{i+1}-\theta_{i,\text{fixation}}^{K} \leq E^{K,1}_{i+1}+\ln(K)\s_K^{-1-\a/2 }\right)
=1-o(\sigma_K). 
\ee
Note that this inequality involves $\theta_{i,\text{fixation}}^{K} $ instead of $\theta^K_i$ since
we apply the Markov property at the fixation time of Lemma~\ref{Step5} before we can apply Theorem~\ref{1.Phase}. However,
Lemma~\ref{Step5} entails that we also have
\be
\P\left(E^{K,2}_{i+1}\leq \th^{K}_{i+1}-\theta^K_i\leq E^{K,1}_{i+1}+6\ln(K)\s_K^{-1-\a/2 }\right)
=1-o(\sigma_K).
\ee
 We then define
\begin{equation}
  \label{eq:conv}
  \theta^{K,1}_{i+1}-\theta^{K,1}_i\equiv E^{K,1}_{i+1}+6\ln(K)\s_K^{-1-\a/2 }
  \quad\text{and}\quad \theta^{K,2}_{i+1}-\theta^{K,2}_i\equiv E^{K,2}_{i+1}.
\end{equation}
In addition, by their construction in Section~\ref{First_Phase}, it is clear that the random vectors
$\{(E^{K,1}_{i+1},E^{K,2}_{i+1}, $ $R^{K,1}_{i+1}-R^{K,1}_i,R^{K,2}_{i+1}-R^{K,2}_i)\}_{i\geq 0}$ are independent conditionally on $(R^K_j)_{j\geq 0}$.

%%%%%%%%%%%%%%%%%%%%%%%%%%%%%%%%%%%%%%%%%%%%%%%%%%%%%%%%%%%%%%%%%%%%%%%%
\begin{lemma}
  \label{lem:conv-1}
  With the previous notations, the stochastic processes, $\;\mu^{K,1}\;$ and $\;\mu^{K,2}\;$, in $\:\mathbb D([0,\infty),\mathcal M(\mathcal X))$
  defined for all $t\geq 0$ by
  \begin{align}
    \mu_t^{1,K}&=(\bar z(R_{j}^{K})- (M\epsilon+\overline{C}) \sigma_K)\delta_{R_{i}^{K,1}},  &\text{ for }
    t\in[\theta_{i}^{K,1},\theta_{i+1}^{K,1} )\cap[\theta_{j}^K,\theta_{j+1}^K), \\
    \mu_t^{2,K}&=(\bar z(R_{j}^{K})+ (M\epsilon+\lceil 3/\alpha\rceil\epsilon+\overline{C})\sigma_K)\delta_{R_{i}^{K,2}} , &\text{ for } t\in[\theta_{i}^{K,2},\theta_{i+1}^{K,2} )\cap[\theta_{j}^K,\theta_{j+1}^K),
  \end{align}
  for some constant $\overline{C}$ independent of $K,x,\e$, satisfy for all $T>0$
  \be
  \lim_{K\rightarrow \infty} \mathbb P\left [\forall\: t\leq \tfrac T{Ku_K\sigma_K^2}:\quad \mu_t^{1,K} \preccurlyeq \nu_t^{K} \preccurlyeq \mu_t^{2,K} \:\right]=1.
  \ee
\end{lemma}

Note that the support of $\mu^{j,K}$, $j=1,2$, is defined from the sequences $(R^{K,j}_i)_{i\geq 0}$ and $(\theta^{K,j}_i)_{i\geq 1}$
but the mass of $\mu^{j,K}$ is defined from the sequences $(R^{K}_i)_{i\geq 0}$ and $(\theta^{K}_i)_{i\geq 1}$.

\begin{proof}
  Let us fix $T>0$ and $\Gamma>0$. Since each of the steps previously described holds with probability $1-o(\sigma_K)$, we deduce
  that the above construction can be done on a so-called \emph{good event} of probability $1-o(1)$ for all integers
  $i\leq\Gamma/\s_K$. Since in addition $a^{K,\e}_2(x)p^\e_2(x)$ is uniformly lower bounded by a positive constant $\underline{a}$ on
  ${\cal X}$, the random variables \ $E^{K,2}_i$ can be coupled with i.i.d.\ exponential ones of parameter $\underline{a}K u_K\sigma_K$, and hence
  $\mathbb{P}\big[\:\theta^{K,2}_{\lfloor \Gamma/\sigma_K\rfloor}<T/(K u_K\sigma_K^2)\:\big]$ is smaller than the probability that a Poisson
  process with parameter $\underline{a}K u_K\sigma_K$ is larger that $\lfloor\Gamma/\sigma_K\rfloor$ at time $T/(Ku_K\sigma_K^2)$. By
  the law of large numbers for Poisson processes, we deduce that, provided that $\Gamma>T\underline{a}$ (which we assume true in the
  sequel),
  \begin{equation}
    \label{eq:conv-bis}
    \lim_{K\rightarrow\infty}\mathbb{P}\left[\theta^{K,2}_{\lfloor
        \Gamma/\sigma_K\rfloor}< \frac{T}{K u_K\sigma_K^2}\right]=0.  
  \end{equation}
 
  Let us recall that, on the previous good event of probability $1-o(1)$, the number, the trait and the size of the living mutant
  populations and the size of the resident population are controlled at any time in the $i$-th first phase
  (Lemmata~\ref{exit_from_domain} and~\ref{2.succ_mutant}). In addition, during the $i$-th second phase, the number, trait and size of living
  mutant populations are controlled (see all the Lemmas of Section~\ref{2.Phase}), the total mass of the population stays within the
  $M\e\sigma_K$-neighborhood of $\phi(y)$ or $\psi(y)$ for some $y\in[0,\bar{z}(R^K_i)]$ (Lemmata~\ref{Step2} and~\ref{Step3}). Since
  $|\phi(y)-\bar{z}(R^K_i)|\leq \overline{C}\sigma_K$ and $|\psi(y)-\bar{z}(R^{K}_i)|\leq\overline{C}\sigma_K$ for some constant
  $\overline{C}$, as seen in~\eqref{phi.1} and~\eqref{psi.1}, and since the sequences $(R^{j,K}_i)_{i\geq 0}$ for $j=1,2$ and
  $(R^K_i)_{i\geq 0}$ are all increasing on the good event, we deduce the required comparison between the supports of
  $\mu_t^{1,K}$, $\nu_t^{K}$ and $\mu_t^{2,K}$ for $t\leq\frac T{Ku_K\sigma_K^2}$, on the good event. Since we used $\bar
  z(R_{j}^{K})$ to define the masses of $\mu_t^{1,K}$ and $\mu_t^{2,K}$, the required comparison between the masses is also clear.
\end{proof}

%%%%%%%%%%%%%%%%%%%%%%%%%%%%%%%%%%%%%%%%%%%%%%%%%%%%%%%%%%%%%%%%%%%%%

 Note that, since the function $\bar z$ may not be non-decreasing, replacing $\bar z(R_{j}^{K})$ by $\bar z(R_{j}^{K,1})$ in the
definition of $\mu^{1,K}_t$ may not imply the required comparison between the masses of $\mu_t^{1,K}$, $\nu_t^{K}$ and $\mu_t^{2,K}$.
\medskip

The next goal is now to prove the convergence of both processes $\mu^{K,j}_{t/K u_K\sigma_K^2}$ for $j=1,2$ to $\bar{z}(x_t)\delta_{x_t}$ in
probability in $L^\infty(\mathcal M(\mathcal X),\|\cdot\|_0)$. For this, we will use standard convergence results of Markov jump
processes. However, the two processes $\mu^{K,j}$, $j=1,2$ are not Markov because the $i$-th jump rates and transition probabilities
defined above depend on $R^K_i$ which is close, but different from $R^{K,j}_i$. Therefore, we introduce a small parameter $\eta>0$,
and we shall construct two Markov processes ${\mu}^{K,j,\e,\eta}$, $j=1,2$ in $\mathbb D([0,\infty),\mathcal M(\mathcal X))$ such that
\begin{equation}
  \label{eq:conv-7}
  \lim_{K\rightarrow+\infty}\mathbb{P}\left[{\mu}^{K,1,\e,\eta}_{(t-1/(Ku_K\sigma_K))\vee 0} \preccurlyeq \mu_t^{1,K} \preccurlyeq \nu_t^{K} \preccurlyeq \mu_t^{2,K} \preccurlyeq
    {\mu}^{K,2,\e,\eta}_t,\ \forall t\leq\frac{T}{K u_K\sigma_K^2}\wedge S^K_\eta\right]=1,
\end{equation}
 where $S^K_\eta$ is the first time where the distance between the support of
$\mu^{K,1,\e,\eta}_t$ and $\mu^{K,2,\e,\eta}_t$ is larger than $\eta$. The last equation will be proved below in Section~\ref{sec:step-2-conv}. The
  time-shift of $-1/(Ku_K\sigma_K)$ in ${\mu}^{K,1,\e,\eta}$ is due to the terms $6\ln(K)\s_K^{-1-\a/2 }$ in~\eqref{eq:conv}. We will next study the convergence of these two Markov
processes when $K\rightarrow\infty$ and prove in Section~\ref{sec:step-3-conv} that, for a convenient choice of $\eta$, there exists some $T_0>0$ independent of
$K,x,\e,\eta$ such that
\begin{equation}
  \label{eq:conv-12}
  \lim_{K\rightarrow+\infty}\mathbb{P}\left[S^K_\eta<\frac {T_0}{Ku_K\sigma_K^2}\right]=0.
\end{equation}
% will then allow us to conclude that the probability\equiv  
% that $S^K_\eta<T/(Ku_K\sigma_K^2)$ converges to 0 when $K\rightarrow+\infty$.

%%%%%%%%%%%%%%%%%%%%%%%%%%%%%%%%%%%%%%%%%%%%%%%%%%%%%%%%%%%%%%%%%%%%%%%
\subsection{Proof of~\eqref{eq:conv-7}}
\label{sec:step-2-conv}

For all $x\in{\cal X}$, we define $(\bar{r}^{\e,\eta}_1(x,h),1\leq h\leq A)$ and $(\bar{r}^{\e,\eta}_2(x,h),1\leq h\leq A)$ by, for all $1\leq
\ell\leq A$,
\be
\sum_{h=1}^\ell\bar{r}^{\e,\eta}_1(x,h)\equiv \left[\sum_{h=1}^\ell(r^\e_1(x,h)+C^r_{\text{Lip}}\eta)\right]\wedge 1\geq
\sup_{y\in[x,x+\eta]}\sum_{h=1}^\ell r^\e_1(y,h)
% \begin{cases}
%   (r_1(h,x)-C^r_{\text{Lipschitz}}\eta)\vee 0 & \text{if }h\in\{2,\ldots,A\} \\
%   1-\sum_{\ell=2}^A \bar{r}^{\eta}_1(\ell,x) & \text{if }h=1.
% \end{cases}
\ee
and
\be
\sum_{h=1}^\ell\bar{r}^{\e,\eta}_2(x,h)\equiv \left[\sum_{h=1}^\ell(r^\e_2(x,h)-C^r_{\text{Lip}}\eta)\right]\vee 0\leq
\inf_{y\in[x,x+\eta]}\sum_{h=1}^\ell r^\e_1(y,h).
% \bar{r}^{\eta}_2(h,x)\equiv 
% \begin{cases}
%   \sup_{y\in[x-\eta,x]\cap{\cal X}}r_1(h,y) & \text{if }h\in\{1,\ldots,A-1\} \\
%   1-\sum_{\ell=1}^{A-1} \bar{r}^{\eta}_2(\ell,x) & \text{if }h=A.
% \end{cases}
\ee
Note that $\bar r^{\e,\eta}_1(x,\cdot)$ and $\bar r^{\e,\eta}_2(x,\cdot)$ are probability distributions on $\{1,\ldots,A\}$ for all
$x\in{\cal X}$ and that, by standard coupling arguments, for all $x<y$ such that $y-x\leq \eta$, the distribution $\bar
r^{\e,\eta}_1(x,\cdot)$ is stochastically dominated by the distribution $r^\e_1(y,\cdot)$ and the distribution $r^\e_2(x,\cdot)$ is
stochastically dominated by the distribution $\bar r^{\e,\eta}_2(y,\cdot)$. We define similarly 
\be
\bar a^{K,\e,\eta}_1(x)
\equiv a^{K,\e}_1(x)p^\e_1(x)-C^a_{\text{Lip}}\eta \leq\inf_{y\in[x,x+\eta]\cap{\cal X}} a^{K,\e}_1(y)p^\e_1(y) ,
\ee
 and
 \be
\bar a^{K,\e,\eta}_2(x)\equiv a^{K,\e}_2(x)p^\e_2(x)+C^a_{\text{Lip}}\eta \geq\sup_{y\in[x-\eta,x]\cap{\cal X}} a^{K,\e}_2(y)p_2(y),
\ee
 where
$C^{a}_{\text{Lip}}$ is a uniform Lipschitz constant for the functions $a^{K,\e}_jp^\e_j$, $j=1,2$. Note that $a^{K,\e,\eta}_1(x)>0$ for all
$x\in{\cal X}$ if $\eta$ is small enough.

It is then clear that there exist two Markov chains $(\bar R^{K,j,\eta}_i)_{i\geq 0}$, $j=1,2$, with initial condition $\bar
R^{K,j,\eta}_0=R^{K,j}_0$ and with transition probabilities $\bar{r}^{K,\e,\eta}_j(x,h)$ from $x$ to $x+h$, such that, for all $i\geq 0$
satisfying $\bar{R}^{K,2,\eta}_{i}-\bar{R}^{K,1,\eta}_{i}\leq\eta$,
\be
\bar{R}^{K,1,\eta}_{i+1}-\bar{R}^{K,1,\eta}_{i}\leq R^{K,1}_{i+1}-R^{K,1}_{i}\quad\text{and}\quad
R^{K,2}_{i+1}-R^{K,2}_{i}\leq\bar{R}^{K,2,\eta}_{i+1}-\bar{R}^{K,2,\eta}_{i}\leq R^{K,1}_{i+1}-R^{K,1}_{i}.
\ee
Similarly, there exists random variables \ $\bar E^{K,j,\eta}_{i+1}$, $j=1,2$, independent and exponentially distributed with parameters $\bar
a^{K,\e,\eta}_j(\bar R^{K,j,\eta}_i)$ conditionally on $(\bar R^{K,j,\eta}_i)_{i\geq 0}$, such that $\bar E^{K,2,\eta}_{i+1}\leq E^{K,2}_{i+1}$ and
$E^{K,1}_{i+1}\leq\bar E^{K,1,\eta}_{i+1}$. We then define
 $\bar\theta^{K,j,\eta}_{i+1}-\bar\theta^{K,j,\eta}_i=E^{K,j,\theta}_{i+1}$
with $\bar\theta^{K,j,\eta}_0=0$.

Since the function $\bar z$ is $C^{\bar z}_{\text{Lip}}$-Lipschitz, it is  clear that~\eqref{eq:conv-7} is satisfied for the processes
\bea
&&\bar\mu^{K,1,\e,\eta}_t=(\bar z(\bar X^{K,1,\eta}_t)- (M\epsilon+\bar C) \sigma_K-C^{\bar z}_{\text{Lip}}\eta)\delta_{\bar X^{K,1,\eta}_t}\\
\text{and }&&
{\mu}^{K,2,\e,\eta}_t=(\bar z(X^{K,2,\eta}_t)+ (M\epsilon+\lceil 3/\alpha\rceil\epsilon+\bar C)\sigma_K+C^{\bar z}_{\text{Lip}}\eta)\delta_{X^{K,2,\eta}_t},
\eea
where
\be
  \bar X^{K,1,\eta}_t  =\bar R_{i}^{K,1,\eta}, \; \text{
    for } t\in\big [\bar
  \theta_{i}^{K,1,\eta}+6i\ln(K)\sigma_K^{-1-\alpha/2},\:\bar\theta_{i+1}^{K,1,\eta}+ 6(i+1)\ln(K)\sigma_K^{-1-\alpha/2}\big), 
  \ee
  and 
  \be
  X^{K,2,\eta}_t=\bar R_{i}^{K,2,\eta} ,\, \text{ for } t\in[\bar\theta_{i}^{K,2,\eta},\bar\theta_{i+1}^{K,2,\eta} ).
\ee
By construction, the processes $X^{K,2,\eta}$ and ${\mu}^{K,2,\eta}$ are Markov jump processes, but the process $\bar X^{K,1,\eta}$ is
not because of the terms $6\ln(K)\sigma_K^{-1-\alpha/2}$ involved in its definition. However, the process ${\mu}^{K,1,\e,\eta}_t=(\bar
z(X^{K,1,\eta}_t)- \epsilon \sigma_K M-C^{\bar z}_{\text{Lip}}\eta)\delta_{X^{K,1,\eta}_t}$ is Markov, where
\be
X^{K,1,\eta}_t=\bar R_{i}^{K,1,\eta}, \quad \text{
    for } t\in[\bar
  \theta_{i}^{K,1,\eta},\bar\theta_{i+1}^{K,1,\eta}).
\ee
The proof of~\eqref{eq:conv-bis} above also applies to the processes ${\mu}^{K,1,\e,\eta}$ and $\bar \mu^{K,1,\e,\eta}$. Since
in addition the support of ${\mu}_t^{K,1,\e,\eta}$ is non-decreasing, it follows that
${\mu}_{(t-6\Gamma\ln(K)\sigma_K^{-2-\alpha/2})\vee 0}^{K,1,\e,\eta}\preccurlyeq\bar{\mu}_t^{K,1,\e,\eta}$ for all $t\leq
T/(Ku_K\sigma_K^2)$ with probability $1+o(1)$. Our assumption~\eqref{conv2} entails~\eqref{eq:conv-7}.
%
%%%%%%%%%%%%%%%%%%%%%%%%%%%%%%%%%%%%%%%%%%%%%%%%%%%%%%%%%%%%%%%%%%%%%%%
%
\subsection{Convergence of $X^{K,j,\eta}$ when $K\rightarrow+\infty$ and proof of~\eqref{eq:conv-12}}
\label{sec:step-3-conv}

The two Markov processes $X^{K,1,\eta}_{t/(Ku_K\sigma_K)}$ and $X^{K,2,\eta}_{t/(Ku_K\sigma_K)}$ fit exactly to the framework and
assumptions of Theorem~2.1 of Chapter 11 of~\cite{E_MP}: their state spaces are (up to a translation) a subset of $\sigma_K\mathbb{Z}$,
and their transition rates from $z$ to $z+h\sigma_K$ have the form $\sigma_K^{-1}[\beta_h(z)+O(\sigma_K)]$ for some Lipschitz
functions $\beta_h$. For such a process $X$, provided $X_0$ converges a.s.\ to $x_0$, the process $(X_{t/\sigma_K},t\geq 0)$
converges when $\sigma_K\rightarrow 0$ almost surely in $L^\infty([0,T])$ for all $T>0$ to the unique deterministic solution of the
ODE\: $dx(t)/dt=\sum_{h}h\beta_h(x)$ with $x(0)=x_0$. In our situation, we obtain, for $j=1,2$, that
\begin{equation}
  \label{eq:conv-15}
  \lim_{K\rightarrow+\infty}\sup_{t\in[0,T]}\left|X^{K,j,\eta}_{t/(Ku_K\sigma^2_K)}-x_j(t)\right|=0\quad\text{a.s.,}
\end{equation}
where $x_1$ and $x_2$ are the unique solutions such that $x_1(0)=x_2(0)=x$ of the ODEs
\be
\frac{d x_1(t)}{dt}=\big[\bar{z}(x_1(t))b(x_1(t))m(x_1(t))p^\e_1(x_1(t))-C^a_{\text{Lip}}\eta\big]\sum_{h=1}^A h\bar r^{\e,\eta}_1(x_1(t),h)
\ee
and
\be
\frac{d x_2(t)}{dt}=\big[\bar{z}(x_2(t))b(x_2(t))m(x_2(t))p^\e_2(x_2(t))+C^a_{\text{Lip}}\eta\big]\sum_{h=1}^A h\bar r^{\e,\eta}_2(x_2(t),h).
\ee

\begin{lemma}
  \label{lem:last-lemma}
  For all $T>0$, and for $j=1,2$,
  \be
  \sup_{t\in[0,T]}|x_j(t)-x_t|\leq CTe^{CT}(\eta+\e),
  \ee
  for a constant $C$ independent of $x$, $T$, $\e$ and $\eta$, where $x_t$ is the solution of the CEAD~\eqref{CEAD} with initial
  condition $x_0=x$.
\end{lemma}

\begin{proof}
  We only write the proof for $j=1$, the case $j=2$ being similar. Since the functions $\bar r^{\e,\eta}_j$, $j=1,2$, $\bar z$, $b$,
  $m$ and $p_1$ are bounded by constants independent of $K,\e,\eta$, we have for all $t\in[0,T]$ and for a constant $C>0$ that may
  change from line to line,
  \begin{align}\nonumber
    |x_t-x_1(t)| & \leq CC^a_{\text{Lip}}\eta T+\int_0^t\Bigg|(\bar{z}bmp^\e_1)(x_1(s))\sum_{h=1}^A
    h\bar{r}^{\e,\eta}_1(x_1(s),h) \\\nonumber
    &\qquad\qquad -(\bar{z}bmp^\e_1)(x_s)\sum_{h=1}^A \frac{h^2M(x_s,h)\partial_1
      f(x_s,x_s)}{b(x_s)p^\e_1(x_s)}\Bigg|ds \\\nonumber
    % & \leq CC^a_{\text{Lip}}\eta T+C\int_0^T\left|(\bar{z}bmp^\e_1)(x_1(s))-(\bar{z}bmp^\e_1)(x(s))\right|ds
    % + C\int_0^t\sum_{h=1}^A\left|\bar{r}^{\e,\eta}_1(x_1(s),h)-\frac{hM(x(s),h)\partial_1
    %     f(x(s),x(s))}{b(x(s))p^\e_1(x(s))}\right|ds \\
    & \leq C(C^a_{\text{Lip}}+AC^r_{\text{Lip}})T\eta+C\int_0^t|x_s-x_1(s)|ds \\ & \qquad\qquad
    +C\int_0^t\sum_{h=1}^A\left|r^{\e}_1(x_s,h)-\frac{hM(x_s,h)\partial_1
        f(x_s,x_s)}{b(x_s)p^\e_1(x_s)}\right|ds,
  \end{align}
  where the last inequality follows from the uniform Lipschitz-continuity of all functions involved in the computation. Now,
  $|p^\e_2(x)-p^\e_1(x)|\leq C\e$ and $p_j^\e(x)\geq c>0$ for $j=1,2$, for some constants $C,c>0$ independent of $\e$ and $x$. Hence,
  there exists a constant $C$ such that
  \begin{align}\nonumber
    |x_t-x_1(t)| & \leq CT(\eta+\e)+C\int_0^t|x_s-x_1(s)|ds \\ & \qquad
    +C\int_0^t\sum_{h=1}^A\left|q^\e_1(x_s,h)-h\frac{\partial_1
        f(x_s,x_s)}{b(x_s)}\right|M(x_s,h)ds.
  \end{align}
  In view of~\eqref{eq:def-q}, we obtain $|x_t-x_1(t)|\leq CT(\eta+\e)+C\int_0^t|x_s-x_1(s)|ds$. Gronwall's lemma ends the
  proof of Lemma~\ref{lem:last-lemma}.
\end{proof}

In view of Lemma~\ref{lem:last-lemma}, there exists $T_0>0$ independent of $x,\e,\eta$ such that, for all $\eta\geq\e$,
$\sup_{t\in[0,T_0]}|x_j(t)-x_t|\leq\eta/4$. Let us fix $\eta=\e$. Combining~\eqref{eq:conv-15} with the last inequality
entails~\eqref{eq:conv-12}.

\subsection{End of the proof}
\label{sec:end-conv}
\begin{proof}[Proof of Theorem \ref{main_thm}]
Defining  $\bar\mu^{K,1,\e}={\mu}^{K,1,\e,\e}$ and $\mu^{K,2,\e}={\mu}^{K,2,\e,\e}$,  and 
combining~\eqref{eq:conv-7} and~\eqref{eq:conv-12}, we see that we have
defined a constant $T_0>0$ such that
\begin{equation}
  % \label{eq:conv-7}
  \lim_{K\rightarrow+\infty}\mathbb{P}\left[\bar{\mu}^{K,1,\e}_{(t-1/(Ku_K\sigma_K))\vee 0} \preccurlyeq \mu_t^{1,K} \preccurlyeq \nu_t^{K} \preccurlyeq \mu_t^{2,K} \preccurlyeq
    {\mu}_t^{K,2,\e},\ \forall t\leq\frac{T_0}{K u_K\sigma_K^2}\right]=1,
\end{equation}
This is~\eqref{eq:conv-1} with $\mu^{K,1,\e}_t=\bar{\mu}^{K,1,\e}_{(t-1/(Ku_K\sigma_K))\vee 0}$. 
It only remains to check~\eqref{eq:conv-2}.

Using that $\eta=\e$, we get
\begin{align}
&  \left\|\mu^{K,1,\e}_{t/Ku_K\sigma_K^2}-\bar{z}(x(t))\delta_{x(t)}\right  \|_0 \\
    & \leq \nonumber
    C\left[\e+\sigma_K+\left|\bar{z}(x_t)-\bar{z}\left(X^{K,1,\eta}_{(t-\sigma_K)\vee
            0/Ku_K\sigma_K^2}\right)\right|+|x_t-x_1((t-\sigma_K)\vee 0)|\right. \\\nonumber
    & \qquad+\left.\left|X^{K,1,\eta}_{(t-\sigma_K)\vee 0/Ku_K\sigma_K^2}-x_1((t-\sigma_K)\vee 0)\right|\right] \\\nonumber
    & \leq C'\left[\e+\sigma_K+\sup_{t\in[0,T]}\left(|x_{(t-\sigma_K)\vee 0}-x_t|+|x_t-x_1(t)|+
      \left|X^{K,1,\eta}_{t/Ku_K\sigma_K^2}-x_1(t)\right|\right)\right],
  \end{align}
for some finite  constants $C,C'>0$. The analogous estimate holds for  for $\mu^{2,K,\eta}_{t/Ku_K\sigma_K^2}$. 
Setting for example
$\delta(\e)=\sqrt{\e}$,~\eqref{eq:conv-2} follows from~\eqref{eq:conv-15}, 
Lemma~\ref{lem:last-lemma} and the uniform continuity of
$x_t$. This ends the proof of Theorem~\ref{main_thm}.
\end{proof}

%%%%%%%%%%%%%%%%%%%%%%%%%%%%%%%%%%%%%%%%%%%%%%%%%%%%%%%%%%%%%%%%%%%%%%%%%%%%%%%%%%%%%%%%%%%%%
%%Appendix
%%%%%%%%%%%%%%%%%%%%%%%%%%%%%%%%%%%%%%%%%%%%%%%%%%%%%%%%%%%%%%%%%%%%%%%%%%%%%%%%%%%%%%%%%%%%%
\section{Appendix}
In this section, we state and prove several elementary results, which we used in the proof of our main theorem. Recall that 
$\Vert \:.\: \Vert^{}_0$ is the Kantorovich-Rubinstein norm on the vector space of finite, signed measures on $\mathcal X$, i.e.
\begin{equation}
		\Vert \mu_t \Vert^{}_0\equiv \sup\Big\{\int_{\mathcal X} f d\mu_t: f\in \text{Lip}_1(\mathcal X) \text{ with } \sup_{x\in\mathcal X}|f(x)|\leq 1\Big\},
\end{equation}		
		where $\text{Lip}_1(\mathcal X)$  is the space of Lipschitz continuous functions from $\mathcal X$ to $\mathbb R$. 
Let $\mathcal M_F(\mathcal X)$ be the set of non-negative finite Borel-measures on $\mathcal X$. 
\begin{proposition}\label{prop0} 
Let  $\{\nu^K,{K\geq 0}\}$ and $\mu$ be random elements in $\mathbb D\left([0,T],\mathcal M_F(\mathcal X)\right)$. 
		%with Lipschitz norm one (cf. \cite{B_MT} p. 191). Note that $\Vert\: . \:\Vert^{}_0$ is a norm on the vector space of finite, signed measures on $\mathcal X$.
	If, for all $\d>0$,
	\begin{equation}\label{conv_in_proba}
			\lim_{K\rightarrow \infty} \mathbb P\left [\:\sup_{0\leq t\leq T} \Vert \nu_t^{K}-\mu_t\Vert^{}_0>
			\d\:\right] = 0,
	 \end{equation}
		then $\nu^K$ converges in probability, as $K\rightarrow \infty$,  with respect to the Skorokhod topology 
on  $\mathbb D([0,T],\mathcal M(\mathcal X))$ to $\mu$.
\end{proposition}
\begin{proof}Let us equip  $\mathcal M_F(\mathcal X)$  with the topology of weak convergence. Obverse that this topology is metrizable with the Kantorovich-Rubinstein norm, see \cite{B_MT} Vol. II, p. 193. Let $\L$ be the  class of strictly increasing,
		 continuous mapping of $[0,T]$ onto itself. If $\l\in\L$, then $\l(0)=0$ and $\l(T)=T$.
		The Skorokhod topology on $\mathbb D\left([0,T],(\mathcal M_F(\mathcal X),\Vert \:.\:\Vert^{}_0)\right)$ is generated by the distance
	\be
		d(\mu,\nu) =\inf_{\l\in\L}\left\{\max\left\{\sup_{t\in[0,T]}|\l (t)-t|,\sup_{t\in[0,T]} \Vert \mu_t-\nu_{\l t} \Vert^{}_0\right\}\right\},
	\ee
		on $\mathbb D\left([0,T],(\mathcal M_F(\mathcal X),\Vert \:.\:\Vert^{}_0)\right)$, see e.g. \cite{B_CoPM}, Chap. 3.
		Since the identity lies in $\L$ it is clear that  $d(\mu,\nu)\leq \sup_{t\in[0,T]}\Vert \mu_t-\nu_{t} \Vert^{}_0 $. 
		Therefore, if a sequence of random elements with state space $\mathbb D([0,T], \mathcal M_F(\mathcal X))$ equipped with the metric induced  by the norm $\sup_{t\in[0,T]}\Vert \mu_t \Vert^{}_0 $
		 convergences in probability to $\mu$, it also convergences in probability to $\mu$ if  
		 $\mathbb D([0,T], \mathcal M_F(\mathcal X))$ is equipped with the metric $d$. 
\end{proof}

\begin{proposition} \label{prop1}  Fix $\e>0$ and let $\s_K$ a sequence in $K$ with $K^{-\nicefrac 1 2+\alpha}\ll\s_K\ll1$.
Let $Z_n$ be a Markov chain with state space $\mathbb N_0$ and with the following transition probabilities 
\be
\mathbb P[Z_{n+1}=j|Z_n=i]=p(i,j)=
	\begin{cases}
				 1 ,  										& \text{for } i=0		\text{ and } j=1,\\
				\tfrac 1 2-C_1i K^{-1}+C_2\epsilon \sigma_K, 	& \text{for } i\geq 1 	\text{ and } j=i+1,\\
				\tfrac 1 2+C_1i K^{-1}-C_2\epsilon \sigma_K, 	& \text{for } i\geq 1 	\text{ and } j=i-1,
	\end{cases}
\ee
for some constants $C_1>0$ and $C_2\geq 0$. 
Let $\tau_i$ be the first hitting time of level $i$ by $Z$  and let $\mathbb P_a$ denote the law of $Z$ conditioned on $Z_0=a$.
Then, for all $M\geq \tfrac {8C_2}{C_1}$  and for all $a\leq \tfrac1 3 M \epsilon \sigma_K K $ 
\be\label{mod_de}
\lim_{K \to \infty}{ e^{K^{2\alpha}}}\:{\mathbb P_a\left[\tau^{}_{\lceil M \epsilon \sigma_KK \rceil }<\tau^{}_0\right]} = 0.
\ee
\end{proposition}
\begin{remark}
The proposition can be seen as a moderate deviation result for this particular Markov chain. 
More precisely, we can prove that there exist two constants $M>0$ and $C_3>0$ which depend only on $C_1$ and $C_2$ such that for $a<\tfrac1 3 M \epsilon \sigma_K K $ 
\be
\mathbb P_a\left[\tau_{\lceil M \epsilon \sigma_KK \rceil }<\tau_0\right] \leq \exp \lb - C_3K^{-1}\lb(\tfrac 1 3 M \epsilon \sigma_KK)^2 -a^2\rb\rb,
\ee
 for all  $ K $ large enough.
\end{remark}
\begin{proof}
We calculate this probability with some standard potential theory arguments.
Let $h_{\lceil M \epsilon \sigma_K K\rceil ,0} (a)$ be the solution of the Dirichlet problem with $\lambda=0$, i.e. 
\begin{align}\nonumber
\mathscr Lh_{\lceil M \epsilon \sigma_K K\rceil ,0} (x)&=0 \qquad\qquad \text{for } 0<x<\lceil M \epsilon \sigma_K K\rceil \\ \nonumber
 h_{\lceil M \epsilon \sigma_K K\rceil ,0} (x)&=1 \qquad\qquad\text{for } x\geq \lceil M \epsilon \sigma_K K\rceil \\
 h_{\lceil M \epsilon \sigma_K K\rceil ,0} (x)&=0\qquad\qquad \text{for } x=0.
 \end{align}
Therefore, we obtain for $0<a<\lceil M \epsilon \sigma_K K\rceil $ (cf. \cite{B_M} p. 188) 
\be\Eq(needsanumber.1)
\mathbb P_a\bigl[\tau_{\lceil M \epsilon \sigma_K K\rceil }<\tau_0\bigr]
		=h_{\lceil M \epsilon \sigma_K K\rceil ,0} (a)=\frac{\sum_{i=1}^a \frac{1}{\pi(i)}\frac{1}{p(i,i-1)}}{\sum_{i=1}^{\lceil M \epsilon \sigma_K K\rceil } \frac{1}{\pi(i)}\frac{1}{p(i,i-1)}},
\ee 
where $\pi=(\pi(0),\pi(1),\pi(2),\ldots)$ is an invariant measure %(not necessarly a probability measure)
of the one-dimensional Markov chain $Z_n$. In our case any invariant measure $\pi$ has to satisfy, for all  $ i\geq 1$,
\be
\pi(0)=p(1,0)\pi(1)\quad\text{ and }\quad
\pi(i)=p(i-1,i)\pi(i-1)+p(i+1,i)\pi(i+1).
\ee
Therefore, $\pi$ with $\pi(0)=1, \pi(1)=\frac 1 {p(1,0)}$ and $\pi(i)=\prod_{j=1}^{i-1}\frac{p(j,j+1)}{p(j,j-1)} \frac{1}{p(i,i-1)}$ is the unique invariant measure for the  Markov chain $Z_n$.
Thus we get  from \eqv(needsanumber.1)
\begin{eqnarray}\nonumber
h_{\lceil M \epsilon \sigma_K K\rceil ,0} (a)
		&=&\frac{\sum_{i=1}^a \prod_{j=1}^{i-1}\frac{p(j,j-1)}{p(j,j+1)}}{\sum_{i=1}^{\lceil M \epsilon \sigma_K K\rceil } \prod_{j=1}^{i-1}\frac{p(j,j-1)}{p(j,j+1)}}\\
&=&\frac{\sum_{i=1}^a \exp\left(\sum_{j=1}^{i-1}\ln\left( \frac{ 1 +2C_1 K^{-1}j- 2C_2\epsilon \sigma_K}{1 -2 C_1K^{-1}j+2C_2 \epsilon \sigma_K}\right)\right)}
	{\sum_{i=1}^{\lceil M\epsilon \sigma_KK\rceil} \exp\Bigl(\sum_{j=1}^{i-1}\underbrace{\ln\Bigl( \tfrac{1 +2C_1 K^{-1}j- 2C_2\epsilon \sigma_K}{ 1 - 2C_1K^{-1}j+ 2C_2\epsilon \sigma_K}\Bigr)}_{=:f(j)}\Bigr)}.
\end{eqnarray}		
For all $j\leq M\epsilon \sigma_KK$ we can approximate $f(j)$ as follwos
\begin{align} \nonumber
f(j) &=\ln\lb1+\tfrac{4C_1 K^{-1}j- 4C_2\epsilon \sigma_K}{ 1 - 2C_1K^{-1}j+ 2C_2\epsilon \sigma_K} \rb
     	= \tfrac{4C_1 K^{-1}j- 4C_2\epsilon \sigma_K}{ 1 - 2C_1K^{-1}j+ 2C_2\epsilon \sigma_K}-O\lb(\tfrac{4C_1 K^{-1}j- 	
     		4C_2\epsilon \sigma_K}{ 1 - 2C_1K^{-1}j+ 2C_2\epsilon \sigma_K})^2\rb\\\nonumber
      &= 4C_1\tfrac j K- 4C_2\epsilon \sigma_K+O\lb (\tfrac j K )^2+\epsilon\sigma_K\tfrac j K+\epsilon^2 \sigma_K^2\rb\\
       &= 4C_1\tfrac j K- 4C_2\epsilon \sigma_K+O\lb (M\epsilon\sigma_K)^2\rb
\end{align}			
Therefore,
\begin{align}
h_{\lceil M \epsilon \sigma_K K\rceil ,0} (a)&
 \leq \frac{\sum_{i=1}^a \exp(\sum_{j=1}^{i-1}4C_1\tfrac j K+O\lb (M\epsilon\sigma_K)^2\rb}
		{\sum_{i=1}^{\lceil M \epsilon \sigma_K K\rceil } \exp(\sum_{j=1}^{i-1}4C_1\tfrac j K- 4C_2\epsilon \sigma_K -O\lb (M\epsilon\sigma_K)^2\rb}\\\nonumber
&\leq \frac{a \exp(2 C_1a^2 K^{-1}+O\lb a (M\epsilon\sigma_K)^2\rb}
		{\sum_{i=1}^{\lceil M \epsilon \sigma_K K\rceil } \exp\lb2C_1 K^{-1}(i^2-i)- 4C_2\epsilon \sigma_Ki -O\lb (i-1)(M\epsilon\sigma_K)^2\rb\rb}\\\nonumber
&\leq\frac{a \exp(2 C_1a^2 K^{-1}+O\lb a (M\epsilon\sigma_K)^2\rb}
		{\sum_{i=\frac 1 2 \lceil M \epsilon \sigma_K K\rceil }^{\lceil M \epsilon \sigma_K K\rceil } \exp\lb2C_1 K^{-1}i^2-(2C_1K^{-1}+ 4C_2\epsilon \sigma_K) i -O\lb i (M\epsilon\sigma_K)^2\rb\rb}.
		\end{align}
Choosing $M\geq \tfrac {8C_2}{C_1}$, if $a<\tfrac { M\epsilon\sigma_KK}{3}$,  then 
\begin{align}\nonumber
h_{\lceil M \epsilon \sigma_K K\rceil ,0}(a)
		&\leq\frac{a \exp(2 C_1a^2 K^{-1}+O\lb a (M\epsilon\sigma_K)^2\rb}
		{\frac 1 2 \lceil M \epsilon \sigma_K K\rceil  \exp\lb (\frac 1 2 C_1M- 2 C_2)M \epsilon^2 \sigma_K^2 K -O\lb (\e \s M)^3K+ \e\s_K M\rb\rb}\\\nonumber
		&\leq 2 a (\lceil M \epsilon \sigma_K K\rceil )^{-1} \: \exp\lb C_1K^{-1}\lb 2 a^2 -\tfrac 1 4 (\lceil M \epsilon \sigma_K K\rceil )^2 \rb\rb\\
&\leq \exp\lb -C_3 K^{-1}\lb (\tfrac 1 3 \lceil M \epsilon \sigma_K K\rceil )^2-a^2\rb\rb.
\end{align}
Since $K^{-\nicefrac 1 2 +\alpha}\ll \sigma_K$ when $K$ tends to infinity, (\ref{mod_de}) follows.
\end{proof}
%
%%%%%%%%%%%%%%%%%%%%%%%%%%%%%%%%%%%%%%%%%%%%%%%%%%%%%%%%%%%%%%%%%%%%
%
\begin{proposition}
\label{prop2}
Let $(Z_t)_{t\geq 0}$ be a branching process with birth rate per individual $b$ and death rate per individual $d$. 
Let $\tau_i$ be the first hitting time of level $i$ by $Z$ and let $\mathbb P_j$ denote the law of $Z$ conditioned on $Z_0=j$,
 and $\mathbb E_j$ the corresponding expectation. Then
\bea
\label{eq:basic} \mathbb P_j\left[ \tau_k<\tau_0\right]&=&\frac{(d/b)^j-1}{(d/b)^k-1} \qquad\qquad\text{for all }1\leq j\leq k-1,\\
\label{rec}\Bigl|\mathbb P_1[\tau_k<\tau_0]-\frac{[b-d]_+}b\Bigr|& \leq &k^{-1}\qquad\qquad \qquad\quad\text{and}\\
\mathbb E_1[\tau_k\wedge \tau_0]&\leq& \frac {1+\ln(k)}{b},
\eea
where $[b-d]_+\equiv\max\{b\!-\!d\:, 0\}$.
Moreover, if $Z_t$ is slightly super-critical i.e. $b=d+\e$, then
\begin{align} \label{max_E/P}
\max_{n\leq k}\frac{\mathbb E_n[\tau_k\wedge \tau_0]}{\mathbb P_n[\tau_k<\tau_0]}\leq\frac {1+\ln(k)}{\e} 
\end{align}
\end{proposition}
\begin{proof}
Let $p_j\equiv \mathbb P_j[\tau_k<\tau_0]$. Then $p_0=0$, $p_k=1$ and 
$p_j ={\frac{b}{b+d}}\:p_{j+1} +{\frac{d}{b+d}}\:p_{j-1}$ for all $1\leq j\leq k-1$
%\begin{align*}
 %p_j&=\sum_{i=0}^{k} p(j,i)\mathbb P_i[\tau_k<\tau_0]\\
%& =p(j,j+1)\mathbb P_{j+1}[\tau_k<\tau_0]+p(j,j-1)\mathbb P_{j-1}[\tau_k<\tau_0]\\&= \small{\frac{b}{b+d}}\:p_{j+1} +\small{\frac{d}{b+d}}\:p_{j-1}\qquad \text{ for all }j\in\{1,\ldots,k-1\}
%\end{align}
by the Markov property. From this recursion, we obtain the characteristic polynomial
\begin{align} 
P(x)=bx^2-(b+d)x+d.
\end{align}
With its roots $1$ and $ d/b$, we obtain the following general solution for the recursion 
\begin{align} 
p_n=\k_0\cdot 1^n +\k_1\Bigl(\frac d b\Bigr)^n, 
\end{align}
where $\k_0$ and $\k_1$ are constants. 
From the initial condition $p_0=0$ and $p_k=1$, we obtain $\k_0=-((\frac d b)^k-1)^{-1}$ and $\k_1=((\frac d b)^k-1)^{-1}$.
Therefore, 
\begin{align}
p_n=\frac {(\frac d b)^n-1} {(\frac d b)^k-1}\quad \text{ and } \quad
p_1=\frac{\frac d b -1}{(\frac d b)^k-1}=\frac{1}{1+\frac d b+\ldots +(\frac d b)^{k-1}}.
\end{align}
If $d\geq b$, this computation implies that $p_1\equiv\mathbb P_1[\tau_k<\tau_0]\leq 1/k$ and $[b-d]_+=0$.
If $d< b$, 
\begin{align}\nonumber
\mathbb P_1[\tau_k<\tau_0]- \small{\frac {b-d}{b}} 
	&=\frac{\frac d b -1}{(\frac d b)^k-1}-(1-\tfrac {d} b)\frac{(\frac d b)^k-1}{(\frac d b)^k-1}
	=\frac{(\frac {d} b-1)(\frac d b)^k}{(\frac d b)^k-1}
	=\frac{\frac {d} b-1}{1-(\frac b d)^k}\\\nonumber
	&= \frac{\frac d b(1-\frac {b} d)}{1-(\frac b d)^k}
	= \frac{1}{\frac b d(1+\frac b d+\ldots +(\frac b d)^{k-1})}
	= \frac{1}{\frac b d+\ldots +(\frac b d)^{k}}\\
	&\leq \frac 1 k.
 \end{align}
Similarly, if $e_n\equiv \mathbb E_n[\tau_k\wedge \tau_0]$, then  $e_n$ is the solution of the following non-homogenous Dirichlet problem:
\begin{align} \nonumber 
\mathscr L\; e_n&=-1,\qquad \;\text{ for } n\in\{1,.., k-1\}\\ 
e_n&=0,\quad \qquad\text{ for } n\in\mathbb N_0\setminus\{1,.., k-1\},
\end{align}
where $(\mathscr L f) (x)=x \big(b[f(x+1)-f(x)]+d[f(x-1)-f(x)]\big)$ is the generator of the branching process $Z$. 
Therefore, we have to solve the following non-homogeneous recurrence
 \be
  e_{n+2}-\tfrac {b+d} {b}e_{n+1}+\tfrac {d}{b}e_{n}=\tfrac{-1}{b(n+1)}\quad\text{ and } 
\quad e_0=e_k=0
\ee
We solve this by variation of parameters. Thus, we first solve the associated linear homogeneous recurrence relation:
\be
 h_{n+2}-\tfrac {b+d} b h_{n+1}+\tfrac d b h_n=0 
 \ee
As we have seen before $h_n=\k_2\: 1+\k_3(\tfrac db)^j $ for any $\k_2,\k_3\in \mathbb R$ solves the equation. 
Obverse that this functions are  the harmonic functions of $\mathscr L$. Second, we have to find a particular solution. 
Let $(x_{1j}, x_{2j})$ the solution of the system of linear equations
\begin{align}
x_{1j}+ (\tfrac d b)^{j+1} x_{2j}&=0\\
x_{1j}+ (\tfrac d d)^{j+2} x_{2j}&=-\tfrac 1{b(j+1)},
\end{align}
then 
\begin{align}\nonumber
e_n^p&=\sum_{j=0}^{n-1}x_{1j}1^n+\sum_{j=0}^{n-1}x_{2j}\Big(\frac d b\Big)^n
		=\tfrac {-1}{b-d}\sum_{j=1}^{n}\frac 1 j + \tfrac 1{b-d}\sum_{j=1}^{n}\frac 1 j \Big(\frac b d\Big)^j\Big(\frac d b\Big)^n\\
		&=\tfrac 1{b-d}\sum_{j=1}^{n}\frac 1 j\Big(\Big(\frac d b\Big)^{n-j}-1\Big)
\end{align}
is a particular solution.
Now, we obtain we obtain the following general solution for the recurrence:
\be
e_n=h_n+e_n^p
=\k_2+\k_3(\tfrac db)^n+\tfrac 1 {b-d}\sum_{j=1}^{n}\frac 1 j\Big(\Big(\frac d b\Big)^{n-j}-1\Big).
\ee
We have the boundary condition $e_0=e_k=0$, therefore $\k_2$ and $\k_3$
 are given by the solution of the following system of linear equations
\begin{align}
\k_2+ \k_3(\frac d b)^0+\tfrac 1 {b-d}\sum_{j=1}^{0}\frac 1 j\Big(\Big(\frac d b\Big)^{0-j}-1\Big)&=0,\\
\k_2+ \k_3(\frac d b)^k+\tfrac 1 {b-d}\sum_{j=1}^{k}\frac 1 j\Big(\Big(\frac d b\Big)^{k-j}-1\Big)&=0,
\end{align}
 and we obtain that
 \begin{align}\nonumber
e_n&=\tfrac 1{b-d}\sum_{j=1}^{k}\frac 1 j\frac{(\frac d b)^{k-j}-1}{(\frac d b)^k-1}
			-\tfrac 1{b-d}\sum_{j=1}^{k}\frac 1 j\frac{(\frac d b)^{k-j}-1}{(\frac d b)^k-1}\Big(\frac db\Big)^n
			+\tfrac 1 {b-d}\sum_{j=1}^{n}\frac 1 j\Big(\Big(\frac d b\Big)^{n-j}-1\Big)\\
	&=\tfrac 1{b-d}\sum_{j=1}^{k}\frac 1 j\frac{((\frac d b)^{k-j}-1)(1-(\frac d b)^n)}{(\frac d b)^k-1}
					+\tfrac 1 {b-d}\sum_{j=1}^{n}\frac 1 j\Big(\Big(\frac d b\Big)^{n-j}-1\Big).		
\end{align}
With this formula we can easily prove the second inequality of the proposition,
\be
e_1%=\tfrac 1{b-d}\sum_{n=1}^{k}\frac 1 n\frac{(\frac d b)^{k-n}-1}{(\frac d b)^k-1}
		%-\tfrac 1{b-d}\sum_{n=1}^{k}\frac 1 n\frac{(\frac d b)^{k-n}-1}{(\frac d b)^k-1}\Big(\frac db\Big)
		%+\tfrac 1 {b-d}\sum_{n=1}^{1}\frac 1 n\Big(\Big(\frac d b\Big)^{1-n}-1\Big)\\
=\tfrac 1{ b -d}\sum_{n=1}^{k}\frac 1 n\frac{(\frac d b)^{k-n}-1}{(\frac d b)^k-1}(1-\tfrac db)+0
\leq\frac 1 b \sum_{n=1}^{k}\frac 1 n \leq \frac {1+\ln(k)} b.
\ee
Finally, we obtain for slightly super-critical  $Z_t$, i.e.  with $b=d+\e$,
\begin{align}\nonumber
\frac{\mathbb E_n[\tau_k\wedge \tau_0]}{\mathbb P_n[\tau_k<\tau_0]}&=\frac{e_n}{p_n}
=\tfrac 1{b-d}\sum_{j=1}^{k}\frac 1 j\underbrace{\Big(\left(\frac d b\right)^{k-j}-1\Big)(-1)}_{\leq 1}
					+\tfrac 1 {b-d}\sum_{j=1}^{n}\frac 1 j\underbrace{\frac{((\frac d b)^{n-j}-1)(1-(\frac d b)^k)}{1-(\frac d b)^n}}_{\leq 0}\\
&\leq \frac{1}{\e}	\sum_{j=1}^{k}\frac 1 j	\leq \frac{ 1+ \ln(k)}{\e},			
 \end{align}
 which proves (\ref{max_E/P}). 
\end{proof}
%
%%%%%%%%%%%%%%%%%%%%%%%%%%%%%%%%%%%%%%%%%%%%%%%%%%%
%
\begin{proposition}
\label{prop2.1}
Let $(Z^K_t)_{t\geq 0}$ be a sequence branching process with birth rate per individual $b\geq0$ and death rate per individual $d\geq0$ and $|b-d|=O(\s_K)$, where $K^{-1/2+\a}\ll\s_K\ll 1$.  
Let $\tau_i$ be the first hitting time of level $i$ by $Z$ and let $\mathbb P_j$ denote the law of $Z$ conditioned on $Z_0=j$. 
\begin{enumerate}[(a)]
\setlength{\itemsep}{3pt}
\item The invasion probability can be approximated up to a error of order $\exp(-K^{\a})$, i.e. 
\be
\lim_{K\to \infty}\exp(K^{\a})\big |\mathbb P_1\left[\tau_{\lceil \e\s_K K\rceil}<\tau_{0}\right] -\frac{[b-d]_+}{b}\big |=0.
\ee
\item If $b>d$ (super-critical case), we have exponential tails, i.e.
 \begin{align}
\lim_{K\to \infty}\exp(\sigma_K^{-\a/3})\mathbb P_1\left[\tau_{\lceil \e\s_K K\rceil}>\ln(K)\s_K^{-1-\a/2}\Bigr| \tau_{\lceil \e\s_K K\rceil}<\tau_{0}\right]=0
\end{align}
and  \begin{align}
\lim_{K\to \infty}\exp(K^{\a})\mathbb P_{\lceil\e\s_K K\rceil}\left[\tau_{\lceil \e K\rceil}>\tau_{0}\right]=0
\end{align}
\end{enumerate}
\end{proposition}
\begin{proof}(a) Compare with (\ref{eq:basic}) that 
\begin{align}\label{eq2.2.2}
\mathbb P_1\left[\tau_{\lceil \e\s_K K\rceil}<\tau_{0}\right]= \frac{ (d/ b) - 1}{( d /b)^{\lceil \e\s_K K\rceil} - 1}.
\end{align}
If $b>d$ (sub-critical case), there exist two constants $\underline C^{\text{sub}}>0$ and $\bar C^{\text{sub}}>0$ such that $1+\underline C^{\text{sub}}\s_K\leq d/b\leq 1+\bar C^{\text{sub}}\s_K$. Therefore, the left hand site of  (\ref{eq2.2.2}) does not 
exceed
\be
\frac{\bar C^{\text{sub}} \s_K}{(1+\underline C^{\text{sub}}\s_K)^{\lceil \e\s_K K\rceil} - 1}
\leq\frac{\bar C^{\text{sub}} \s_K}
{\exp(\underline C^{\text{sub}}\s_K\lceil \e\s_K K\rceil-O(\s_K^3\e K))-1}
=o(\exp(-K^{\a})).
\ee
The last equality holds, since  $K^{2\a}\ll\s_K^2 K$.
If $b>d$ (super-critical case),  we obtain similarly 
\be
\left |\mathbb P_1[\tau_k<\tau_0]- \small{\frac {b-d}{b}} \right|
=\left |\frac{\frac {d} b-1}{1-(\frac b d)^k}\right|=o(\exp(-K^{\a})).
\ee
	\\[0.5em]
(b) Compare with \cite{A_MC} page 41, that
  \bea
 &&\mathbb P_1\left[\tau_{\lceil \e\s_K K\rceil}>\ln(K)\s_K^{-1-\a/2}\Bigr| \tau_{\lceil \e\s_K K\rceil}<\tau_{0}\right]\\\nonumber
 &&\leq \exp\left(-\left\lfloor\frac{ \ln(K)\s_K^{-1-\a/2} }{e \max_{n\leq \lceil \e\s_K K\rceil} 
 \mathbb E_n\left[\tau_{\lceil \e\s_K K\rceil}\big| \tau_{\lceil \e\s_K K\rceil}<\tau_{0}\right]} \right\rfloor\right)\leq \exp\left(-\s_K^{-\a/3} \right), 
  \eea
where the last inequality holds, because we can apply Proposition \ref{prop2} 
\bea
\max_{n\leq \lceil \e\s_K K\rceil} 
 \mathbb E_n\left[\tau_{\lceil \e\s_K K\rceil}\big| \tau_{\lceil \e\s_K K\rceil}<\tau_{0}\right]
& =&\max_{n\leq\lceil \e\s_K K\rceil }\frac{ \mathbb E_n\big[\tau_{\lceil \e\s_K K\rceil}\wedge \tau_0\mathds 1_{\tau_0>\tau_{\lceil \e\s_K K\rceil}}\big]}
{\mathbb P_n\big[\tau_0>\tau_{\lceil \e\s_K K\rceil}\big]}\quad
\\\nonumber&\leq&O(\ln(K)\s_K^{-1}).
 \eea 
 On the other hand, we have  
 \be
 \mathbb P_{\lceil\e\s_K K\rceil}\left[\tau_{\lceil \e K\rceil}>\tau_{0}\right]=1- \frac{(d/b)^{\lceil\e\s_K K\rceil} -1}{ (d/b)^{\lceil\e K \rceil}- 1}
 \leq \exp(-K^{2\a})
 \ee
 since $d/b=1-O(\s_K)$ and $ K^{2\a}\ll \s_K\e K$.
\end{proof}
\begin{proposition}
\label{prop2.2}
Let $(Z^K_n)_{n\geq 0}$ a sequence of discrete time Markov Chain with state space $\mathbb Z$ and with transition probabilities 
\be
\mathbb P[Z^K_{n+1}=j|Z^K_n=i]=p(i,j)=
	\begin{cases}
				\tfrac 1 2+C \sigma_K, 	& \text{if } j=i+1,\\
				\tfrac 1 2-C \sigma_K, 	& \text{if }  j=i-1,\\
				0, 	& \hbox{\rm else} ,
	\end{cases}
\ee 
for some constant $C\neq 0$. 
Let $\tau_i$ be the first hitting time of level $i$ by $Z^K$ and let $\mathbb P_j$ denote the law of $Z^K$ conditioned on $Z^K_0=j$ and let $\s_K$ a zero sequence such that $K^{-\frac 1 2+\a}\ll \s_K\ll 1$. 
\begin{enumerate}[(a)]
\setlength{\itemsep}{3pt}
\item If $Z^K$ is slightly supercritical, i.e. $C>0$, then, for all $i\geq 1$
 \be
\lim_{K\to \infty}\exp(K^{\a})\:\mathbb P_{i\lceil(\e/2)\s_K K \rceil}\left[\tau_{(i-1)\lceil (\e/2)\s_K K \rceil}<\tau_{(i+1)\lceil (\e/2)\s_K K \rceil}\right]=0.
\ee
\item  If $Z^K$ is slightly subcritical, i.e. $C<0$, then, for all constants $C_1,C_2,C_3>0$
 \be\label{eq:very last}
\lim_{K\to \infty}\exp(K^{\a})\:\mathbb P_{(C_1+C_2)\lceil\e\s_K K \rceil}\left[\tau_{(C_1+C_2+C_3)\lceil \e\s_K K \rceil}<\tau_{C_1\lceil \e\s_K K \rceil}\right]=0.
\ee
\end{enumerate}
\end{proposition}
\begin{proof}  Since the transition probabilities of $Z^K$ do not depend on the state of $Z^K$, we have that
\begin{align}\label{eq2.2.1}
\mathbb P_{i\lceil (\e/2)\s_K K \rceil}\left[\tau_{(i-1)\lceil  (\e/2) \s_K K \rceil}>\tau_{(i+1)\lceil (\e/2)\s_K K \rceil}\right]=
\mathbb P_{\lceil (\e/2)\s_K K \rceil}\left[\tau_{0}>\tau_{2\lceil (\e/2)\s_K K \rceil}\right]
\end{align} By (\ref{eq:basic}) the left site of (\ref{eq2.2.1}) is equal
\be
\frac{1-(1-2C\s_K+O(\s_K^2))^{\lceil(\e/2)\s_K K \rceil}}{1-(1-2C\s_K+O(\s_K^2))^{2\lceil (\e/2)\s_K K \rceil}}
\geq 1-\exp(- K^2\a),
\ee since $\s_K^2K\gg K^{2\a}$. With the same arguments, we obtain also (\ref{eq:very last}).
\end{proof}
\bibliographystyle{abbrv}

\begin{thebibliography}{10}

\bibitem{A_MC}
D.~Aldous and J.~Fill.
\newblock {\em Reversible Markov chains and random walks on graphs}.
\newblock In progress.
\newblock Manuscript available at
  \href{https://www.stat.berkeley.edu/~aldous/RWG/book.pdf}{https://www.stat.berkeley.edu/~aldous/RWG/book.pdf}.

\bibitem{A_BP}
K.~B. Athreya and P.~E. Ney.
\newblock {\em Branching processes}.
\newblock Die Grundlehren der mathematischen Wissenschaften, Band 196.
  Springer-Verlag, New York-Heidelberg, 1972.

\bibitem{B_CoPM}
P.~Billingsley.
\newblock {\em Convergence of probability measures}.
\newblock Wiley Series in Probability and Statistics: Probability and
  Statistics. John Wiley \& Sons, Inc., New York, second edition, 1999.

\bibitem{B_MT}
V.~I. Bogachev.
\newblock {\em Measure theory. {V}ol. {I}, {II}}.
\newblock Springer-Verlag, Berlin, 2007.

\bibitem{B_M}
A.~Bovier.
\newblock Metastability.
\newblock In {\em Methods of contemporary mathematical statistical physics},
  volume 1970 of {\em Lecture Notes in Math.}, pages 177--221. Springer,
  Berlin, 2009.

\bibitem{C_TSS}
N.~Champagnat.
\newblock A microscopic interpretation for adaptive dynamics trait substitution
  sequence models.
\newblock {\em Stochastic Process. Appl.}, 116(8):1127--1160, 2006.

\bibitem{C_CEAD}
N.~Champagnat, R.~Ferri{\`e}re, and G.~Ben~Arous.
\newblock The canonical equation of adaptive dynamics: a mathematical view.
\newblock {\em Selection}, 2:73--83, 2001.

\bibitem{C_ME}
N.~Champagnat, R.~Ferri{\`e}re, and S.~M{\'e}l{\'e}ard.
\newblock From individual stochastic processes to macroscopic models in
  adaptive evolution.
\newblock {\em Stoch. Models}, 24(suppl. 1):2--44, 2008.

\bibitem{C_PES}
N.~Champagnat and S.~M{\'e}l{\'e}ard.
\newblock Polymorphic evolution sequence and evolutionary branching.
\newblock {\em Probab. Theory Related Fields}, 151(1-2):45--94, 2011.

\bibitem{D_DTC}
U.~Dieckmann and R.~Law.
\newblock The dynamical theory of coevolution: a derivation from stochastic
  ecological processes.
\newblock {\em J. Math. Biol.}, 34(5-6):579--612, 1996.

\bibitem{D_ADPS}
M.~Durinx, J.~A.~J. Metz, and G.~Mesz{\'e}na.
\newblock Adaptive dynamics for physiologically structured population models.
\newblock {\em J. Math. Biol.}, 56(5):673--742, 2008.

\bibitem{E_MP}
S.~N. Ethier and T.~G. Kurtz.
\newblock {\em Markov processes. Characterization and convergence}.
\newblock Wiley Series in Probability and Mathematical Statistics. John Wiley
  \& Sons Inc., New York, 1986.

\bibitem{F_MA}
N.~Fournier and S.~M{\'e}l{\'e}ard.
\newblock A microscopic probabilistic description of a locally regulated
  population and macroscopic approximations.
\newblock {\em Ann. Appl. Probab.}, 14(4):1880--1919, 2004.

\bibitem{F_RPoDS}
M.~I. Freidlin and A.~D. Wentzell.
\newblock {\em Random perturbations of dynamical systems}, volume 260 of {\em
  Grundlehren der Mathematischen Wissenschaften [Fundamental Principles of
  Mathematical Sciences]}.
\newblock Springer, Heidelberg, third edition, 2012.

\bibitem{G_CSS}
S.~A.~H. Geritz.
\newblock Resident-invader dynamics and the coexistence of similar strategies.
\newblock {\em J. Math. Biol.}, 50(1):67--82, 2005.

\bibitem{H_AD}
J.~Hofbauer and K.~Sigmund.
\newblock Adaptive dynamics and evolutionary stability.
\newblock {\em Appl. Math. Lett.}, 3(4):75--79, 1990.

\bibitem{I_MMLS}
J.~Istas.
\newblock {\em Mathematical modeling for the life sciences}.
\newblock Universitext. Springer-Verlag, Berlin, 2005.

\bibitem{J_BPCE}
P.~Jagers and A.~N. Lageras.
\newblock General branching processes conditioned on extinction are still
  branching processes.
\newblock {\em Electron. Commun. Probab.}, 13:540--547, 2008.

\bibitem{M_F}
J.~Metz.
\newblock Fitness.
\newblock {\em Encyclopedia of Ecology}, 2:1599--1612, 2008.

\bibitem{M_HSDF}
J.~Metz, R.~Nisbet, and S.~Geritz.
\newblock How should we define `fitness' for general ecological scenarios?
\newblock {\em Trends in Ecology and Evolution}, 7(6):198 -- 202, 1992.

\bibitem{M_AD}
J.~A.~J. Metz, S.~A.~H. Geritz, G.~Mesz{\'e}na, F.~J.~A. Jacobs, and J.~S. van
  Heerwaarden.
\newblock Adaptive dynamics, a geometrical study of the consequences of nearly
  faithful reproduction.
\newblock In {\em Stochastic and spatial structures of dynamical systems
  ({A}msterdam, 1995)}, Konink. Nederl. Akad. Wetensch. Verh. Afd. Natuurk.
  Eerste Reeks, 45, pages 183--231. North-Holland, Amsterdam, 1996.

\bibitem{P_CSP}
D.~Pollard.
\newblock {\em Convergence of stochastic processe}.
\newblock Springer Series in Statistics. Springer-Verlag, New York, 1984.

\end{thebibliography}

% AOS,AOAS: If there are supplements please fill:
%\begin{supplement}[id=suppA]
%  \sname{Supplement A}
%  \stitle{Title}
%  \slink[doi]{10.1214/00-AOASXXXXSUPP}
%  \sdatatype{.pdf}" 
%  \sdescription{Some text}
%\end{supplement}

\end{document}